\documentclass[a4paper]{amsart}
\author[Raghavan]{Dilip Raghavan}
\address[Raghavan]{Department of Mathematics\\
 National University of Singapore\\
 Singapore 119076.}
\email{\href{dilip.raghavan@protonmail.com}{dilip.raghavan@protonmail.com}}
\urladdr{\url{https://dilip-raghavan.github.io/}}
\author[Xiao]{Ming Xiao}
\thanks{Both authors were partially supported by Singapore Ministry of Education's research grant number MOE2017-T2-2-125}
\address[Xiao]{School of Mathematical Sciences\\ Nankai University\\ No 94.\@ Weijin Road\\ Tianjin 300071, China.}
\email{\href{ming.xiao@nankai.edu.cn}{ming.xiao@nankai.edu.cn}}
\urladdr{\url{https://logic.nankai.edu.cn/2022/0829/c29601a469780/page.htm}}
\date{\today}
\title{Borel Order Dimension}
\usepackage[only, llbracket, rrbracket]{stmaryrd}
\usepackage{amssymb, amsmath, amsthm, mathrsfs, enumitem, amsfonts, latexsym, bbm, cancel}
\usepackage{hyperref}
\newtheorem{Theorem}{Theorem}[section]

\newtheorem{Lemma}[Theorem]{Lemma}
\newtheorem{Corollary}[Theorem]{Corollary}

\newtheorem{Question}[Theorem]{Question}

\theoremstyle{definition}
\newtheorem{Definition}[Theorem]{Definition}

\newtheorem{nota}[Theorem]{Notation}
\theoremstyle{remark}

\newcommand{\forces}{\Vdash}
\newcommand{\restrict}{\mathord\upharpoonright}

\renewcommand{\[}{\left[}
\renewcommand{\]}{\right]}
\newcommand{\PP}{\mathbb{P}}
\newcommand{\PPP}{\mathcal{P}}

\newcommand{\QQ}{\mathbb{Q}}
\newcommand{\lc}{\left|}
\newcommand{\rc}{\right|}

\newcommand\ZFC{\mathrm{ZFC}}

\newcommand\PFA{\mathrm{PFA}}

\newcommand\Fn{\mathrm{Fn}}
  
\newcommand{\BS}{{\omega}^{\omega}}

\DeclareMathOperator{\proj}{Proj}

\DeclareMathOperator{\otp}{otp}
\DeclareMathOperator{\cov}{cov}

\DeclareMathOperator{\dom}{dom}
\DeclareMathOperator{\ran}{ran}
\DeclareMathOperator{\succc}{succ}

\DeclareMathOperator{\cf}{cf}

\DeclareMathOperator{\odim}{\mathtt{odim}}
\DeclareMathOperator{\homo}{\mathtt{Hom}}
\DeclareMathOperator{\di}{\mathtt{diam}}
\newcommand{\dicr}{\vec{\chi}}
\newcommand{\wbdicr}{{\vec{\chi}}^{\mathit{w}}_{\mathtt{B}}\mkern-2mu}
\newcommand{\bdicr}{{\vec{\chi}}^{}_{\mathtt{B}}\mkern-1mu}
\newcommand{\wbcr}{{\chi}^{\mathit{w}}_{\mathtt{B}}\mkern-2mu}
\newcommand{\bcr}{{\chi}^{}_{\mathtt{B}}\mkern-1mu}
\newcommand{\llex}{{\mathord\leq}_{\mathtt{lex}}}
\newcommand{\Pset}{\mathcal{P}}
\newcommand{\MMM}{\mathcal{M}}
\newcommand{\mm}{\mathbbm{m}}
\newcommand{\nn}{\mathbbm{n}}

\newcommand{\BBB}{\mathcal{B}}
\newcommand{\SSS}{\mathcal{S}}

\newcommand{\AAA}{{\mathcal{A}}}

\newcommand{\DDD}{\mathcal{D}}
\newcommand{\G}{{\mathscr{G}}}
\newcommand{\GGG}{{\mathcal{G}}}

\newcommand{\HHH}{{\mathcal{H}}}
\newcommand{\sH}{{\mathscr{H}}}

\newcommand{\FFF}{{\mathcal{F}}}

\newcommand{\F}{{\mathscr{F}}}

\newcommand{\V}{{\mathbf{V}}}

\newcommand{\VG}{{{\mathbf{V}}[G]}}

\newcommand{\RR}{\mathbb{R}}
\newcommand{\fSS}{\mathbb{S}}
\newcommand{\RRR}{\mathcal{R}}
\newcommand{\KKK}{\mathcal{K}}

\newcommand{\pr}[2]{\left\langle #1, #2 \right\rangle}
\newcommand{\seq}[4]{\left\langle {#1}_{#2}: #2 #3 #4 \right\rangle}

\newcommand{\pc}[2]{{\[#1\]}^{#2}}

\newcommand{\lv}[2]{{\mathtt{Lev}}_{#2}{\mathord{\left(#1\right)}}}
\newcommand{\cat}[2]{{#1}^{\frown}{#2}}
\newcommand{\next}[2]{{#1}^{\frown}{\langle #2 \rangle}}
\newcommand{\ins}[3]{{#1}^{\frown}{{\langle #2 \rangle}^{\frown}{#3}}}
\newcommand{\lex}[2]{#1 \; {\leq}_{\mathtt{lex}} \; #2}
\begin{document}
\begin{abstract}
 We introduce and study a notion of Borel order dimension for Borel quasi orders.
 It will be shown that this notion is closely related to the notion of Borel dichromatic number for simple directed graphs.
 We prove a dichotomy, which generalizes the ${\GGG}_{0}$-dichotomy, for the Borel dichromatic number of Borel simple directed graphs.
 By applying this dichotomy to Borel quasi orders, another dichotomy that characterizes the Borel quasi orders of uncountable Borel dimension is proved.
 We obtain further structural information about the Borel quasi orders of countable Borel dimension by showing that they are all Borel linearizable.
 We then investigate the locally countable Borel quasi orders in more detail, paying special attention to the Turing degrees, and produce models of set theory where the continuum is arbitrarily large and all locally countable Borel quasi orders are of Borel dimension less than the continuum.
 Combining our results here with earlier work shows that the Borel order dimension of the Turing degrees is usually strictly larger than its classical order dimension.
\end{abstract}
\maketitle
\section{Introduction} \label{sec:intro}
Order dimension is a measure of the complexity of a partial order.
It was first considered by Dushnik and Miller in \cite{DM} and it is now well-studied in combinatorics (e.g.\@ \cite{MR0025541}, \cite{MR1169299}, \cite{MR0300944}, \cite{MR0070681}) and in computer science (e.g.\@ \cite{MR0666860}, \cite{MR3601672}).
In this paper, we will consider a Borel analog of this notion for Borel quasi orders and their quotients.
Our work falls into the general framework of Borel combinatorics, where one studies Borel analogs of classical combinatorial notions (see e.g.\@ \cite{borelorder}, \cite{MR1057041}, \cite{MR1667145}).
Like other works in this genre, one of our main aims is to prove a dichotomy theorem characterizing those Borel quasi orders whose Borel dimension is uncountable.
In such dichotomy theorems, one identifies a collection of canonical basic objects together with appropriate morphisms, and shows that any object with a sufficiently complicated structure admits a morphism from one of the basic objects.
A major example is the theorem of Harrington, Kechris, and Louveau~\cite{MR1057041} characterizing the non-smooth Borel equivalence relations in terms of the existence of a continuous reduction from the basic object ${E}_{0}$.
Another major example is the ${\GGG}_{0}$-dichotomy of Kechris, Solecki, and Todorcevic~\cite{MR1667145} characterizing the analytic graphs of uncountable Borel chromatic number in terms of the existence of a continuous graph homomorphism from the basic object ${\GGG}_{0}$.

Our first task here will be to provide a robust definition of the Borel order dimension of a Borel quasi order.
This concept requires careful definition because, unlike the classical case, most interesting Borel quasi orders do not admit any Borel linear extensions.
We will first show that an appropriate definition of this notion connects it to the Borel analog of the concept of dichromatic number of a digraph.
This concept was first defined by Neumann-Lara~\cite{MR0693366} in the late 1970s and has since become important in finite combinatorics, where it is often seen as the appropriate generalization of the concept of chromatic number to digraphs.
With the relationship between Borel order dimension and Borel dichromatic number established in Section \ref{sec:prelim}, we are naturally led to seek a dichotomy theorem that characterizes Borel and analytic digraphs of uncountable Borel dichromatic number.
Two such dichotomies are established in Section~\ref{sec:dichotomy}.
We consider two types of morphisms: continuous digraph homomorphisms as well as those that preserve certain non-edges, namely continuous digraph homomorphisms that map every induced copy of the directed cycle of length $n$ to another such induced copy.
The latter are known as \emph{minimal homomorphisms}.
The basic objects of our dichotomies are digraphs generalizing the graph ${\GGG}_{0}$ of \cite{MR1667145}.
But instead of one basic object, we require a family that is uniformly parametrized by a real, so there are continuum many basic objects.
Our dichotomies characterize the Borel digraphs of uncountable Borel dichromatic number in terms of the existence of a continuous minimal homomorphism from one of the basic objects.
In the case of analytic digraphs of uncountable Borel dichromatic number, we can characterize them by a continuous homomorphism from one of the basic objects, though we are unable to guarantee minimality.
Further, it is proved in Section \ref{sec:dichotomy} that the ${\GGG}_{0}$-dichotomy of \cite{MR1667145} is a special case of these dichotomies.

Having obtained a dichotomy for the Borel dichromatic number, we associate a Borel quasi order with each of its basic objects to obtain another parametrized family of continuum many basic Borel quasi orders for the third dichotomy of Section \ref{sec:dichotomy}.
Each basic object of the third dichotomy is a locally countable Borel quasi order of height $2$.
We are able to characterize the Borel quasi orders of uncountable Borel order dimension in terms of the existence of a continuous map from one of these basic objects that is order preserving and also preserves the incomparability of certain elements.
Additionally, it is shown in Section \ref{sec:dichotomy} that every Borel quasi order whose Borel order dimension is countable admits a Borel linear extension.

The classical order dimension of locally countable orders, especially the Turing degrees, was investigated in \cite{posetdim} and \cite{MR4228343}.
Obtaining Borel analogs of some of the results from those works was one of the original motivations for this paper.
Using the results in Section \ref{sec:dichotomy} and some ideas from the papers \cite{posetdim, MR4228343, MR4518086}, we obtain several results about the Borel order dimension of locally finite and locally countable Borel quasi orders in Section \ref{sec:consistency}.
It was proved in \cite{posetdim} that the classical order dimension of the Turing degrees is usually strictly smaller than the continuum.
Combining this result with the dichotomy theorems of Section \ref{sec:dichotomy} shows that the classical order dimension of the Turing degrees is usually strictly smaller than its Borel order dimension.
In \cite{MR4228343}, the classical order dimension of the Turing degrees was characterized in terms of families of partial orders on the reals that separate countable sets of reals from points.
By using a Borel analog of this notion, we obtain upper bounds on the Borel order dimension of locally finite and locally countable Borel quasi orders, in particular, it is shown that all locally finite Borel quasi orders are of countable Borel order dimension.
Finally, by combining a forcing notion from \cite{MR4518086} with Harrington's generic ${G}_{\delta}$ forcing, we show that the Borel order dimension of all locally countable Borel quasi orders can consistently be strictly smaller than the continuum.
It is worth noting that non-trivial interactions between the definable and non-definable aspects of a concept are also common in the study of Tukey reducibility, which is another notion of complexity for ordered sets.
Important examples where the definable aspect of the theory informs the non-definable case and vice versa can be found in \cite{basic}, \cite{tukey}, \cite{suslincombdic}, \cite{nextbestpaper} among many others.
A similar confluence of ideas can also be seen in the definable and non-definable cases of the study of chain conditions in forcing notions, e.g.\@ \cite{MR4075403} and \cite{MR973109}.
\section{Preliminaries} \label{sec:prelim}
Different authors use the terms ``partial order'' and ``quasi order'' differently.
We fix our terminology once and for all with the following definition.
\begin{Definition} \label{def:orders}
 We define the following:
 \begin{enumerate}
  \item
  $\leq$ is a \emph{quasi order on $P$} if $\leq$ is a reflexive and transitive relation on $P$.
  $\PPP = \pr{P}{\leq}$ is said to be a \emph{quasi order} if $\leq$ is a quasi order on $P$.
  \item
  $<$ is a \emph{partial order on $P$} if $<$ is an irreflexive and transitive relation on $P$.
  $\PPP = \pr{P}{<}$ is a \emph{partial order} if $<$ is a partial order on $P$.
  \item
  A quasi order $\leq$ on $P$ is \emph{linear} or \emph{total} if for any $x, y \in P$, $(x \leq y \vee y \leq x)$.
  $\PPP = \pr{P}{\leq}$ is a \emph{total quasi order} or a \emph{linear quasi order} if $\leq$ is a total quasi order on $P$.
  \item
  A partial order $<$ on $P$ is \emph{linear} or $\emph{total}$ if for any $x, y \in P$, $(x < y \vee y < x \vee x = y)$.
  $\PPP = \pr{P}{<}$ is called a \emph{linear order} or a \emph{total order} if $<$ is a linear order on $P$.
  \item
  For a quasi order $\leq$ on $P$, \emph{${E}_{\leq}$} is the equivalence relation on $P$ defined by
  \begin{align*}
   p \; {E}_{\leq} \; q \iff (p \leq q \wedge q \leq p).
  \end{align*}
  \item
  For a quasi order $\leq$, \emph{$x < y$} means $(x \leq y \wedge y \nleq x)$.
  It is easily checked that $<$ is a partial order.
  For a partial order $<$, \emph{$x \leqq y$} means $(x < y \vee x = y)$.
  Then $\leqq$ is a quasi order with ${E}_{\leqq}$ being equality.
  \item
  For a quasi order $\leq$ on $P$, the relation $<$ induces a partial order on $P \slash {E}_{\leq}$, which will also be denoted \emph{$<$}.
  Thus for any $x, y \in P$,
  \begin{align*}
   {\[x\]}_{{E}_{\leq}} < {\[y\]}_{{E}_{\leq}} \iff x < y.
  \end{align*}
  When there is no risk of ambiguity, we will omit the subscript from ${\[x\]}_{{E}_{\leq}}$.
 \end{enumerate}
\end{Definition}
The main focus of this paper is on quasi orders and partial orders that are nicely definable on some Polish space.
The final section of the paper will study locally finite and locally countable examples of such orders, like the Turing degrees.
\begin{Definition} \label{def:borelorders}
 A quasi order $\PPP = \pr{P}{\leq}$ is called a \emph{Borel (analytic) quasi order} if $P$ is a Polish space and $\leq$ is a Borel (analytic) subset of $P \times P$.
\end{Definition}
\begin{Definition} \label{def:locally}
 A quasi order $\PPP = \pr{P}{\leq}$ is said to be \emph{locally countable} (\emph{locally finite}) if for every $x \in P$, $\{ y \in P: y \leq x \}$ is countable (finite).
\end{Definition}
The following notion was first considered by Dushnik and Miller in 1941.
Roughly speaking, it asks how many linear orders are necessary to resolve each incomparability in a partial order in both directions.
\begin{Definition}[Dushnik--Miller~\cite{DM}, 1941]
 Let $\PPP = \pr{P}{<}$ be a partial order.
 The \emph{order dimension} (or simply
 \emph{dimension}) \emph{of~$\PPP$}, denoted \emph{$\odim(\PPP)$}, is the smallest cardinality of a collection of linear orders on $P$ whose intersection is $<$.
\end{Definition}
The use of the word ``dimension'' is justified by the following theorem proved by Dushnik and Miller in the same paper where they introduced their concept.
\begin{Theorem}[Dushnik--Miller~\cite{DM}, 1941]
 For any partial order $\PPP$, $\odim(\PPP)$ is the minimal cardinal $\kappa$ such that $\PPP$ embeds into a product of $\kappa$ many linear orders with the coordinatewise ordering on the product.
\end{Theorem}
\begin{Definition} \label{def:extendpartial}
 Suppose ${<}_{0}$ and $<$ are both partial orders on $P$.
 We say \emph{$<$ extends ${<}_{0}$} if ${<}_{0} \; \subseteq \; <$.
 In other words, $<$ extends ${<}_{0}$ if and only if for any $x, y \in P$, $x \; {<}_{0} \; y \implies x < y$.
\end{Definition}
The following equivalent characterization of the order dimension of a partial order can be easily proved by appealing to the fact that Zorn's lemma allows every partial order to be extended to a linear order.
\begin{Lemma} \label{lem:dimpartial}
 For any partial order $\PPP$, $\odim(\PPP)$ is the minimal cardinal $\kappa$ such that there exists a sequence $\seq{<}{i}{\in}{\kappa}$ of partial orders on $P$ extending $<$ such that
 \begin{align} \label{cond:partial}
  \forall x, y \in P \exists i \in \kappa\[y \; {<}_{i} \;x \vee x < y \vee x = y\].
 \end{align}
\end{Lemma}
As mentioned in the introduction, there has been much interest in the order dimension of finite partial orders coming from finite combinatorics and computer science.
There has also been growing interest in the order dimension of infinite partial orders.
The order dimension of $\pr{\pc{\kappa}{< {\aleph}_{0}}}{\subsetneq}$ was studied in \cite{MR1420396}.
More recently, the order dimension of locally countable orders, especially of the Turing degrees, was investigated in \cite{posetdim, MR4228343}.
Although Dushnik and Miller defined their notion only for partial orders, it is natural to extend their definition to quasi orders as follows.
\begin{Definition} \label{def:dimquasi}
 Let $\PPP = \pr{P}{\leq}$ be a quasi order.
 We define the \emph{order dimension} (or simply \emph{dimension}) \emph{of $\PPP$}, denoted \emph{$\odim(\PPP)$}, to be $\odim\left( \pr{P \slash {E}_{\leq}}{<} \right)$.
\end{Definition}
\begin{Definition} \label{def:extendquasi}
 Suppose ${\leq}_{0}$ and $\leq$ are both quasi orders on $P$.
 $\leq$ is said to \emph{extend} ${\leq}_{0}$ if
 \begin{enumerate}
  \item
  $x \; {\leq}_{0} \; y \implies x \leq y$ and
  \item
  $x \; {E}_{{\leq}_{0}} \; y \iff x \; {E}_{\leq} \; y$,
 \end{enumerate}
 for all $x, y \in P$.
 
 It is easily verified that $\leq$ extends ${\leq}_{0}$ iff
 \begin{enumerate}
  \item[(a)]
  $P \slash {E}_{{\leq}_{0}} = P \slash {E}_{\leq}$ and
  \item[(b)]
  $\[x\] \; {<}_{0} \; \[y\] \implies \[x\] \; < \; \[y\]$, for all $x, y \in P$.
 \end{enumerate}
 If $\leq$ is a linear quasi order which extends ${\leq}_{0}$, then we say $\leq$ \emph{linearizes} ${\leq}_{0}$.
\end{Definition}
When $\PPP = \pr{P}{{\leq}_{0}}$ is a Borel quasi order, the quotient $P \slash {E}_{{\leq}_{0}}$ is almost never a nicely definable object.
Furthermore, by the work of Harrington, Marker, and Shelah~\cite{borelorder} and of Kanovei~\cite{MR1607451}, there is no Borel $\mathord\leq$ that linearizes ${\leq}_{0}$ in most cases.
Hence it will be useful to have a characterization of $\odim(\PPP)$ only in terms of quasi orders on $P$ that extend ${\leq}_{0}$ in the sense of Definition \ref{def:extendquasi}.
The next lemma provides such a characterization.
Although its proof is simple, we include it for completeness.
\begin{Lemma} \label{lem:quasidimchar1}
 Let $\PPP = \pr{P}{\leq}$ be a quasi order.
 Then $\odim(\PPP)$ is the minimal cardinal $\kappa$ such that there exists a sequence $\seq{\leq}{i}{\in}{\kappa}$ of quasi orders on $P$ extending $\leq$ such that
 \begin{align} \label{cond:quasi}
  \forall x, y \in P \exists i \in \kappa\[y \; {\leq}_{i} \; x \vee x \leq y\].
 \end{align}
\end{Lemma}
\begin{proof}
 Let $\lambda = \odim(\PPP)$.
 By Lemma \ref{lem:dimpartial}, fix a sequence $\seq{<}{i}{\in}{\lambda}$ of partial orders on $P \slash {E}_{\leq}$ extending $<$ and satisfying (\ref{cond:partial}) of Lemma \ref{lem:dimpartial} for $P \slash {E}_{\leq}$ and $<$.
 Define ${\leq}_{i}$ on $P$ by
 \begin{align*}
  x \; {\leq}_{i} \; y \iff \left(x \leq y \vee {\[x\]}_{{E}_{\leq}} \; {<}_{i} \; {\[y\]}_{{E}_{\leq}} \right),
 \end{align*}
 for all $x, y \in P$.
 Then ${\leq}_{i}$ is a quasi order on $P$ which extends $\leq$ and the partial order induced by ${\leq}_{i}$ on $P \slash {E}_{\leq} = P \slash {E}_{{\leq}_{i}}$ is ${<}_{i}$.
 Now given any $x, y \in P$, there exists $i \in \lambda$ such that $\left( {\[y\]}_{{E}_{\leq}} \; {<}_{i} \; {\[x\]}_{{E}_{\leq}} \vee {\[x\]}_{{E}_{\leq}} < {\[y\]}_{{E}_{\leq}} \vee {\[x\]}_{{E}_{\leq}} = {\[y\]}_{{E}_{\leq}} \right)$.
 By the definitions, this is easily seen to imply $(y \; {\leq}_{i} \; x \vee x \leq y)$.
 Thus $\seq{\leq}{i}{\in}{\lambda}$ satisfies (\ref{cond:quasi}) of this lemma.
 
 On the other hand suppose that $\seq{\leq}{i}{\in}{\kappa}$ is a sequence of quasi orders on $P$ extending $\leq$ and satisfying (\ref{cond:quasi}) of this lemma.
 Let ${<}_{i}$ be the partial order induced by ${\leq}_{i}$ on $P \slash {E}_{\leq} = P \slash {E}_{{\leq}_{i}}$.
 Then ${<}_{i}$ extends $<$.
 Further, suppose
 \begin{align*}
  {\[x\]}_{{E}_{\leq}}, {\[y\]}_{{E}_{\leq}} \in P \slash {E}_{\leq}.
 \end{align*}
 Let $i \in \kappa$ be so that $(y \; {\leq}_{i} \; x \vee x \leq y)$.
 If $(x \leq y \wedge y \leq x)$, then ${\[x\]}_{{E}_{\leq}} = {\[y\]}_{{E}_{\leq}}$.
 If $(x \leq y \wedge y \nleq x)$, then ${\[x\]}_{{E}_{\leq}} < {\[y\]}_{{E}_{\leq}}$.
 If $(x \nleq y \wedge y \; {\leq}_{i} \; x)$, then ${\[y\]}_{{E}_{\leq}} \; {<}_{i} \; {\[x\]}_{{E}_{\leq}}$.
 Thus (\ref{cond:partial}) of Lemma \ref{lem:dimpartial} is satisfied by $\seq{<}{i}{\in}{\kappa}$ for $P \slash {E}_{\leq}$ and $<$.
\end{proof}
We next provide a characterization of the order dimension of a quasi order in terms of the dichromatic number of a certain associated digraph.
This apparent detour will ultimately payoff by guiding us to the correct definition of Borel order dimension and allowing us to prove a dichotomy theorem that characterizes the Borel quasi orders of uncountable Borel order dimension.
\begin{Definition} \label{def:digraph}
 $\GGG = \pr{X}{R}$ is called a \emph{simple directed graph} or a \emph{digraph} for short if $R \subseteq X \times X$ and $\pr{x}{x} \notin R$, for every $x \in X$.
\end{Definition}
Some authors use the term digraph to allow self edges, but we do not allow this here.
\begin{Definition} \label{def:graph}
 A digraph $\pr{X}{R}$ is called a \emph{graph} if
 \begin{align*}
  \forall x, y \in X\[\pr{x}{y} \in R \leftrightarrow \pr{y}{x} \in R\].
 \end{align*}
\end{Definition}
\begin{Definition} \label{def:boreldigraph}
 A digraph $\GGG = \pr{X}{R}$ is said to be Borel (analytic) if $X$ is a Polish space and $R \subseteq X \times X$ is Borel (analytic).
\end{Definition}
\begin{Definition} \label{def:path}
 For a digraph $\GGG = \pr{X}{R}$ and a subset $Y \subseteq X$, an \emph{$R$-path in $Y$} is a finite non-empty sequence $\langle {y}_{0}, \dotsc, {y}_{k} \rangle$ of elements of $Y$ such that ${y}_{i} R {y}_{i+1}$, for all $i < k$.
 An $R$-path $\langle {y}_{0}, \dotsc, {y}_{k} \rangle$ in $Y$ is an \emph{$R$-cycle in $Y$} if ${y}_{k} R {y}_{0}$.
 The \emph{length} of an $R$-path is one less than the length of the sequence.
 In other words, the length of an $R$-path $\langle {y}_{0}, \dotsc, {y}_{k} \rangle$ is $k$.
 Observe that for any $y \in Y$, $\langle y \rangle$ is an $R$-path of length $0$ in $Y$, but it is not an $R$-cycle because $\pr{y}{y} \notin R$.
 We will write \emph{path} and \emph{cycle} instead of $R$-path and $R$-cycle when $R$ is clear from the context.
\end{Definition}
The next definition associates a digraph with every quasi order.
It will be shown that in all non-trivial cases, the order dimension of a quasi order coincides with the dichromatic number of its associated digraph.
This characterization will point us to the correct definition of Borel order dimension.
\begin{Definition} \label{def:aprp}
 Let $\PPP = \pr{P}{\leq}$ be a quasi order.
 Define
 \begin{align*}
  {\AAA}_{\PPP} = \{\pr{x}{y} \in P \times P: y \nleq x \}.
 \end{align*}
 For $\pr{{x}_{0}}{{y}_{0}}, \pr{{x}_{1}}{{y}_{1}} \in {\AAA}_{\PPP}$, define
 \begin{align*}
  \pr{{x}_{0}}{{y}_{0}} \; {\RRR}_{\PPP} \; \pr{{x}_{1}}{{y}_{1}} \iff {y}_{0} \leq {x}_{1}.
 \end{align*}
 Note that $\pr{{\AAA}_{\PPP}}{{\RRR}_{\PPP}}$ is a digraph.
 Further, if $\PPP$ is a Borel quasi order, then $\pr{{\AAA}_{\PPP}}{{\RRR}_{\PPP}}$ is a Borel digraph.
\end{Definition}
\begin{Lemma} \label{lem:paths}
 Suppose that $\PPP = \pr{P}{\leq}$ is a quasi order and $X \subseteq {\AAA}_{\PPP}$.
 Let $Y$ be the transitive closure of $X$ in $P \times P$.
 For any $\pr{p}{q} \in Y$, there is a path $\langle \pr{{x}_{0}}{{y}_{0}}, \dotsc, \pr{{x}_{k}}{{y}_{k}} \rangle$ in $X$ with $p = {x}_{0}$ and $q = {y}_{k}$.
\end{Lemma}
\begin{proof}
 Let $\pr{p}{q} \in Y$.
 Then there are $l \geq 1$ and a sequence $\langle {r}_{0}, \dotsc, {r}_{l} \rangle$ of elements of $P$ such that $p = {r}_{0}$, $q = {r}_{l}$, and $\pr{{r}_{i}}{{r}_{i+1}} \in X$, for all $i < l$.
 The proof is by induction on this $l$.
 If $l = 1$, then $\langle \pr{p}{q} \rangle$ is the required path in $X$.
 Assume the statement for some $l \geq 1$ and let $\langle {r}_{0}, \dotsc, {r}_{l}, {r}_{l+1} \rangle$ be a sequence of elements of $P$ such that $p = {r}_{0}$, $q = {r}_{l+1}$, and $\pr{{r}_{i}}{{r}_{i+1}} \in X$, for all $i < l+1$.
 By the induction hypothesis, there is a path $\langle \pr{{x}_{0}}{{y}_{0}}, \dotsc, \pr{{x}_{k}}{{y}_{k}} \rangle$ in $X$ with ${x}_{0} = {r}_{0} = p$ and ${y}_{k} = {r}_{l}$.
 Since $\pr{{r}_{l}}{{r}_{l+1}} \in X$ and ${y}_{k} \leq {r}_{l}$, $\pr{{x}_{k}}{{y}_{k}} {\RRR}_{\PPP} \pr{{r}_{l}}{{r}_{l+1}}$ holds.
 Thus $\langle \pr{{x}_{0}}{{y}_{0}}, \dotsc, \pr{{x}_{k}}{{y}_{k}}, \pr{{r}_{l}}{{r}_{l+1}} \rangle$ is the required path in $X$, completing the induction.
\end{proof}
\begin{Lemma} \label{lem:pathsandleq}
 Suppose that $\PPP = \pr{P}{\leq}$ is a quasi order and $X \subseteq {\AAA}_{\PPP}$.
 Let $Z$ be the transitive closure of $\leq \cup \; X$ in $P \times P$.
 For any $\pr{p}{q} \in Z \; \setminus \leq$, there is a path $\langle \pr{{x}_{0}}{{y}_{0}}, \dotsc, \pr{{x}_{k}}{{y}_{k}} \rangle$ in $X$ with $p \leq {x}_{0}$ and ${y}_{k} \leq q$.
\end{Lemma}
\begin{proof}
 Let $Y$ be the transitive closure of $X$ in $P \times P$.
 Then $Z$ is the transitive closure of $\leq \cup \; Y$ in $P \times P$.
 Let $\pr{p}{q} \in Z \; \setminus \leq$.
 Then there are $l \geq 1$ and a sequence $\langle {r}_{0}, \dotsc, {r}_{l} \rangle$ of elements of $P$ such that $p = {r}_{0}$, $q = {r}_{l}$, and either $\pr{{r}_{i}}{{r}_{i+1}} \in \; \leq$ or $\pr{{r}_{i}}{{r}_{i+1}} \in Y$, for all $i < l$.
 The proof is by induction on this $l$.
 First suppose $l = 1$.
 Then since $\pr{p}{q} \notin \; \leq$, $\pr{p}{q} \in Y$.
 So by Lemma \ref{lem:paths}, there is a path $\langle \pr{{x}_{0}}{{y}_{0}}, \dotsc, \pr{{x}_{k}}{{y}_{k}} \rangle$ in $X$ with $p \leq p = {x}_{0}$ and ${y}_{k} = q \leq q$.
 This is as required.
 Now assume the statement for some $l \geq 1$ and let $\langle {r}_{0}, \dotsc, {r}_{l}, {r}_{l+1} \rangle$ be a sequence of elements of $P$ such that $p = {r}_{0}$, $q = {r}_{l+1}$, and either $\pr{{r}_{i}}{{r}_{i+1}} \in \; \leq$ or $\pr{{r}_{i}}{{r}_{i+1}} \in Y$, for all $i < l+1$.
 Then $\pr{{r}_{0}}{{r}_{l}} \in Z$.
 First suppose that $\pr{{r}_{0}}{{r}_{l}} \in \; \leq$.
 If $\pr{{r}_{l}}{{r}_{l+1}} \in \; \leq$, then $p = {r}_{0} \leq {r}_{l} \leq {r}_{l+1} = q$, contradicting $\pr{p}{q} \notin \; \leq$.
 So $\pr{{r}_{l}}{{r}_{l+1}} \in Y$.
 By Lemma \ref{lem:paths}, there is a path $\langle \pr{{x}_{0}}{{y}_{0}}, \dotsc, \pr{{x}_{k}}{{y}_{k}} \rangle$ in $X$ with $p = {r}_{0} \leq {r}_{l} = {x}_{0}$ and ${y}_{k} = {r}_{l+1} = q \leq q$.
 This is as needed.
 Now suppose that $\pr{{r}_{0}}{{r}_{l}} \notin \; \leq$.
 Then the induction hypothesis applies and implies there is a path $\langle \pr{{x}_{0}}{{y}_{0}}, \dotsc, \pr{{x}_{k}}{{y}_{k}} \rangle$ in $X$ with $p = {r}_{0} \leq {x}_{0}$ and ${y}_{k} \leq {r}_{l}$.
 If $\pr{{r}_{l}}{{r}_{l+1}} \in \; \leq$, then ${y}_{k} \leq {r}_{l} \leq {r}_{l+1} = q$, and this is as required.
 So suppose that $\pr{{r}_{l}}{{r}_{l+1}} \in Y$.
 Then by Lemma \ref{lem:paths}, there is a path $\langle \pr{{w}_{0}}{{z}_{0}}, \dotsc, \pr{{w}_{m}}{{z}_{m}} \rangle$ in $X$ with ${r}_{l} = {w}_{0}$ and ${r}_{l+1} = {z}_{m}$.
 Since ${y}_{k} \leq {r}_{l} = {w}_{0}$, $\pr{{x}_{k}}{{y}_{k}} {\RRR}_{\PPP} \pr{{w}_{0}}{{z}_{0}}$ holds.
 Therefore, ${\langle \pr{{x}_{0}}{{y}_{0}}, \dotsc, \pr{{x}_{k}}{{y}_{k}} \rangle}^{\frown}{\langle \pr{{w}_{0}}{{z}_{0}}, \dotsc, \pr{{w}_{m}}{{z}_{m}} \rangle}$ is a path in $X$.
 Since $p \leq {x}_{0}$ and ${z}_{m} = {r}_{l+1} = q \leq q$, this is as required, completing the induction.
\end{proof}
\begin{Lemma} \label{lem:cyclefreeextends}
 Suppose that $\PPP = \pr{P}{\leq}$ is a quasi order and $X \subseteq {\AAA}_{\PPP}$.
 Let $\preceq$ be the transitive closure of $\leq \cup \; X$ in $P \times P$.
 If ${E}_{\leq} \neq {E}_{\preceq}$, then $X$ contains a cycle.
\end{Lemma}
\begin{proof}
 Note that $\preceq$ is a quasi order on $P$ with $\leq \; \subseteq \; \preceq$.
 Hence ${E}_{\preceq}$ is an equivalence relation with ${E}_{\leq} \subseteq {E}_{\preceq}$.
 Suppose $p, q \in P$ and that $p \preceq q \preceq p$.
 Assume first that $p \nleq q$ and $q \nleq p$.
 Then by Lemma \ref{lem:pathsandleq} there exist paths $\langle \pr{{x}_{0}}{{y}_{0}}, \dotsc, \pr{{x}_{k}}{{y}_{k}} \rangle$ and $\langle \pr{{w}_{0}}{{z}_{0}}, \dotsc, \pr{{w}_{l}}{{z}_{l}} \rangle$ in $X$ with $p \leq {x}_{0}$, ${y}_{k} \leq q$, $q \leq {w}_{0}$ and ${z}_{l} \leq p$.
 Thus ${y}_{k} \leq {w}_{0}$ and ${z}_{l} \leq {x}_{0}$, and so, $\pr{{x}_{k}}{{y}_{k}} {\RRR}_{\PPP} \pr{{w}_{0}}{{z}_{0}}$ and $\pr{{w}_{l}}{{z}_{l}} {\RRR}_{\PPP} \pr{{x}_{0}}{{y}_{0}}$ hold.
 It follows that ${\langle \pr{{x}_{0}}{{y}_{0}}, \dotsc, \pr{{x}_{k}}{{y}_{k}} \rangle}^{\frown}{\langle \pr{{w}_{0}}{{z}_{0}}, \dotsc, \pr{{w}_{l}}{{z}_{l}} \rangle}$ is a cycle in $X$.
 Next, suppose that $p \nleq q$, but $q \leq p$.
 Then there exists a path $\langle \pr{{x}_{0}}{{y}_{0}}, \dotsc, \pr{{x}_{k}}{{y}_{k}} \rangle$ in $X$ with $p \leq {x}_{0}$ and ${y}_{k} \leq q$.
 Thus ${y}_{k} \leq q \leq p \leq {x}_{0}$, and so $\pr{{x}_{k}}{{y}_{k}} {\RRR}_{\PPP} \pr{{x}_{0}}{{y}_{0}}$ holds.
 It follows that $\langle \pr{{x}_{0}}{{y}_{0}}, \dotsc, \pr{{x}_{k}}{{y}_{k}} \rangle$ is a cycle in $X$.
 Finally, suppose that $p \leq q$, but $q \nleq p$.
 Then there exists a path $\langle \pr{{w}_{0}}{{z}_{0}}, \dotsc, \pr{{w}_{l}}{{z}_{l}} \rangle$ in $X$ with $q \leq {w}_{0}$ and ${z}_{l} \leq p$.
 Thus ${z}_{l} \leq p \leq q \leq {w}_{0}$, and so $\pr{{w}_{l}}{{z}_{l}} {\RRR}_{\PPP} \pr{{w}_{0}}{{z}_{0}}$ holds.
 It follows that $\langle \pr{{w}_{0}}{{z}_{0}}, \dotsc, \pr{{w}_{l}}{{z}_{l}} \rangle$ is a cycle in $X$.
 Therefore if $X$ contains no cycles, then necessarily, $p \leq q \leq p$.
 This shows ${E}_{\preceq} \; \subseteq \; {E}_{\leq}$.
\end{proof}
The dichromatic number of a digraph was defined by Neumann-Lara~\cite{MR0693366} (see also \cite[pp.\@ 17--20]{MR0593699}).
It has become a well-studied notion in finite combinatorics (e.g.\@ \cite{MR1310883}, \cite{MR4039600}, \cite{MR3593495}) with important connections to various other notions in graph theory.
For instance, the well-known conjecture of Erd{\H{o}}s and Hajnal about cliques or anti-cliques in graphs with forbidden subgraphs is known, by \cite{MR1832443}, to be equivalent to a statement about the dichromatic number of certain tournaments.
Soukup~\cite{MR3818601} has studied the dichromatic number of uncountable digraphs.
In a recent work, Higgins~\cite{arXiv:2405.00991} has proved a Brooks type theorem for the measurable dichromatic number of Borel digraphs.

The next definition collects the various notions of chromatic number and dichromatic number that will be considered in this paper together.
We define both a weak notion and a strong notion of Borel dichromatic number.
The weak notion only requires a cover of the digraph by Borel cycle-free subsets, while the strong notion asks for a Borel mapping into a Polish space such that the preimage of each point is cycle-free.
The difference between the two notions boils down to this: the weak notion sees cardinals between ${\aleph}_{0}$ and ${2}^{{\aleph}_{0}}$ whereas the strong notion does not.
More precisely, the weak notion can consistently take values between ${\aleph}_{0}$ and ${2}^{{\aleph}_{0}}$ (see Section \ref{sec:consistency}), while the strong notion is equal to ${2}^{{\aleph}_{0}}$ as soon as it is uncountable.
Nevertheless, there is a very straightforward relationship between these two notions.
The weak notion is countable if and only if the strong one is, and in this case the two notions coincide.
The weak notion is uncountable if and only if the strong one is equal to ${2}^{{\aleph}_{0}}$.
We will formulate most of our results, including our dichotomy theorems, in terms of the weaker notion, but given the simple relationship between the two, the statements of our results would not change if the stronger notion were substituted everywhere.
In the classical case, there is no difference between the strong notion and the weak one.
\begin{Definition} \label{def:H}
 Let $\GGG = \pr{X}{R}$ be a digraph.
 The \emph{dichromatic number of $\GGG$}, denoted \emph{$\dicr(\GGG)$}, is the minimal cardinal $\kappa$ such that there exists a collection $\{{X}_{\alpha}: \alpha < \kappa\}$ such that $X = {\bigcup}_{\alpha < \kappa}{{X}_{\alpha}}$ and for each $\alpha < \kappa$, there are no $R$-cycles in ${X}_{\alpha}$.
 
 For a graph $\GGG = \pr{X}{R}$, the \emph{chromatic number of $\GGG$}, denoted \emph{$\chi(\GGG)$}, is the minimal cardinal $\kappa$ such that there exists a collection $\{{X}_{\alpha}: \alpha < \kappa\}$ such that $X = {\bigcup}_{\alpha < \kappa}{{X}_{\alpha}}$ and $\forall \alpha < \kappa \forall x, y \in {X}_{\alpha}\[\pr{x}{y} \notin R\]$.
 
 If $\GGG = \pr{X}{R}$ is an analytic digraph, then the \emph{weak Borel dichromatic number of $\GGG$}, denoted \emph{$\wbdicr(\GGG)$}, is the minimal cardinal $\kappa$ such that there exists a collection $\{{X}_{\alpha}: \alpha < \kappa\}$ of Borel subsets of $X$ such that $X = {\bigcup}_{\alpha < \kappa}{{X}_{\alpha}}$ and for each $\alpha < \kappa$, there are no $R$-cycles in ${X}_{\alpha}$.
 
 If $\GGG = \pr{X}{R}$ is an analytic graph, then the \emph{weak Borel chromatic number of $\GGG$}, denoted \emph{$\wbcr(\GGG)$}, is the minimal cardinal $\kappa$ such that there exists a collection $\{{X}_{\alpha}: \alpha < \kappa\}$ of Borel subsets of $X$ such that $X = {\bigcup}_{\alpha < \kappa}{{X}_{\alpha}}$, $\forall \alpha < \beta < \kappa\[{X}_{\alpha} \cap {X}_{\beta} = \emptyset\]$, and $\forall \alpha < \kappa \forall x, y \in {X}_{\alpha}\[\pr{x}{y} \notin R\]$.
 
 For an analytic digraph $\GGG = \pr{X}{R}$, the \emph{Borel dichromatic number of $\GGG$}, denoted $\bdicr(\GGG)$, is the minimal $\kappa$ such that there exist a Polish space $M$ and a Borel function $\varphi: X \rightarrow M$ such that $\lc M \rc = \kappa$ and for every $m \in M$, ${\varphi}^{-1}(\{m\})$ does not contain any $R$-cycles.
 
 For an analytic graph $\GGG = \pr{X}{R}$, the \emph{Borel chromatic number of $\GGG$}, denoted $\bcr(\GGG)$, is the minimal $\kappa$ such that there exist a Polish space $M$ and a Borel function $\varphi: X \rightarrow M$ such that $\lc M \rc = \kappa$ and $\forall x, y \in X\[\varphi(x) = \varphi(y) \implies \pr{x}{y} \notin R\]$.
\end{Definition}
Several remarks are in order.
The definitions of $\dicr(\GGG)$ and $\chi(\GGG)$ would not change if we added the requirement that $\{{X}_{\alpha}: \alpha < \kappa\}$ be a pairwise disjoint family.
Clearly, $\dicr(\GGG) \leq \wbdicr(\GGG) \leq \bdicr(\GGG)$, for any analytic digraph $\GGG$ and $\chi(\GGG) \leq \wbcr(\GGG) \leq \bcr(\GGG)$, for any analytic graph $\GGG$.
For any graph $\GGG$, $\chi(\GGG) = \dicr(\GGG)$ because not containing cycles is the same as not containing edges.
Clearly, any family of at most ${\aleph}_{1}$ many Borel sets can be refined to a pairwise disjoint family of Borel sets with the same union.
Therefore, for any analytic graph $\GGG$, if $\wbdicr(\GGG) \leq {\aleph}_{1}$, then $\wbdicr(\GGG) = \wbcr(\GGG)$.
In particular, for any analytic graph $\GGG$, $\wbdicr(\GGG) > {\aleph}_{0}$ if and only if $\wbcr(\GGG) > {\aleph}_{0}$.
Lastly, $\bdicr(\GGG)$ and $\bcr(\GGG)$ can only take on these values: finite, ${\aleph}_{0}$, or ${2}^{{\aleph}_{0}}$.
We will see in Section \ref{sec:consistency} that $\wbdicr(\GGG)$ and $\wbcr(\GGG)$ have a much larger spectrum of possible values.
Observe, however, that $\wbdicr(\GGG) > {\aleph}_{0}$ if and only if $\bdicr(\GGG) = {2}^{{\aleph}_{0}}$, for an analytic digraph $\GGG$, and $\wbcr(\GGG) > {\aleph}_{0}$ if and only if $\bcr(\GGG) = {2}^{{\aleph}_{0}}$, for an analytic graph $\GGG$.
Furthermore, if $\wbdicr(\GGG) \leq {\aleph}_{0}$, then $\wbdicr(\GGG) = \bdicr(\GGG)$, and if $\wbcr(\GGG) \leq {\aleph}_{0}$, then $\wbcr(\GGG) = \bcr(\GGG)$.
\begin{Lemma} \label{lem:extendscycles}
 Let $\PPP = \pr{P}{\leq}$ be a quasi order.
 Let $\preceq$ be a quasi order on $P$ that extends $\leq$.
 Define
 \begin{align*}
  X(\preceq, \leq) = \{\pr{x}{y} \in P \times P: y \nleq x \wedge x \preceq y\}.
 \end{align*}
 Then $X(\preceq, \leq)$ is a subset of ${\AAA}_{\PPP}$ that does not contain any cycles.
\end{Lemma}
\begin{proof}
 It is clear from the definitions that $X(\preceq, \leq) \subseteq {\AAA}_{\PPP}$.
 Suppose for a contradiction that $\langle \pr{{x}_{0}}{{y}_{0}}, \dotsc, \pr{{x}_{k}}{{y}_{k}} \rangle$ is a cycle in $X(\preceq, \leq)$.
 Then ${x}_{0} \preceq {y}_{0} \leq \dotsb \leq {x}_{k} \preceq {y}_{k} \leq {x}_{0}$.
 It follows that ${x}_{0} \; {E}_{\preceq} \; {y}_{0}$.
 Since ${E}_{\preceq} = {E}_{\leq}$, ${x}_{0} \; {E}_{\leq} \; {y}_{0}$.
 However, this is a contradiction because ${y}_{0} \nleq {x}_{0}$.
\end{proof}
\begin{Lemma} \label{lem:odim=H}
 Let $\PPP = \pr{P}{\leq}$ be a quasi order such that $\lc P \slash {E}_{\leq} \rc > 1$.
 Then $\odim(\PPP) = \dicr\left(\pr{{\AAA}_{\PPP}}{{\RRR}_{\PPP}}\right)$.
\end{Lemma}
\begin{proof}
 Suppose $\seq{\leq}{i}{\in}{\kappa}$ is a sequence of quasi orders on $P$ extending $\leq$ such that (\ref{cond:quasi}) of Lemma \ref{lem:quasidimchar1} is satisfied.
 For each $i \in \kappa$, define ${X}_{i} = X\left( {\leq}_{i}, \leq \right)$.
 By Lemma \ref{lem:extendscycles}, ${X}_{i} \subseteq {\AAA}_{\PPP}$ and it does not contain any cycles.
 Suppose $\pr{x}{y} \in {\AAA}_{\PPP}$.
 Then there is $i \in \kappa$ so that $x \; {\leq}_{i} \; y$, whence $\pr{x}{y} \in {X}_{i}$.
 Therefore ${\AAA}_{\PPP} = {\bigcup}_{i \in \kappa}{{X}_{i}}$.
 It follows that $\dicr\left(\pr{{\AAA}_{\PPP}}{{\RRR}_{\PPP}}\right) \leq \odim(\PPP)$.
 
 For the other direction, suppose ${\AAA}_{\PPP} = {\bigcup}_{i \in \kappa}{{X}_{i}}$, where each ${X}_{i}$ contains no cycles.
 The hypothesis that $\lc P \slash {E}_{\leq} \rc > 1$ implies that ${\AAA}_{\PPP} \neq \emptyset$.
 Therefore $0 < \kappa$.
 Define ${\leq}_{i}$ to be the transitive closure of $\leq \cup \; {X}_{i}$ in $P \times P$.
 By Lemma \ref{lem:cyclefreeextends}, each ${\leq}_{i}$ is a quasi order on $P$ that extends $\leq$.
 Let $x, y \in P$.
 If $x \leq y$, then $i = 0$ fulfills (\ref{cond:quasi}) of Lemma \ref{lem:quasidimchar1}.
 If $x \nleq y$, then $\pr{y}{x} \in {\AAA}_{\PPP}$, whence $\pr{y}{x} \in {X}_{i}$, for some $i \in \kappa$.
 Then $y \; {\leq}_{i} \; x$ and (\ref{cond:quasi}) of Lemma \ref{lem:quasidimchar1} is fulfilled by $i$.
 It follows that $\odim(\PPP) \leq \dicr\left(\pr{{\AAA}_{\PPP}}{{R}_{\PPP}}\right)$.
\end{proof}
Lemma \ref{lem:odim=H} suggests that the Borel order dimension of a Borel quasi order $\PPP$ \emph{ought} to be $\wbdicr(\pr{{\AAA}_{\PPP}}{{\RRR}_{\PPP}})$, or $\bdicr(\pr{{\AAA}_{\PPP}}{{\RRR}_{\PPP}})$ if a strong version is needed.
In fact, we will now prove that this is exactly what we get if we take the characterization of $\odim(\PPP)$ given by Lemma \ref{lem:quasidimchar1} and add the requirement that the witnessing sequence consist entirely of Borel quasi orders.
The proof of this will use Corollary \ref{cor:reflection}, which will also play a key role in the proofs of the dichotomies in Section \ref{sec:dichotomy}.
The proof of Corollary \ref{cor:reflection}, in turn, will use Theorem \ref{thm:firstreflection}, which can be found in Kechris~\cite{kechrisbook}.
\begin{Definition} \label{def:pi11onsigma11}
 Let $X$ be a Polish space and $\Phi \subseteq \Pset(X)$.
 We say that $\Phi$ is \emph{${\mathbf{\Pi}}^{1}_{1}$ on ${\mathbf{\Sigma}}^{1}_{1}$} if for any Polish space $Y$ and any $A \subseteq Y \times X$ which is ${\mathbf{\Sigma}}^{1}_{1}$, the set ${A}_{\Phi} = \{y \in Y: {A}_{y} \in \Phi\}$ is ${\mathbf{\Pi}}^{1}_{1}$.
 Here ${A}_{y} = \{x \in X: \pr{y}{x} \in A\}$, for every $y \in Y$.
\end{Definition}
\begin{Theorem}[First Reflection Theorem] \label{thm:firstreflection}
 Let $X$ be a Polish space and let $\Phi \subseteq \Pset(X)$ be ${\mathbf{\Pi}}^{1}_{1}$ on ${\mathbf{\Sigma}}^{1}_{1}$.
 Then for any ${\mathbf{\Sigma}}^{1}_{1}$ set $A \in \Phi$, there is a Borel set $B \subseteq X$ such that $A \subseteq B$ and $B \in \Phi$.
\end{Theorem}
\begin{Lemma} \label{lem:cyclefreereflection}
 Let $\GGG = \pr{X}{R}$ be an analytic digraph.
 Let
 \begin{align*}
  \Phi = \{A \subseteq X: A \ \text{does not contain a cycle}\}.
 \end{align*}
 Then $\Phi$ is ${\mathbf{\Pi}}^{1}_{1}$ on ${\mathbf{\Sigma}}^{1}_{1}$.
\end{Lemma}
\begin{proof}
 Let $Y$ be Polish and let $A \subseteq Y \times X$ be ${\mathbf{\Sigma}}^{1}_{1}$.
 Fix Polish spaces $M$ and $N$ and continuous functions $f: M \rightarrow X \times X$ and $g: N \rightarrow Y \times X$ with $f''M = R$ and $g''N = A$.
 For $k \in \omega$ define
 \begin{equation*}
  \begin{split}
   {C}_{k} = &\left\{ \left\langle y, {x}_{0}, \dotsc, {x}_{k}, {n}_{0}, \dotsc, {n}_{k}, {m}_{0}, \dotsc, {m}_{k} \right\rangle \in Y \times {X}^{k+1} \times {N}^{k+1} \times {M}^{k+1}:\vphantom{\left( {\bigwedge}_{i \leq k}{g({n}_{i}) = \pr{y}{{x}_{i}}} \right)}\right. \\
   &\left. \left( {\bigwedge}_{i \leq k}{g({n}_{i}) = \pr{y}{{x}_{i}}} \right) \wedge \left( {\bigwedge}_{i < k}{f({m}_{i}) = \pr{{x}_{i}}{{x}_{i+1}}} \right) \wedge f({m}_{k}) = \pr{{x}_{k}}{{x}_{0}} \right\}.
  \end{split}
 \end{equation*}
 Clearly ${C}_{k}$ is a closed subset of $Y \times {X}^{k+1} \times {N}^{k+1} \times {M}^{k+1}$, and so ${D}_{k} = {\proj}_{Y}\left({C}_{k}\right)$ is a ${\mathbf{\Sigma}}^{1}_{1}$ subset of $Y$.
 Therefore $D = {\bigcup}_{k \in \omega}{{D}_{k}}$ is a ${\mathbf{\Sigma}}^{1}_{1}$ subset of $Y$ and $E = Y \setminus D$ is a ${\mathbf{\Pi}}^{1}_{1}$ subset of $Y$.
 It is easily verified that $E = \{y \in Y: {A}_{y} \in \Phi\} = {A}_{\Phi}$.
\end{proof}
\begin{Corollary} \label{cor:reflection}
 Let $\GGG = \pr{X}{R}$ be an analytic digraph.
 If $A \subseteq X$ is ${\mathbf{\Sigma}}^{1}_{1}$ and does not contain cycles, then there exists a Borel set $B \subseteq X$ such that $A \subseteq B$ and $B$ does not contain cycles.
\end{Corollary}
\begin{Lemma} \label{lem:borelextension}
 Suppose $\PPP = \pr{P}{\leq}$ is a Borel quasi order.
 If ${\leq}_{0}$ is a ${\mathbf{\Sigma}}^{1}_{1}$ quasi order on $P$ that extends $\leq$, then there exists a Borel quasi order ${\leq}_{1}$ on $P$ that extends ${\leq}_{0}$.
\end{Lemma}
\begin{proof}
 As noted in Definition \ref{def:aprp}, $\pr{{\AAA}_{\PPP}}{{\RRR}_{\PPP}}$ is a Borel digraph.
 Build sequences $\seq{\preceq}{n}{\in}{\omega}$ and $\seq{X}{n}{\in}{\omega}$ satisfying the following:
 \begin{enumerate}
  \item
  ${\preceq}_{0}$ is ${\leq}_{0}$ and ${\preceq}_{n}$ is a ${\mathbf{\Sigma}}^{1}_{1}$ quasi order on $P$ extending $\leq$;
  \item
  ${\preceq}_{n+1}$ extends ${\preceq}_{n}$;
  \item
  ${X}_{n} \subseteq {\AAA}_{\PPP}$ is Borel and does not contain cycles;
  \item
  $X\left({\preceq}_{n}, \leq\right) \subseteq {X}_{n}$ and ${\preceq}_{n+1}$ is the transitive closure of $\leq \cup \; {X}_{n}$ in $P \times P$.
 \end{enumerate}
 To construct the sequences, let ${\preceq}_{0}$ be ${\leq}_{0}$.
 Given ${\preceq}_{n}$, by Lemma \ref{lem:extendscycles}, $X\left({\preceq}_{n}, \leq\right) \subseteq {\AAA}_{\PPP}$ and it does not contain cycles.
 Since ${\preceq}_{n}$ is ${\mathbf{\Sigma}}^{1}_{1}$ and $\leq$ is Borel, $X\left({\preceq}_{n}, \leq\right)$ is ${\mathbf{\Sigma}}^{1}_{1}$.
 By Corollary \ref{cor:reflection}, there exists a Borel set ${X}_{n} \subseteq {\AAA}_{\PPP}$ such that $X\left({\preceq}_{n}, \leq\right) \subseteq {X}_{n}$ and ${X}_{n}$ does not contain cycles.
 Let ${\preceq}_{n+1}$ be the transitive closure of $\leq \cup \; {X}_{n}$ in $P \times P$.
 By Lemma \ref{lem:cyclefreeextends}, ${\preceq}_{n+1}$ is a quasi order on $P$ extending $\leq$.
 Since $\leq$ and ${X}_{n}$ are Borel sets, ${\preceq}_{n+1}$ is ${\mathbf{\Sigma}}^{1}_{1}$.
 Note that ${E}_{{\preceq}_{n+1}} = {E}_{\leq} = {E}_{{\preceq}_{n}}$.
 Let $x, y \in P$ and suppose $x \; {\preceq}_{n} \; y$.
 If $y \leq x$, then $x \; {E}_{{\preceq}_{n}} \; y$, whence $x \; {E}_{{\preceq}_{n+1}} \; y$ and $x \; {\preceq}_{n+1} \; y$.
 If $y \nleq x$, then $\pr{x}{y} \in X\left({\preceq}_{n}, \leq\right) \subseteq {X}_{n}$, whence $x \; {\preceq}_{n+1} \; y$.
 Therefore ${\preceq}_{n+1}$ extends ${\preceq}_{n}$ and (1)--(4) are satisfied, finishing the construction.
 
 Define ${\leq}_{1} = {\bigcup}_{n \in \omega}{{\preceq}_{n}}$.
 By (1) and (2), ${\leq}_{1}$ is a quasi order on $P$ that extends ${\leq}_{0}$.
 Define $X = {\bigcup}_{n \in \omega}{{X}_{n}}$ and $R \; = \; \leq \cup \; X$.
 Then $R$ is Borel.
 It will be verified that $R = {\leq}_{1}$.
 If $\pr{x}{y} \in \; \leq$, then $\pr{x}{y} \in \; {\leq}_{0} = {\preceq}_{0} \subseteq {\leq}_{1}$.
 If $\pr{x}{y} \in {X}_{n}$, for some $n \in \omega$, then $\pr{x}{y} \in {\preceq}_{n+1} \subseteq {\leq}_{1}$.
 This shows $R \subseteq {\leq}_{1}$.
 For the other direction, consider $\pr{x}{y} \in {\preceq}_{n}$, for some $n \in \omega$.
 If $y \leq x$, then $y \; {\preceq}_{n} \; x$, whence $x \; {E}_{{\preceq}_{n}} \; y$, and so $x \; {E}_{\leq} \; y$ and $\pr{x}{y} \in \; \leq \; \subseteq R$.
 If $y \nleq x$, then $\pr{x}{y} \in X\left({\preceq}_{n}, \leq\right) \subseteq {X}_{n} \subseteq X \subseteq R$.
 Therefore, $R = {\leq}_{1}$ and ${\leq}_{1}$ is Borel.
\end{proof}
As our official definition of Borel order dimension, we will take the one derived from Definition \ref{def:dimquasi} and Lemma \ref{lem:quasidimchar1}, and we will show that it agrees with the one we would expect from Lemma \ref{lem:odim=H}.
\begin{Definition} \label{def:borelpartial}
 Let $\PPP = \pr{P}{\leq}$ be a Borel quasi order.
 A partial order $\prec$ on $P \slash {E}_{\leq}$ is said to be \emph{Borel} if there is a Borel quasi order $\preceq$ on $P$ such that ${E}_{\preceq} = {E}_{\leq}$ and $\prec$ is the partial order on $P \slash {E}_{\preceq} = P \slash {E}_{\leq}$ induced by $\preceq$.
\end{Definition}
\begin{Definition} \label{def:boreldim}
 Let $\PPP = \pr{P}{\leq}$ be a Borel quasi order.
 The \emph{Borel order dimension} (or simply \emph{Borel dimension}) \emph{of $\PPP$}, denoted \emph{${\odim}_{B}\left(\PPP\right)$}, is the minimal cardinal $\kappa$ such that there exists a sequence $\seq{<}{i}{\in}{\kappa}$ of Borel partial orders on $P \slash {E}_{\leq}$ extending $<$ such that
 \begin{align}
  \forall x, y \in P \exists i \in \kappa\[\[y\] \; {<}_{i} \; \[x\] \vee \[x\] < \[y\] \vee \[x\] = \[y\]\]. \label{cond:borelpartial}
 \end{align}
\end{Definition}
\begin{Lemma} \label{lem:boreldimquasi}
 Let $\PPP = \pr{P}{\leq}$ be a Borel quasi order.
 Then ${\odim}_{B}\left(\PPP\right)$ is the minimal cardinal $\kappa$ such that there exists a sequence $\seq{\leq}{i}{\in}{\kappa}$ of Borel quasi orders on P extending $\leq$ such that
 \begin{align}
  \forall x, y \in P \exists i \in \kappa\[y \; {\leq}_{i} \; x \vee x \leq y\]. \label{cond:borelquasi}
 \end{align}
\end{Lemma}
\begin{proof}
 Suppose $\seq{<}{i}{\in}{\kappa}$ is a sequence of Borel partial orders on $P \slash {E}_{\leq}$ extending $<$ and satisfying (\ref{cond:borelpartial}) of Definition \ref{def:boreldim}.
 Then for every $i \in \kappa$, there is a Borel quasi order ${\leq}_{i}$ on $P$ so that ${E}_{{\leq}_{i}} = {E}_{\leq}$ and ${<}_{i}$ is the partial order on $P \slash {E}_{{\leq}_{i}} = P \slash {E}_{\leq}$ induced by ${\leq}_{i}$.
 Each ${\leq}_{i}$ extends $\leq$.
 Given $x, y \in P$, let $i \in \kappa$ be such that $\left(\[y\] \; {<}_{i} \; \[x\] \vee \[x\] < \[y\] \vee \[x\] = \[y\] \right)$.
 If $\[y\] \; {<}_{i} \; \[x\]$, then $y \; {\leq}_{i} \; x$, while if $\left( \[x\] < \[y\] \vee \[x\] = \[y\] \right)$, then $x \leq y$.
 Thus (\ref{cond:borelquasi}) holds.
 
 Conversely, suppose $\seq{\leq}{i}{\in}{\kappa}$ is a sequence of Borel quasi orders on $P$ extending $\leq$ and satisfying (\ref{cond:borelquasi}).
 For each $i \in \kappa$, define ${<}_{i}$ to be the partial order on $P \slash {E}_{\leq} = P \slash {E}_{{\leq}_{i}}$ induced by ${\leq}_{i}$.
 Each ${<}_{i}$ is a Borel partial order on $P \slash {E}_{\leq}$.
 Each ${<}_{i}$ extends $<$.
 Given $x, y \in P$, there exists $i \in \kappa$ with $\left(y \; {\leq}_{i} \; x \vee x \leq y\right)$.
 If $\[x\] \neq \[y\]$, then $\left(\[y\] \; {<}_{i} \; \[x\] \vee \[x\] < \[y\]\right)$.
 Therefore, (\ref{cond:borelpartial}) of Definition \ref{def:boreldim} is satisfied.
\end{proof}
\begin{Theorem} \label{thm:odimB=HB}
 Let $\PPP = \pr{P}{\leq}$ be a Borel quasi order such that $\lc P \slash {E}_{\leq} \rc > 1$.
 Then ${\odim}_{B}\left(\PPP\right) = \wbdicr\left(\pr{{\AAA}_{\PPP}}{{\RRR}_{\PPP}}\right)$.
\end{Theorem}
\begin{proof}
 Suppose $\seq{\leq}{i}{\in}{\kappa}$ is a sequence of Borel quasi orders on $P$ extending $\leq$ such that (\ref{cond:borelquasi}) of Lemma \ref{lem:boreldimquasi} is satisfied.
 For each $i \in \kappa$, define ${X}_{i} = X\left({\leq}_{i}, \leq\right)$.
 By the proof of Lemma \ref{lem:odim=H}, ${X}_{i} \subseteq {\AAA}_{\PPP}$, ${X}_{i}$ does not contain any cycles, and ${\AAA}_{\PPP} = {\bigcup}_{i \in \kappa}{{X}_{i}}$.
 Further, since both $\leq$ and ${\leq}_{i}$ are Borel, ${X}_{i}$ is a Borel set.
 It follows that $\wbdicr\left(\pr{{\AAA}_{\PPP}}{{\RRR}_{\PPP}}\right) \leq {\odim}_{B}\left(\PPP\right)$.
 
 For the other direction, suppose ${\AAA}_{\PPP} = {\bigcup}_{i \in \kappa}{{X}_{i}}$, where each ${X}_{i}$ is Borel and contains no cycles.
 The hypothesis that $\lc P \slash {E}_{\leq} \rc > 1$ implies that ${\AAA}_{\PPP} \neq \emptyset$.
 Therefore $0 < \kappa$.
 Define ${\preceq}_{i}$ to be the transitive closure of $\leq \cup \; {X}_{i}$ in $P \times P$.
 By Lemma \ref{lem:cyclefreeextends}, each ${\preceq}_{i}$ is a quasi order on $P$ that extends $\leq$.
 Since $\leq$ and ${X}_{i}$ are both Borel, ${\preceq}_{i}$ is ${\mathbf{\Sigma}}^{1}_{1}$.
 By Lemma \ref{lem:borelextension}, there is a Borel quasi order ${\leq}_{i}$ on $P$ that extends ${\preceq}_{i}$.
 Let $x, y \in P$.
 If $x \leq y$, then $i = 0$ fulfills (\ref{cond:borelquasi}) of Lemma \ref{lem:boreldimquasi}.
 If $x \nleq y$, then $\pr{y}{x} \in {\AAA}_{\PPP}$, whence $\pr{y}{x} \in {X}_{i}$, for some $i \in \kappa$.
 Then $y \; {\preceq}_{i} \; x$ and $y \; {\leq}_{i} \; x$, and (\ref{cond:borelquasi}) of Lemma \ref{lem:boreldimquasi} is fulfilled by $i$.
 It follows that ${\odim}_{B}\left(\PPP\right) \leq \wbdicr\left(\pr{{\AAA}_{\PPP}}{{\RRR}_{\PPP}}\right)$.
\end{proof}
In view of Theorem \ref{thm:odimB=HB}, we could think of ${\odim}_{B}(\PPP)$ as the weak Borel order dimension of $\PPP$.
The strong notion would then be defined as $\bdicr(\pr{{\AAA}_{\PPP}}{{\RRR}_{\PPP}})$ in the case when $\lc P \slash {E}_{\leq} \rc > 1$, and as $0$ or $1$ in the cases when $P \slash {E}_{\leq}$ is empty or a singleton respectively.
Replacing our weaker definition with the stronger notion will not make any difference to the dichotomy theorems in Section \ref{sec:dichotomy}.
\begin{Definition} \label{def:linearizable}
 A Borel quasi order $\PPP = \pr{P}{\leq}$ is said to be \emph{Borel linearizable} if there exists a Borel linear quasi order $\preceq$ on $P$ extending $\leq$.
\end{Definition}
Our definition of ${\AAA}_{\PPP}$ considers all pairs $\pr{x}{y}$ where $y \nleq x$.
If we considered only the induced sub-digraph of $\pr{{\AAA}_{\PPP}}{{\RRR}_{\PPP}}$ on the collection of pairs $\pr{x}{y}$ where $x$ and $y$ are incomparable, then we would be off from the Borel dichromatic number of $\PPP$ by at most $1$.
We will conclude this section by proving this fact, which will be useful in the proof of one of the dichotomies to be presented in the next section.
\begin{Definition} \label{def:BpSp}
 Let $\PPP = \pr{P}{\leq}$ be a Borel quasi order.
 Define
 \begin{align*}
  &{\BBB}_{\PPP} = \left\{ \pr{x}{y} \in {\AAA}_{\PPP}: x \nleq y \right\};\\
  &{\SSS}_{\PPP} = {\RRR}_{\PPP} \cap \left( {\BBB}_{\PPP} \times {\BBB}_{\PPP} \right).
 \end{align*}
 It is clear that $\pr{{\BBB}_{\PPP}}{{\SSS}_{\PPP}}$ is a Borel digraph.
\end{Definition}
\begin{Lemma} \label{lem:H1+}
 For any Borel quasi order $\PPP = \pr{P}{\leq}$, $\wbdicr\left( \pr{{\AAA}_{\PPP}}{{\RRR}_{\PPP}} \right) \leq 1 + \wbdicr\left( \pr{{\BBB}_{\PPP}}{{\SSS}_{\PPP}} \right)$.
\end{Lemma}
\begin{proof}
 Let $\kappa = \wbdicr\left( \pr{{\BBB}_{\PPP}}{{\SSS}_{\PPP}} \right)$ and let $\seq{Y}{\alpha}{<}{\kappa}$ witness this.
 Define ${X}_{0} = {\AAA}_{\PPP} \setminus {\BBB}_{\PPP}$.
 Suppose $\left\langle \pr{{p}_{0}}{{q}_{0}}, \dotsc, \pr{{p}_{k}}{{q}_{k}} \right\rangle$ is an ${\RRR}_{\PPP}$-cycle in ${X}_{0}$.
 Then we have ${p}_{0} \leq {q}_{0} \leq \dotsb \leq {p}_{k} \leq {q}_{k} \leq {p}_{0}$, whence ${q}_{0} \leq {p}_{0}$, contradicting $\pr{{p}_{0}}{{q}_{0}} \in {\AAA}_{\PPP}$.
 Thus ${X}_{0}$ is a Borel subset of ${\AAA}_{\PPP}$ without any ${\RRR}_{\PPP}$-cycles.
 For each $\alpha < \kappa$, define ${X}_{1+\alpha} = {Y}_{\alpha}$.
 It is easy to check that each ${X}_{1+\alpha}$ is a Borel subset of ${\AAA}_{\PPP}$ without any ${\RRR}_{\PPP}$-cycles.
 Since ${\AAA}_{\PPP} = {X}_{0} \cup {\BBB}_{\PPP} = {\bigcup}_{\beta < 1+\kappa}{{X}_{\beta}}$, $\seq{X}{\beta}{<}{1+\kappa}$ witnesses that $\wbdicr\left( \pr{{\AAA}_{\PPP}}{{\RRR}_{\PPP}} \right) \leq 1 + \kappa = 1 + \wbdicr\left( \pr{{\BBB}_{\PPP}}{{\SSS}_{\PPP}} \right)$.
\end{proof}
\begin{Corollary} \label{cor:HBS}
 For a Borel quasi order $\PPP = \pr{P}{\leq}$, if ${\odim}_{B}(\PPP) > {\aleph}_{0}$, then $\wbdicr\left(\pr{{\BBB}_{\PPP}}{{\SSS}_{\PPP}}\right) > {\aleph}_{0}$.
\end{Corollary}
\begin{proof}
 Since ${\odim}_{B}(\PPP) > {\aleph}_{0}$, $\lc P \slash {E}_{\leq} \rc > 1$, and $\wbdicr\left( \pr{{\AAA}_{\PPP}}{{\RRR}_{\PPP}} \right) > {\aleph}_{0}$.
 Since by Lemma \ref{lem:H1+}, $\wbdicr\left( \pr{{\AAA}_{\PPP}}{{\RRR}_{\PPP}} \right) \leq 1 + \wbdicr\left( \pr{{\BBB}_{\PPP}}{{\SSS}_{\PPP}} \right)$, $\wbdicr\left( \pr{{\BBB}_{\PPP}}{{\SSS}_{\PPP}} \right) > {\aleph}_{0}$.
\end{proof}
\section{Three Dichotomies} \label{sec:dichotomy}
This section contains the main results of this paper.
These results take the form of three dichotomy theorems that characterize the Borel and analytic digraphs of uncountable weak Borel dichromatic number as well as the Borel quasi orders of uncountable Borel order dimension in terms of a basis.
The first two dichotomies generalize the celebrated ${\GGG}_{0}$-dichotomy of Kechris, Solecki, and Todorcevic~\cite{MR1667145}.
Kechris, Solecki, and Todorcevic identified a particular Borel graph, called ${\GGG}_{0}$, and showed that an arbitrary analytic graph $\GGG$ has uncountable weak Borel chromatic number if and only if there is a continuous graph homomorphism from ${\GGG}_{0}$ to $\GGG$.
The ${\GGG}_{0}$-dichotomy has found numerous applications in descriptive set theory (see Miller~\cite{MR3053069}) and has inspired an impressive body of work in Borel combinatorics (e.g.\@ \cite{MR2467208}, \cite{MR3454384}, \cite{MR3362223}, \cite{MR4309494}).

The basic object ${\GGG}_{0}$ of \cite{MR1667145} is an acyclic graph defined on ${2}^{\omega}$.
The basic objects in the dichotomies about to be presented will be digraphs defined on finitely branching trees, where the nodes of height $n$ branch into $f(n)$ successors, for some $f \in {\omega}^{\omega}$ that satisfies $f(n) \geq 2$.
Thus instead of one basic object, we get a family of continuum many basic objects parametrized by a real in a uniform way.
In the ${\GGG}_{0}$ dichotomy, the natural morphisms of interest are continuous graph homomorphisms, and it is known that, in general, one cannot ask for an embedding.
In the category of digraphs, it is natural to consider homomorphisms that map every induced copy of the directed cycle ${C}_{n}$ to an induced copy of ${C}_{n}$ (see, for example, \cite{MR3488937} and \cite{MR3567535}).
In particular, such digraph homomorphisms preserve specific non-edges as well.
Our first dichotomy will show that if $\GGG$ is a Borel digraph with $\wbdicr(\GGG) > {\aleph}_{0}$, then one can always ask for a continuous digraph homomorphism from one of our basic objects to $\GGG$ that preserves all induced copies of ${C}_{n}$.
\begin{Definition} \label{def:Fn<omega}
 Define $\FFF = \left\{f \in \BS: \forall k \in \omega\[f(k) \geq 2\]\right\}$.
 For $l \in \omega$, define ${\FFF}_{l} = \left\{s \in {\omega}^{l}: \forall k \in l\[s(k) \geq 2\]\right\}$.
 Define ${\FFF}_{< \omega} = {\bigcup}_{l < \omega}{{\FFF}_{l}}$.
\end{Definition}
\begin{Definition} \label{def:tree}
 Let $X$ be any set.
 $T \subseteq {X}^{< \omega}$ is called a \emph{subtree} if $T$ is downwards closed.
 In other words, $T \subseteq {X}^{< \omega}$ is a subtree if and only if $\forall t \in T \forall l' \leq \dom(t)\[t\restrict l' \in T\]$.
 For a subtree $T \subseteq {X}^{< \omega}$, the \emph{$n$th level of $T$}, denoted \emph{$\lv{T}{n}$}, is $\{s \in T: \dom(s) = n \}$, for all $n \in \omega$.
 For a subtree $T \subseteq {X}^{< \omega}$ and $s \in T$, define
 \begin{align*}
  &T\langle s \rangle = \left\{ t \in T: s \subseteq t \vee t \subseteq s \right\};\\
  &{\succc}_{T}(s) = \left\{ x: {s}^{\frown}{\langle x \rangle} \in T \right\}.
 \end{align*}
 For a subtree $T \subseteq {X}^{< \omega}$, $\[T\]$ denotes the collection of all infinite branches through $T$.
 In other words,
 \begin{align*}
  \[T\] = \left\{ f \in {X}^{\omega}: \forall l \in \omega \[f \restrict l \in T \]\right\}.
 \end{align*}
\end{Definition}
\begin{Definition} \label{def:T[s]Ts}
 For $\sigma \in {\FFF}_{< \omega}$, define
 \begin{align*}
  &{T}_{\[\sigma\]} = \displaystyle\prod_{k \in \dom(\sigma)}{\sigma(k)}\\
  &{T}_{\sigma} = \displaystyle\bigcup_{l \leq \dom(\sigma)}{\displaystyle\prod_{k \in l}{\sigma(k)}}
 \end{align*}
 Note ${T}_{\sigma} \subseteq {\omega}^{\leq \dom(\sigma)} \subseteq {\omega}^{< \omega}$ is a subtree and ${T}_{\[\sigma\]} \subseteq {T}_{\sigma}$.
\end{Definition}
\begin{Definition} \label{def:Tf}
 For $f \in \FFF$, define
 \begin{align*}
  {T}_{f} = \displaystyle\bigcup_{l \in \omega}{\displaystyle\prod_{k \in l}{f(k)}}
 \end{align*}
 Note ${T}_{f} \subseteq {\omega}^{< \omega}$ is a subtree and that $\[{T}_{f}\] = \displaystyle\prod_{k \in \omega}{f(k)}$.
\end{Definition}
The definition of the graph ${\GGG}_{0}$ of~\cite{MR1667145} depends on a parameter $D$, which is a dense subset of ${2}^{< \omega}$ that intersects every level exactly once.
The choice of $D$ does not affect the properties of ${\GGG}_{0}$, but in our case, the choice of a dense set will determine the function $f$, which specifies the extent of the branching at each level.
Hence, we take a general approach and define the following.
\begin{Definition} \label{def:selector}
 A function $E: {\FFF}_{< \omega} \rightarrow {\omega}^{< \omega}$ is called a \emph{selector} if for every $\sigma \in {\FFF}_{< \omega}$, $E(\sigma) \in {T}_{\[\sigma\]}$.
 A selector $E$ is \emph{dense} if for every $f \in \FFF$, the set $\left\{ E(f \restrict l): l \in \omega \right\}$ is dense in ${T}_{f}$.
 In other words, a selector $E$ is dense if and only if
 \begin{align*}
  \forall f \in \FFF \forall s \in {T}_{f} \exists l \in \omega\[s \subseteq E(f \restrict l)\].
 \end{align*}
\end{Definition}
It is an easy exercise to see that dense selectors exist.
We will outline a construction for completeness.
Let $\{{s}_{l}: l \in \omega\}$ be an enumeration of ${\omega}^{< \omega}$ such that $\lc {s}_{l} \rc \leq l$, for every $l \in \omega$.
Define $E: {\FFF}_{< \omega} \rightarrow {\omega}^{< \omega}$ as follows.
Given $\sigma \in {\FFF}_{< \omega}$, let $l = \lc \sigma \rc$.
If ${s}_{l} \in {T}_{\sigma}$, then $E(\sigma) = \tau$, where $\tau \in {T}_{\[\sigma\]}$ is such that $\tau \restrict \lc {s}_{l} \rc = {s}_{l}$ and $\tau(i) = 0$, for all $\lc {s}_{l} \rc \leq i < \lc \sigma \rc$.
If ${s}_{l} \notin {T}_{\sigma}$, then $E(\sigma) = \tau$, where $\tau \in {T}_{\[\sigma\]}$ is such that $\tau(i) = 0$, for all $i < \lc \sigma \rc$.
To see that this works, suppose $f \in \FFF$ and $s \in {T}_{f}$.
Then $s = {s}_{l}$, for some $l \in \omega$.
Let $\sigma = f \restrict l$.
Since $\lc s \rc \leq l$ and $s \in {T}_{f}$, $s \in {T}_{\sigma}$.
So, by definition, $s \subseteq E(\sigma) = E(f \restrict l)$.

We now define a Borel digraph ${\GGG}_{0}(E, f)$ for every selector $E$ and every $f \in {\FFF}_{< \omega}$.
It will be shown that the family $\{{\GGG}_{0}(E, f): f \in \FFF \ \text{and} \ E \ \text{is a dense selector}\}$ characterizes the Borel and analytic digraphs of uncountable weak Borel dichromatic number.
In fact, we will show more.
We will show that if $\GGG$ is any analytic digraph whose weak Borel dichromatic number is uncountable, then for every $E$ there exist $f$ and a continuous homomorphism $\Pi$ from ${\GGG}_{0}(E, f)$ to $\GGG$, with the additional feature that $\Pi$ can be made to preserve copies of ${C}_{n}$ when $\GGG$ is Borel.
Moreover, if there exists $k$ so that $\GGG$ does not contain copies of ${C}_{n}$ for any $n \geq k$, then $f$ is bounded by the constant function $k-1$.
Thus, our dichotomy directly generalizes the ${\GGG}_{0}$ dichotomy (see the proof of Corollary \ref{cor:g0dichotomy}).
\begin{Definition} \label{def:G0}
 Let $E$ be a selector and let $\sigma \in {\FFF}_{< \omega}$.
 Define ${R}_{E, \sigma} \subseteq {T}_{\[\sigma\]} \times {T}_{\[\sigma\]}$ as follows.
 Given $\pr{s}{t} \in {T}_{\[\sigma\]} \times {T}_{\[\sigma\]}$, $\pr{s}{t} \in {R}_{E, \sigma}$ if and only if there exist $k < \dom(\sigma)$ and $x \in {\omega}^{< \omega}$ such that $t(k) = s(k) + 1 \mod \sigma(k)$, $s = \ins{E(\sigma \restrict k)}{s(k)}{x}$, and $t = \ins{E(\sigma \restrict k)}{t(k)}{x}$.
 Define $\KKK(E, \sigma) = \pr{{T}_{\[\sigma\]}}{{R}_{E, \sigma}}$.
 Note that $\KKK(E, \sigma)$ is a digraph.
 
 Let $E$ be a selector let $f \in \FFF$.
 Define ${R}_{E, f} \subseteq \[{T}_{f}\] \times \[{T}_{f}\]$ as follows.
 Given $\pr{y}{z} \in \[{T}_{f}\] \times \[{T}_{f}\]$, $\pr{y}{z} \in {R}_{E, f}$ if and only if there exist $l \in \omega$ and $x \in {\omega}^{\omega}$ such that $z(l) = y(l) + 1 \mod f(l)$, $y = \ins{E(f \restrict l)}{y(l)}{x}$, and $z = \ins{E(f \restrict l)}{z(l)}{x}$.
 Define ${\GGG}_{0}(E, f) = \pr{\[{T}_{f}\]}{{R}_{E, f}}$.
 Note that ${\GGG}_{0}(E, f)$ is a Borel digraph.
\end{Definition}
Note that when $f(n) > 2$, for every $n$, then $\pr{y}{x} \notin {R}_{E, f}$ whenever $\pr{x}{y} \in {R}_{E, f}$.
Also, if $f$ is constantly $2$, then ${\GGG}_{0}(E, f)$ is a copy of the graph ${\GGG}_{0}$.
\begin{Definition} \label{def:covm}
 Recall the following cardinal characteristic known as the \emph{covering number of the meager ideal}, or \emph{covering of meager} for short.
 \begin{align*}
  &\MMM = \{M \subseteq {2}^{\omega}: M \ \text{is Borel and meager}\},\\
  &\cov(\MMM) = \min\left\{\lc \FFF \rc: \FFF \subseteq \MMM \wedge {2}^{\omega} = \bigcup \FFF \right\}.
 \end{align*}
\end{Definition}
\begin{Lemma} \label{lem:covM}
 For every dense selector $E$ and every $f \in \FFF$, $\wbdicr\left( {\GGG}_{0}(E, f) \right) \geq \cov(\MMM)$.
\end{Lemma}
\begin{proof}
 For any $f \in \FFF$, $\[{T}_{f}\]$ is a perfect Polish space and hence it cannot be covered by fewer than $\cov(\MMM)$ many Borel meager sets.
 Therefore it suffices to show that if $X \subseteq \[{T}_{f}\]$ is Borel and non-meager, then $X$ contains a cycle.
 Indeed, we will argue that if $X \subseteq \[{T}_{f}\]$ is Baire measurable and non-meager, then $X$ contains a cycle.
 By Baire measurability, there is a non-empty open subset $U \subseteq \[{T}_{f}\]$ such that $X$ is comeager in $U$.
 By the hypothesis that $E$ is a dense selector, we may assume that there exists $l \in \omega$ so that, letting $s = E(f \restrict l)$, $U = \{x \in \[{T}_{f}\]: s \subseteq x\}$.
 Define $\psi: U \rightarrow U$ by
 \begin{align*}
  \psi(\ins{s}{i}{y}) = \ins{s}{i+1 \mod f(l)}{y},
 \end{align*}
 for each $i$ and $y$ satisfying $i \in f(l)$, $y \in \BS$, and $\ins{s}{i}{y} \in \[{T}_{f}\]$.
 Then $\psi$ and ${\psi}^{-1}$ are auto-homeomorphisms of $U$.
 For $k \in \omega$, define the iterates ${\psi}^{k}: U \rightarrow U$ and ${\psi}^{-k}: U \rightarrow U$ by induction as usual.
 In other words, define ${\psi}^{0}$ as the identity function on $U$, and given ${\psi}^{k}$ and ${\psi}^{-k}$, define ${\psi}^{k+1}$ as $\psi \circ {\psi}^{k}$ and ${\psi}^{-k-1}$ as ${\psi}^{-1} \circ {\psi}^{-k}$.
 Then for each $k$, ${\psi}^{-k}(U \cap X)$ is comeager in $U$.
 So $Y = {\bigcap}_{k < f(l)}{{\psi}^{-k}(U \cap X)}$ is comeager, in particular it is non-empty.
 Choose $x \in Y$.
 Then ${\psi}^{k}(x) \in U \cap X$, for every $k < f(l)$, and by the definitions of $\psi$ and of ${R}_{E, f}$, $\langle x, \psi(x), \dotsc, {\psi}^{f(l)-1}(x) \rangle$ is a cycle in $X$.
\end{proof}
As remarked after Definition \ref{def:H}, this implies that $\bdicr\left({\GGG}_{0}(E, f)\right) = {2}^{{\aleph}_{0}}$, for every dense selector $E$ and $f \in \FFF$.
\begin{Definition} \label{def:minimalcycle}
 Let $\pr{X}{R}$ be a digraph and $Y \subseteq X$.
 A cycle $\langle {x}_{0}, \dotsc, {x}_{n} \rangle$ in $Y$ is said to be \emph{minimal} if
 \begin{align*}
  \pr{{x}_{i}}{{x}_{j}} \in R \iff j = i + 1 \mod n+1,
 \end{align*}
 for all $i, j < n+1$.
 Observe that if $\langle {x}_{0}, \dotsc, {x}_{n} \rangle$ is a minimal cycle, then ${x}_{i} \neq {x}_{j}$, for all $i \neq j$.
\end{Definition}
Observe that a minimal cycle of length $n$ is just an isomorphic copy of the directed cycle ${C}_{n}$ of length $n$ -- in other words, the induced digraph on $\{{x}_{0}, \dotsc, {x}_{n}\}$ is isomorphic to ${C}_{n}$.
We will be interested in homomorphisms of digraphs that preserve copies of ${C}_{n}$.
\begin{Definition} \label{def:homandminhom}
 Let $\GGG = \pr{X}{R}$ and $\HHH = \pr{Y}{E}$ be digraphs.
 We will say that \emph{$\Pi$ is a homomorphism from $\GGG$ to $\HHH$} if $\Pi: X \rightarrow Y$ is a function such that $\forall \pr{x}{y} \in R\[\pr{\Pi(x)}{\Pi(y)} \in E\]$.
 
 We will say that \emph{$\Pi$ is a minimal homomorphism from $\GGG$ to $\HHH$} if $\Pi: X \rightarrow Y$ is a function such that the following hold:
 \begin{enumerate}
  \item
  $\forall \pr{x}{y} \in R\[\pr{\Pi(x)}{\Pi(y)} \in E\]$;
  \item
  for every minimal $R$-cycle $\langle {z}_{0}, \dotsc, {z}_{m} \rangle$ in $X$ and every $i, j < m+1$ such that $\pr{{z}_{i}}{{z}_{j}} \notin R$, $\pr{\Pi({z}_{i})}{\Pi({z}_{j})} \notin E$.
 \end{enumerate}
\end{Definition}
A minimal homomorphism is merely a homomorphism which maps every copy of ${C}_{n}$ to a copy of ${C}_{n}$.
Such homomorphisms are of interest in the theory of digraphs, for example, see \cite{MR3488937, MR3567535}.

The proofs of our dichotomies are based on the same ideas as the proof of the ${\GGG}_{0}$-dichotomy in Miller~\cite{MR3053069} and in Kechris, Solecki, and Todorcevic~\cite{MR1667145}, but our notation and presentation will follow Bernshteyn~\cite{bern}.
\begin{nota} \label{nota:homo}
 Fix the following notation from now until the end of the proof of Lemma \ref{lem:diamter}.
 Let $\pr{X}{R}$ be a fixed Borel digraph.
 Then both $R$ and $\left( X \times X \right) \setminus R$ are ${\mathbf{\Sigma}}^{1}_{1}$.
 Fix Polish spaces $M$ and $N$ as well as continuous functions $p: M \rightarrow X \times X$ and $q: N \rightarrow X \times X$ with $p''M = R$ and $q''N = \left(X \times X\right) \setminus R$.
 Fix also a selector $E$.
 These parameters will be suppressed for as long as they remain fixed.
 In particular, we will write ${R}_{\sigma}$ instead of ${R}_{E, \sigma}$, for $\sigma \in {\FFF}_{< \omega}$.
\end{nota}
\begin{Definition} \label{def:homo}
 For any $\sigma \in {\FFF}_{< \omega}$, ${S}_{\sigma}$ will denote the Polish space
 \begin{align*}
  {X}^{{T}_{\[\sigma\]}} \times {M}^{\left( {T}_{\[\sigma\]} \times {T}_{\[\sigma\]} \right)} \times {N}^{\left( {T}_{\[\sigma\]} \times {T}_{\[\sigma\]} \right)}.
 \end{align*}
 We will say that $\pr{\pi}{\varphi, \psi} \in {S}_{\sigma}$ is a \emph{minimal $S$-homomorphism} if the following hold:
 \begin{enumerate}
  \item
  $\forall \pr{s}{t} \in {R}_{\sigma}\[p(\varphi(\pr{s}{t})) = \pr{\pi(s)}{\pi(t)}\]$;
  \item
  for every minimal cycle $\langle {s}_{0}, \dotsc, {s}_{m} \rangle$ in ${T}_{\[\sigma\]}$ and for every $i, j < m+1$ such that $i \neq j$ and $\pr{{s}_{i}}{{s}_{j}} \notin {R}_{\sigma}$, $q(\psi(\pr{{s}_{i}}{{s}_{j}})) = \pr{\pi({s}_{i})}{\pi({s}_{j})}$.
 \end{enumerate}
 Define $\homo(\sigma) = \left\{ \pr{\pi}{\varphi, \psi} \in {S}_{\sigma}: \pr{\pi}{\varphi, \psi} \ \text{is a minimal $S$-homomorphism} \right\}$.
 
 Note that $\homo(\sigma)$ is a closed subset of the Polish space ${S}_{\sigma}$.
\end{Definition}
\begin{Definition} \label{def:HsHst}
 For $\sigma \in {\FFF}_{< \omega}$, $H \subseteq \homo(\sigma)$, and $s \in {T}_{\[\sigma\]}$, $H(s)$ denotes the set $\{\pi(s): \exists \varphi \exists \psi\[\pr{\pi}{\varphi, \psi} \in H\]\} \subseteq X$.
 For $\pr{s}{t} \in {T}_{\[\sigma\]} \times {T}_{\[\sigma\]}$, ${H}_{M}(s, t)$ denotes the set $\{\varphi(\pr{s}{t}): \exists \pi \exists \psi \[\pr{\pi}{\varphi, \psi} \in H\]\} \subseteq M$, and ${H}_{N}(s, t)$ denotes the set $\{\psi(\pr{s}{t}): \exists \pi \exists \varphi \[\pr{\pi}{\varphi, \psi} \in H\]\} \subseteq N$.
 
 Observe that if $H$ is a Borel subset of $\homo(\sigma)$, then $H(s)$, ${H}_{M}(s, t)$, and ${H}_{N}(s, t)$ are all ${\mathbf{\Sigma}}^{1}_{1}$.
\end{Definition}
\begin{Definition} \label{def:projections}
 For $\sigma \in {\FFF}_{< \omega}$, $2 \leq n < \omega$, and $i < n$, define ${r}_{n, i, \sigma}: {S}_{\next{\sigma}{n}} \rightarrow {S}_{\sigma}$ as follows.
 Given $\pr{\pi}{\varphi, \psi} \in {S}_{\next{\sigma}{n}}$, ${r}_{n, i, \sigma}(\pr{\pi}{\varphi, \psi}) = \pr{\pi'}{\varphi', \psi'}$, where for each $s \in {T}_{\[\sigma\]}$, $\pi'(s) = \pi(\next{s}{i})$, and for each $\pr{s}{t} \in {T}_{\[\sigma\]} \times {T}_{\[\sigma\]}$, $\varphi'(\pr{s}{t}) = \varphi(\pr{\next{s}{i}}{\next{t}{i}})$, and $\psi'(\pr{s}{t}) = \psi(\pr{\next{s}{i}}{\next{t}{i}})$.
 
 It is clear that each ${r}_{n, i, \sigma}$ is continuous.
\end{Definition}
\begin{Definition} \label{def:Hotimesn}
 Let $\sigma \in {\FFF}_{< \omega}$, $2 \leq n < \omega$, and $H \subseteq \homo(\sigma)$.
 Define
 \begin{align*}
  {H}^{\otimes n} = \left\{ \pr{\pi}{\varphi, \psi} \in \homo(\next{\sigma}{n}): \forall i < n \[{r}_{n, i, \sigma}(\pr{\pi}{\varphi, \psi}) \in H\] \right\}.
 \end{align*}
 If $H$ is Borel, then so is ${H}^{\otimes n}$ because ${r}_{n, i, \sigma}$ is continuous, for each $i< n$.
\end{Definition}
\begin{Definition} \label{def:tiny}
 Let $\sigma \in {\FFF}_{< \omega}$.
 $H \subseteq \homo(\sigma)$ is called \emph{tiny} if there exists $s \in {T}_{\[\sigma\]}$ such that $H(s)$ does not contain any cycles.
 $H \subseteq \homo(\sigma)$ is called \emph{small} if $H$ can be covered by countably many tiny Borel subsets of $\homo(\sigma)$.
 $H \subseteq \homo(\sigma)$ is called \emph{large} if $H$ is not small.
 
 Note that a union of countably many small subsets of $\homo(\sigma)$ is small.
 Note also that if $H = \emptyset$, then $H$ is tiny.
\end{Definition}
\begin{Lemma} \label{lem:0islarge}
 If $\wbdicr\left(\pr{X}{R}\right) > {\aleph}_{0}$, then $\homo(\emptyset)$ is large.
\end{Lemma}
\begin{proof}
 The hypothesis implies that $R \neq \emptyset$.
 Fix some $\pr{{x}^{\ast}}{{y}^{\ast}} \in R$ as well as $\mm \in M$ and $\nn \in N$ with $p(\mm) = \pr{{x}^{\ast}}{{y}^{\ast}}$ and $q(\nn) = \pr{{x}^{\ast}}{{x}^{\ast}}$.
 Note that ${T}_{\[\emptyset\]} = \{\emptyset\}$ and ${T}_{\[\emptyset\]} \times {T}_{\[\emptyset\]} = \{\pr{\emptyset}{\emptyset}\}$.
 Thus, for any $x \in X$, if we define $\pr{\pi}{\varphi, \psi} \in {S}_{\emptyset}$ by $\pi(\emptyset) = x$, $\varphi(\pr{\emptyset}{\emptyset}) = \mm$, and $\psi(\pr{\emptyset}{\emptyset}) = \nn$, then $\pr{\pi}{\varphi, \psi} \in \homo(\emptyset)$.
 Now assume for a contradiction that there are Borel tiny subsets $\{{H}_{n}: n \in \omega\}$ of $\homo(\emptyset)$ with $\homo(\emptyset) = {\bigcup}_{n \in \omega}{{H}_{n}}$.
 As ${H}_{n}$ is tiny, ${H}_{n}(\emptyset)$ is a ${\mathbf{\Sigma}}^{1}_{1}$ subset of $X$ that does not contain any cycles.
 By Corollary \ref{cor:reflection}, find a Borel set ${H}_{n}(\emptyset) \subseteq {B}_{n} \subseteq X$ such that ${B}_{n}$ does not contain any cycles.
 For any $x \in X$, there are $n \in \omega$ and $\pr{\pi}{\varphi, \psi} \in {H}_{n}$ with $\pi(\emptyset) = x$, whence $x \in {H}_{n}(\emptyset) \subseteq {B}_{n}$.
 Therefore $X = {\bigcup}_{n \in \omega}{{B}_{n}}$, contradicting $\wbdicr\left(\pr{X}{R}\right) > {\aleph}_{0}$.
\end{proof}
\begin{Lemma} \label{lem:Hotimesnnon0}
 Let $\sigma \in {\FFF}_{< \omega}$ and write $s = E(\sigma) \in {T}_{\[\sigma\]}$.
 Suppose $H \subseteq \homo(\sigma)$ and $\langle {x}_{0}, \dotsc, {x}_{m} \rangle$ is a minimal cycle in $H(s)$.
 Then there exists $\pr{\pi}{\varphi, \psi} \in {H}^{\otimes \left(m+1\right)}$ such that $\pi(\next{s}{i}) = {x}_{i}$, for each $i < m+1$.
\end{Lemma}
\begin{proof}
 Write $n = m+1$ and note that since there are no cycles of length $0$, $m \geq 1$, and so $n \geq 2$.
 For each $i < n$, choose $\pr{{\pi}_{i}}{{\varphi}_{i}, {\psi}_{i}} \in H$ with ${\pi}_{i}(s) = {x}_{i}$.
 For each $i, j < n$ such that $i \neq j$, if $j = i+1 \mod n$, then choose ${\mm}_{i} \in M$ with $p({\mm}_{i}) = \pr{{x}_{i}}{{x}_{j}}$, while if $j \neq i+1 \mod n$, then choose ${\nn}_{i, j} \in N$ with $q({\nn}_{i, j}) = \pr{{x}_{i}}{{x}_{j}}$.
 Define $\pi: {T}_{\[\next{\sigma}{n}\]} \rightarrow X$ by setting $\pi(\next{u}{i}) = {\pi}_{i}(u)$, for each $u \in {T}_{\[\sigma\]}$ and $i < n$.
 Define $\varphi: {T}_{\[\next{\sigma}{n}\]} \times {T}_{\[\next{\sigma}{n}\]} \rightarrow M$ and $\psi: {T}_{\[\next{\sigma}{n}\]} \times {T}_{\[\next{\sigma}{n}\]} \rightarrow N$ as follows.
 Suppose $u, t \in {T}_{\[\sigma\]}$ and $i, j < n$ are given.
 Unless $u=t=s$ and $i \neq j$, set $\varphi(\pr{\next{u}{i}}{\next{t}{j}}) = {\varphi}_{i}(\pr{u}{t})$ and $\psi(\pr{\next{u}{i}}{\next{t}{j}}) = {\psi}_{i}(\pr{u}{t})$.
 If $u=t=s$ and $j = i+1 \mod n$, then set $\varphi(\pr{\next{u}{i}}{\next{t}{j}}) = {\mm}_{i}$ and $\psi(\pr{\next{u}{i}}{\next{t}{j}}) = {\psi}_{i}(\pr{u}{t})$.
 If $u=t=s$, $i \neq j$, and $j \neq i+1 \mod n$, then set $\psi(\pr{\next{u}{i}}{\next{t}{j}}) = {\nn}_{i, j}$ and $\varphi(\pr{\next{u}{i}}{\next{t}{j}}) = {\varphi}_{i}(\pr{u}{t})$.
 This concludes the definition of $\pr{\pi}{\varphi, \psi}$.
 Note that by definition, ${r}_{n, i, \sigma}(\pr{\pi}{\varphi, \psi}) = \pr{{\pi}_{i}}{{\varphi}_{i}, {\psi}_{i}} \in H$, for each $i < n$.
 Also, for each $i < n$, $\pi(\next{s}{i}) = {\pi}_{i}(s) = {x}_{i}$.
 Hence we only need to verify that $\pr{\pi}{\varphi, \psi} \in \homo(\next{\sigma}{n})$.
 Suppose $u, t \in {T}_{\[\sigma\]}$, $i, j < n$, and $\pr{\next{u}{i}}{\next{t}{j}} \in {R}_{\next{\sigma}{n}}$.
 There are $2$ cases to consider.
 The first possibility is that $u=t=s$ and $j = i+1 \mod n$.
 Then $p(\varphi(\pr{\next{u}{i}}{\next{t}{j}})) = p({\mm}_{i}) = \pr{{x}_{i}}{{x}_{j}} = \pr{\pi(\next{u}{i})}{\pi(\next{t}{j})}$, as required.
 The other possibility is that $\pr{u}{t} \in {R}_{\sigma}$ and $i=j$.
 Then $p(\varphi(\pr{\next{u}{i}}{\next{t}{j}})) = p({\varphi}_{i}(\pr{u}{t})) = \pr{{\pi}_{i}(u)}{{\pi}_{i}(t)} = \pr{{\pi}_{i}(u)}{{\pi}_{j}(t)} = \pr{\pi(\next{u}{i})}{\pi(\next{t}{j})}$, as needed.
 Next consider some minimal cycle $\langle {u}_{0}, \dotsc, {u}_{k} \rangle$ in ${T}_{\[\next{\sigma}{n}\]}$, and fix $i, j < k+1$ such that $i \neq j$ and $\pr{{u}_{i}}{{u}_{j}} \notin {R}_{\next{\sigma}{n}}$.
 Again there are $2$ possibilities to consider.
 The first possibility is that there exist $i', j' < n$ such that $i' \neq j'$, $j' \neq i'+1 \mod n$, and ${u}_{i} = \next{s}{i'}$ and ${u}_{j} = \next{s}{j'}$.
 Then $q(\psi(\pr{{u}_{i}}{{u}_{j}})) = q(\psi(\pr{\next{s}{i'}}{\next{s}{j'}})) = q({\nn}_{i', j'}) = \pr{{x}_{i'}}{{x}_{j'}} = \pr{\pi(\next{s}{i'})}{\pi(\next{s}{j'})} = \pr{\pi({u}_{i})}{\pi({u}_{j})}$, as required.
 The other possibility is that for some minimal cycle $\langle {t}_{0}, \dotsc, {t}_{k} \rangle$ in ${T}_{\[\sigma\]}$ and for some $i' < n$, ${u}_{l} = \next{{t}_{l}}{i'}$, for every $l < k+1$.
 In this case, $\pr{{t}_{i}}{{t}_{j}} \notin {R}_{\sigma}$ and since $\pr{{\pi}_{i'}}{{\varphi}_{i'}, {\psi}_{i'}} \in \homo(\sigma)$, $q(\psi(\pr{{u}_{i}}{{u}_{j}})) = q(\psi(\pr{\next{{t}_{i}}{i'}}{\next{{t}_{j}}{i'}})) = q({\psi}_{i'}(\pr{{t}_{i}}{{t}_{j}})) = \pr{{\pi}_{i'}({t}_{i})}{{\pi}_{i'}({t}_{j})} = \pr{\pi(\next{{t}_{i}}{i'})}{\pi(\next{{t}_{j}}{i'})} = \pr{\pi({u}_{i})}{\pi({u}_{j})}$, as required.
 This verifies that $\pr{\pi}{\varphi, \psi} \in \homo(\next{\sigma}{n})$ and concludes the proof.
\end{proof}
\begin{Lemma} \label{lem:various}
 Let $\sigma \in {\FFF}_{< \omega}$, $2 \leq n < \omega$, and $H \subseteq \homo(\sigma)$.
 The following hold:
 \begin{enumerate}
  \item
  for each $i < n$, for each $s \in {T}_{\[\sigma\]}$, ${H}^{\otimes n}(\next{s}{i}) \subseteq H(s)$, and for each $\pr{s}{t} \in {T}_{\[\sigma\]} \times {T}_{\[\sigma\]}$, ${H}^{\otimes n}_{M}(\next{s}{i}, \next{t}{i}) \subseteq {H}_{M}(s, t)$ and ${H}^{\otimes n}_{N}(\next{s}{i}, \next{t}{i}) \subseteq {H}_{N}(s, t)$;
  \item
  if $H$ is tiny, then for every $i < n$,
  \begin{align*}
   {H}^{\otimes \pr{i}{n}} = \left\{ \pr{\pi}{\varphi, \psi} \in \homo(\next{\sigma}{n}): {r}_{n, i, \sigma}(\pr{\pi}{\varphi, \psi}) \in H \right\}
  \end{align*}
  is tiny;
  \item
  if $H$ is small, then so is ${H}^{\otimes \pr{i}{n}}$, for every $i < n$;
  \item
  if $G \subseteq H$, $H \setminus G$ is small, and ${G}^{\otimes n}$ is small, then so is ${H}^{\otimes n}$.
 \end{enumerate}
\end{Lemma}
\begin{proof}
 For (1): for $\pr{\pi}{\varphi, \psi} \in {H}^{\otimes n}$ and $i < n$, let $\pr{{\pi}_{i}}{{\varphi}_{i}, {\psi}_{i}} = {r}_{n, i, \sigma}(\pr{\pi}{\varphi, \psi}) \in H$.
 Then for all $s \in {T}_{\[\sigma\]}$, $\pi(\next{s}{i}) = {\pi}_{i}(s) \in H(s)$.
 Similarly, for all $\pr{s}{t} \in {T}_{\[\sigma\]} \times {T}_{\[\sigma\]}$, $\varphi(\pr{\next{s}{i}}{\next{t}{i}}) = {\varphi}_{i}(\pr{s}{t}) \in {H}_{M}(s, t)$ and $\psi(\pr{\next{s}{i}}{\next{t}{i}}) = {\psi}_{i}(\pr{s}{t}) \in {H}_{N}(s, t)$.
 (1) follows from this.
 
 For (2): fix $s \in {T}_{\[\sigma\]}$ so that $H(s)$ does not contain any cycles.
 Then $\next{s}{i} \in {T}_{\[\next{\sigma}{n}\]}$, and if $x \in {H}^{\otimes \pr{i}{n}}(\next{s}{i})$, then for some $\pr{\pi}{\varphi, \psi} \in {H}^{\otimes \pr{i}{n}}$, $x = \pi(\next{s}{i})$.
 Since $\pr{\pi'}{\varphi', \psi'} = {r}_{n, i, \sigma}(\pr{\pi}{\varphi, \psi}) \in H$, then $x = \pi(\next{s}{i}) = \pi'(s) \in H(s)$.
 Therefore, ${H}^{\otimes \pr{i}{n}}(\next{s}{i}) \subseteq H(s)$, and so ${H}^{\otimes \pr{i}{n}}(\next{s}{i})$ does not contain any cycles.
 This shows ${H}^{\otimes \pr{i}{n}}$ is tiny.
 
 For (3): let $\{{G}_{l}: l \in \omega\}$ be Borel tiny subsets of $\homo(\sigma)$ so that $H \subseteq {\bigcup}_{l \in \omega}{{G}_{l}}$.
 Fix $i < n$.
 By (2), ${G}^{\otimes \pr{i}{n}}_{l} \subseteq \homo(\next{\sigma}{n})$ is tiny.
 Since ${r}_{n, i, \sigma}$ is continuous and ${G}_{l}$ is Borel, ${G}^{\otimes \pr{i}{n}}_{l}$ is also Borel.
 If $\pr{\pi}{\varphi, \psi} \in {H}^{\otimes \pr{i}{n}}$, then $\pr{\pi}{\varphi, \psi} \in \homo(\next{\sigma}{n})$ and ${r}_{n, i, \sigma}(\pr{\pi}{\varphi, \psi}) \in H$, whence for some $l \in \omega$, ${r}_{n, i, \sigma}(\pr{\pi}{\varphi, \psi}) \in {G}_{l}$ and $\pr{\pi}{\varphi, \psi} \in {G}^{\otimes \pr{i}{n}}_{l}$.
 Therefore, ${H}^{\otimes \pr{i}{n}} \subseteq {\bigcup}_{l \in \omega}{{G}^{\otimes \pr{i}{n}}_{l}}$, and so ${H}^{\otimes \pr{i}{n}}$ is small.
 
 For (4): write $S = H \setminus G$.
 By (3), ${S}^{\otimes \pr{i}{n}}$ is small for every $i < n$.
 Therefore ${G}^{\otimes n} \cup \left( {\bigcup}_{i < n}{{S}^{\otimes \pr{i}{n}}} \right)$ is a small subset of $\homo(\next{\sigma}{n})$.
 Suppose $\pr{\pi}{\varphi, \psi} \in {H}^{\otimes n}$.
 Then $\pr{\pi}{\varphi, \psi} \in \homo(\next{\sigma}{n})$ and for each $i < n$, ${r}_{n, i, \sigma}(\pr{\pi}{\varphi, \psi}) \in H$.
 If for some $i < n$, ${r}_{n, i, \sigma}(\pr{\pi}{\varphi, \psi}) \in S$, then $\pr{\pi}{\varphi, \psi} \in {S}^{\otimes \pr{i}{n}}$.
 Otherwise $\pr{\pi}{\varphi, \psi} \in {G}^{\otimes n}$.
 Therefore, ${H}^{\otimes n} \subseteq {G}^{\otimes n} \cup \left( {\bigcup}_{i < n}{{S}^{\otimes \pr{i}{n}}} \right)$.
 So ${H}^{\otimes n}$, being a subset of a small set, is small.
\end{proof}
\begin{Lemma} \label{lem:removebad}
 Let $\sigma \in {\FFF}_{< \omega}$ and write $s = E(\sigma) \in {T}_{\[\sigma\]}$.
 Suppose $H \subseteq \homo(\sigma)$ is Borel.
 Let $X \subseteq \omega \setminus 2$ be so that for every $n \in X$, ${H}^{\otimes n}$ is small.
 Then there exists $G \subseteq H$ such that $G$ is Borel, $H \setminus G$ is small, and for every $n \in X$, $G(s)$ does not contain a minimal cycle of length $n-1$.
\end{Lemma}
\begin{proof}
 For each $n \in X$, we have that ${H}^{\otimes n} \subseteq {\bigcup}_{l < \omega}{{K}_{n, l}}$, where each ${K}_{n, l}$ is a Borel tiny subset of $\homo(\next{\sigma}{n})$.
 For each $n \in X$ and $l < \omega$, choose ${s}_{n, l} \in {T}_{\[\sigma\]}$ and ${i}_{n, l} < n$ so that ${K}_{n, l}(\next{{s}_{n, l}}{{i}_{n, l}})$ does not contain any cycles.
 Recall that since ${K}_{n, l}$ is Borel, ${K}_{n, l}(\next{{s}_{n, l}}{{i}_{n, l}})$ is a ${\mathbf{\Sigma}}^{1}_{1}$ subset of $X$.
 So by Corollary \ref{cor:reflection}, there exists a Borel set ${K}_{n, l}(\next{{s}_{n, l}}{{i}_{n, l}}) \subseteq {B}_{n, l} \subseteq X$ such that ${B}_{n, l}$ does not contain any cycles.
 Define ${T}_{n, l} = \{\pr{\pi'}{\varphi', \psi'} \in \homo(\sigma): \pi'({s}_{n, l}) \in {B}_{n, l}\}$.
 Since ${B}_{n, l}$ is Borel, ${T}_{n, l}$ is Borel, and since ${T}_{n, l}({s}_{n, l}) \subseteq {B}_{n, l}$, ${T}_{n, l}$ is tiny.
 Therefore, $S = \bigcup\{{T}_{n, l}: n \in X \wedge l < \omega\}$ is Borel and small.
 Define $G = H \setminus S$.
 Then $G \subseteq H$, $G$ is Borel, and $H \setminus G \subseteq S$, so $H \setminus G$ is small.
 Suppose for a contradiction that for some $n \in X$, $G(s)$ contains a minimal cycle of length $n-1$.
 Using Lemma \ref{lem:Hotimesnnon0}, find $\pr{\pi}{\varphi, \psi} \in {G}^{\otimes n} \subseteq {H}^{\otimes n}$.
 Fix $l < \omega$ with $\pr{\pi}{\varphi, \psi} \in {K}_{n, l}$, and note that $\pr{\pi'}{\varphi', \psi'} = {r}_{n, {i}_{n, l}, \sigma}(\pr{\pi}{\varphi, \psi}) \in G$.
 However, $\pi'({s}_{n, l}) = \pi(\next{{s}_{n, l}}{{i}_{n, l}}) \in {K}_{n, l}(\next{{s}_{n, l}}{{i}_{n, l}}) \subseteq {B}_{n, l}$.
 This implies $\pr{\pi'}{\varphi', \psi'} \in {T}_{n, l} \subseteq S$, contradicting $\pr{\pi'}{\varphi', \psi'} \in G$.
\end{proof}
\begin{Corollary} \label{cor:bigotimes}
 Let $\sigma \in {\FFF}_{< \omega}$ and let $H \subseteq \homo(\sigma)$ be Borel and large.
 Then there exists $2 \leq n < \omega$ such that ${H}^{\otimes n}$ is large.
 Furthermore, if $s = E(\sigma)$ and $2 \leq n < \omega$ is minimal such that ${H}^{\otimes n}$ is large, there exists $G \subseteq H$ such that $G$ is Borel, $H \setminus G$ is small, ${G}^{\otimes n}$ is large, and $G(s)$ does not contain any cycles of length less than $n-1$.
\end{Corollary}
\begin{proof}
 Let $X = \{2 \leq n < \omega: {H}^{\otimes n} \ \text{is small}\}$.
 Apply Lemma \ref{lem:removebad} to find $G \subseteq H$ such that $G$ is Borel, $H \setminus G$ is small, and $G(s)$ does not contain a minimal cycle of length $n-1$, for any $n \in X$.
 $G$ cannot be tiny because $H$ is large and $H \setminus G$ is small.
 So $G(s)$ contains a cycle, and hence, a minimal cycle of some length $1 \leq m < \omega$.
 Hence $2 \leq n = m+1 < \omega$ and $n \notin X$, whence ${H}^{\otimes n}$ is large.
 
 Suppose $2 \leq n < \omega$ is minimal so that ${H}^{\otimes n}$ is large.
 If ${G}^{\otimes n}$ were small, then by (4) of Lemma \ref{lem:various}, ${H}^{\otimes n}$ would also be small, contradicting the choice of $n$.
 So ${G}^{\otimes n}$ is large.
 Finally, suppose $G(s)$ contains a cycle of some length less than $n-1$.
 Then it contains a minimal cycle of some length $1 \leq m < n-1$.
 Hence $2 \leq l = m+1 < n$ and $l \notin X$, whence ${H}^{\otimes l}$ is large.
 But this is a contradiction as $l < n$.
\end{proof}
\begin{Lemma} \label{lem:diamter}
 Suppose $d$, ${d}_{1}$, and ${d}_{2}$ are compatible complete metrics on $X$, $M$, and $N$ respectively.
 Suppose $\sigma \in {\FFF}_{< \omega}$ and $H \subseteq \homo(\sigma)$ is Borel and large.
 Then for every $\varepsilon > 0$, there exists $G \subseteq H$ so that $G$ is Borel and large, for every $t \in {T}_{\[\sigma\]}$, ${\di}_{d}(\overline{G(t)}) < \varepsilon$, and for every $\pr{u}{t} \in {T}_{\[\sigma\]} \times {T}_{\[\sigma\]}$, ${\di}_{{d}_{1}}(\overline{{G}_{M}(u, t)}) < \varepsilon$ and ${\di}_{{d}_{2}}(\overline{{G}_{N}(u, t)}) < \varepsilon$.
\end{Lemma}
\begin{proof}
 This easily follows from the fact the Polish spaces $X$, $M$, and $N$ can each be covered by closed sets of diameter $< \varepsilon$, the fact that the union of countably many small sets is small, and the fact that the functions $\pr{\pi}{\varphi, \psi} \mapsto \pi(t)$, $\pr{\pi}{\varphi, \psi} \mapsto \varphi(\pr{u}{t})$, and $\pr{\pi}{\varphi, \psi} \mapsto \psi(\pr{u}{t})$ are all continuous, for every $t \in {T}_{\[\sigma\]}$ and every $\pr{u}{t} \in {T}_{\[\sigma\]} \times {T}_{\[\sigma\]}$.
\end{proof}
\begin{Theorem} \label{thm:dichotomy1}
 Let $\pr{X}{R}$ be any Borel digraph.
 Exactly one of the following alternatives holds:
 \begin{enumerate}[series=dichotomy1]
  \item
  $\wbdicr(\pr{X}{R}) \leq {\aleph}_{0}$;
  \item
  for every selector $E$, there exists $f \in \FFF$ such that both of the following hold:
  \begin{enumerate}
   \item[(a)]
   $\forall n, m \in \omega\[E(f \restrict m) \subseteq E(f \restrict n) \implies f(m) \leq f(n)\]$;
   \item[(b)]
   there is a continuous minimal homomorphism $\Pi$ from ${\GGG}_{0}(E, f)$ to $\pr{X}{R}$.
  \end{enumerate}
 \end{enumerate}
\end{Theorem}
\begin{proof}
 Suppose $E$ is some dense selector and $\Pi: \[{T}_{f}\] \rightarrow X$ is a continuous homomorphism.
 If $X = {\bigcup}_{l \in \omega}{{X}_{l}}$, where each ${X}_{l}$ is Borel and contains no cycles, then $\[{T}_{f}\] = {\bigcup}_{l \in \omega}{{\Pi}^{-1}({X}_{l})}$, where each ${\Pi}^{-1}({X}_{l})$ is Borel and contains no cycles because $\Pi$ is a homomorphism.
 Since this would contradict Lemma \ref{lem:covM} and since dense selectors are easily seen to exist, it follows that (1) and (2) are mutually exclusive.
 
 Now assume that $\wbdicr(\pr{X}{R}) > {\aleph}_{0}$.
 Fix $M$, $N$, $p$, $q$ as in Notation \ref{nota:homo}.
 Fix also a selector $E$.
 As before, these parameters will be suppressed for the remainder of this proof.
 Also, fix compatible complete metrics $d$, ${d}_{1}$, and ${d}_{2}$ on $X$, $M$, and $N$ respectively.
 By Lemma \ref{lem:0islarge}, $\homo(\emptyset)$ is large.
 Construct by induction $f \in \FFF$ and $\seq{H}{n}{\in}{\omega}$ such that:
 \begin{enumerate} [resume=dichotomy1]
  \item
  ${H}_{n} \subseteq \homo({\sigma}_{n})$ is Borel and large, where ${\sigma}_{n} = f \restrict n$, and ${H}_{n}({s}_{n})$ does not contain any cycles of length less than $f(n)-1$, where ${s}_{n} = E({\sigma}_{n})$;
  \item
  ${H}^{\otimes f(n)}_{n}$ is large and ${H}_{n+1} \subseteq {H}^{\otimes f(n)}_{n}$;
  \item
  for each $t \in {T}_{\[{\sigma}_{n}\]}$, ${\di}_{d}(\overline{{H}_{n}(t)}) < {2}^{-n}$ and for each $\pr{u}{t} \in {T}_{\[{\sigma}_{n}\]} \times {T}_{\[{\sigma}_{n}\]}$, ${\di}_{{d}_{1}}(\overline{{({H}_{n})}_{M}(u, t)}) < {2}^{-n}$ and ${\di}_{{d}_{2}}(\overline{{({H}_{n})}_{N}(u, t)}) < {2}^{-n}$;
  \item
  for any $m < n$, $s \in {T}_{\[{\sigma}_{m}\]}$ and $x \in {\omega}^{< \omega}$ such that $\cat{s}{x} \in {T}_{\[{\sigma}_{n}\]}$,
  \begin{align*}
   {H}_{n}(\cat{s}{x}) \subseteq {H}_{m}(s),
  \end{align*}
  and for any $\pr{s}{t} \in {T}_{\[{\sigma}_{m}\]} \times {T}_{\[{\sigma}_{m}\]}$ and $x \in {\omega}^{< \omega}$ such that $\pr{\cat{s}{x}}{\cat{t}{x}} \in {T}_{\[{\sigma}_{n}\]} \times {T}_{\[{\sigma}_{n}\]}$,
  \begin{align*}
   &{\left({H}_{n}\right)}_{M}(\pr{\cat{s}{x}}{\cat{t}{x}}) \subseteq {\left({H}_{m}\right)}_{M}(\pr{s}{t}) \ \text{and}\\ &{\left({H}_{n}\right)}_{N}(\pr{\cat{s}{x}}{\cat{t}{x}}) \subseteq {\left({H}_{m}\right)}_{N}(\pr{s}{t}).
  \end{align*}
 \end{enumerate}
 It is easily seen that such $f$ and $\seq{H}{n}{\in}{\omega}$ can be inductively constructed starting from $\homo(\emptyset)$ and applying Lemma \ref{lem:diamter} and Corollary \ref{cor:bigotimes} at every stage.
 Item (6) is satisfied because of Item (1) of Lemma \ref{lem:various}.
 
 To see that (2)(a) holds, suppose ${s}_{m} \subseteq {s}_{n}$.
 If $m=n$, then there is nothing to show, so assume $m < n$ and let $x \in {\omega}^{< \omega}$ be such that ${s}_{n} = \cat{{s}_{m}}{x}$.
 As ${H}_{n+1} \neq \emptyset$, choose $\pr{\pi}{\varphi, \psi} \in {H}_{n+1}$.
 Then $\langle \pi(\next{{s}_{n}}{0}), \dotsc, \pi(\next{{s}_{n}}{f(n)-1}) \rangle$ is a cycle and each $\pi(\next{{s}_{n}}{i}) \in {H}_{n+1}(\next{{s}_{n}}{i}) \subseteq {H}_{n}({s}_{n}) \subseteq {H}_{m}({s}_{m})$.
 So there is a cycle of length $f(n)-1$ in ${H}_{m}({s}_{m})$, whence $f(m)-1 \leq f(n)-1$ and $f(m) \leq f(n)$, as claimed.
 
 For (2)(b), define $\Pi: \[{T}_{f}\] \rightarrow X$ as follows.
 For $z \in \[{T}_{f}\]$, by (5), (6), and the completeness of the metric $d$, there is a unique element in ${\bigcap}_{n \in \omega}{\overline{{H}_{n}(z\restrict n)}}$.
 Define $\Pi(z)$ to be this unique element of ${\bigcap}_{n \in \omega}{\overline{{H}_{n}(z\restrict n)}}$.
 It is clear that $\Pi$ is continuous.
 To see that $\Pi$ is a homomorphism, suppose $\pr{y}{z} \in {R}_{f}$.
 Then for some $m \in \omega$, $\pr{u}{t} \in {T}_{\[{\sigma}_{m}\]} \times {T}_{\[{\sigma}_{m}\]}$, and $x \in \BS$, we have that $y = \cat{u}{x}$, $z = \cat{t}{x}$, and $\forall n \in \omega\[\pr{\cat{u}{x \restrict n}}{\cat{t}{x \restrict n}} \in {R}_{{\sigma}_{m+n}}\]$.
 By (5), (6), and the completeness of the metric ${d}_{1}$, there exists $\mm \in M$ with
 \begin{align*}
  \{\mm\} = {\bigcap}_{n \in \omega}{\overline{{\left({H}_{m+n}\right)}_{M}(\cat{u}{x\restrict n}, \cat{t}{x \restrict n})}}.
 \end{align*}
 For each $n \in \omega$, choose ${\mm}_{n} \in {\left({H}_{m+n}\right)}_{M}(\cat{u}{x\restrict n}, \cat{t}{x \restrict n})$ with ${d}_{1}(\mm, {\mm}_{n}) < {2}^{-m-n}$ and then choose $\pr{\pi}{\varphi, \psi} \in {H}_{m+n}$ such that $\varphi(\pr{\cat{u}{x\restrict n}}{\cat{t}{x \restrict n}}) = {\mm}_{n}$.
 Observe that $p({\mm}_{n}) = \pr{\pi(\cat{u}{x\restrict n})}{\pi(\cat{t}{x \restrict n})}$.
 Observe also that $d(\pi(\cat{u}{x\restrict n}), \Pi(y)) < {2}^{-m-n}$ and $d(\pi(\cat{t}{x \restrict n}), \Pi(z)) < {2}^{-m-n}$ because $\pi(\cat{u}{x\restrict n}) \in {H}_{m+n}(\cat{u}{x\restrict n})$ and $\pi(\cat{t}{x \restrict n}) \in {H}_{m+n}(\cat{t}{x \restrict n})$.
 By the continuity of $p$, $p(\mm) = \pr{\Pi(y)}{\Pi(z)}$.
 And since $p''M = R$, $\pr{\Pi(y)}{\Pi(z)} \in R$ as needed.
 The verification that $\Pi$ is minimal is similar.
 Let $\langle {z}_{0}, \dotsc, {z}_{m} \rangle$ be a minimal cycle in $\[{T}_{f}\]$ and suppose $i, j < m+1$ such that $\pr{{z}_{i}}{{z}_{j}} \notin {R}_{f}$.
 We may assume $i \neq j$, for otherwise $\pr{\Pi({z}_{i})}{\Pi({z}_{i})} \notin R$ by our definition of a digraph.
 Then for some $l \in \omega$, $\langle {u}_{0}, \dotsc, {u}_{m} \rangle$ in ${T}_{\[{\sigma}_{l}\]}$, and $x \in \BS$, we have that $\forall 0 \leq k \leq m\[{z}_{k} = \cat{{u}_{k}}{x}\]$, and for each $n \in \omega$, $\langle \cat{{u}_{0}}{x \restrict n}, \dotsc, \cat{{u}_{m}}{x \restrict n} \rangle$ is a minimal cycle in ${T}_{\[{\sigma}_{l+n}\]}$ and $\pr{\cat{{u}_{i}}{x \restrict n}}{\cat{{u}_{j}}{x \restrict n}} \notin {R}_{{\sigma}_{l+n}}$.
 By (5), (6), and the completeness of the metric ${d}_{2}$, there exists $\nn \in N$ with
 \begin{align*}
  \{\nn\} = {\bigcap}_{n \in \omega}{\overline{{\left({H}_{l+n}\right)}_{N}(\cat{{u}_{i}}{x \restrict n}, \cat{{u}_{j}}{x \restrict n})}}.
 \end{align*}
 For each $n \in \omega$, choose ${\nn}_{n} \in {\left({H}_{l+n}\right)}_{N}(\cat{{u}_{i}}{x \restrict n}, \cat{{u}_{j}}{x \restrict n})$ with ${d}_{2}(\nn, {\nn}_{n}) < {2}^{-l-n}$ and choose $\pr{\pi}{\varphi, \psi} \in {H}_{l+n}$ such that $\psi(\pr{\cat{{u}_{i}}{x \restrict n}}{\cat{{u}_{j}}{x \restrict n}}) = {\nn}_{n}$.
 By (2) of Definition \ref{def:homo}, $q({\nn}_{n}) = \pr{\pi(\cat{{u}_{i}}{x \restrict n})}{\pi(\cat{{u}_{j}}{x \restrict n})}$.
 Observe that $d(\pi(\cat{{u}_{i}}{x \restrict n}), \Pi({z}_{i})) < {2}^{-l-n}$ and $d(\pi(\cat{{u}_{j}}{x \restrict n}), \Pi({z}_{j})) < {2}^{-l-n}$ because $\pi(\cat{{u}_{i}}{x \restrict n}) \in {H}_{l+n}(\cat{{u}_{i}}{x \restrict n})$ and $\pi(\cat{{u}_{j}}{x \restrict n}) \in {H}_{l+n}(\cat{{u}_{j}}{x \restrict n})$.
 By the continuity of $q$, $q(\nn) = \pr{\Pi({z}_{i})}{\Pi({z}_{j})}$.
 And since $q'' N = \left( X \times X \right) \setminus R$, $\pr{\Pi({z}_{i})}{\Pi({z}_{j})} \notin R$ as needed.
 This verifies all of (2) and concludes the proof. 
\end{proof}
\begin{Definition} \label{def:uniform}
 A digraph $\pr{X}{R}$ is said to be \emph{$k$-uniform} if every minimal cycle in $X$ has length at most $k$.
\end{Definition}
\begin{Theorem} \label{thm:dichotomy2}
 Let $\pr{X}{R}$ be any analytic digraph.
 Exactly one of the following alternatives holds:
 \begin{enumerate}[series=dichotomy2]
  \item
  $\wbdicr(\pr{X}{R}) \leq {\aleph}_{0}$;
  \item
  for every selector $E$, there exists $f \in \FFF$ such that both of the following hold:
  \begin{enumerate}
   \item[(a)]
   $\forall n, m \in \omega\[E(f \restrict m) \subseteq E(f \restrict n) \implies f(m) \leq f(n)\]$;
   \item[(b)]
   there is a continuous homomorphism $\Pi$ from ${\GGG}_{0}(E, f)$ to $\pr{X}{R}$.
  \end{enumerate}
  Furthermore, if $\pr{X}{R}$ is $k$-uniform, then $f(n) \leq k+1$, for every $n \in \omega$.
 \end{enumerate}
\end{Theorem}
\begin{proof}
 Just as in the proof of Theorem \ref{thm:dichotomy1}, alternatives (1) and (2) are mutually exclusive by an application of Lemma \ref{lem:covM} and by the existence of dense selectors.
 
 Now assume that $\wbdicr(\pr{X}{R}) > {\aleph}_{0}$.
 Fix a selector $E$.
 Repeat the proof of Theorem \ref{thm:dichotomy1} without any reference to $N$, $q$, ${d}_{2}$, or the minimality of $S$-homomorphisms.
 More explicitly, let $M$ be a Polish space and let $p: M \rightarrow X \times X$ be a continuous function with $p'' M = R$.
 For any $\sigma \in {\FFF}_{< \omega}$, ${S}_{\sigma}$ denotes the Polish space ${X}^{{T}_{\[\sigma\]}} \times {M}^{\left({T}_{\[\sigma\]} \times {T}_{\[\sigma\]}\right)}$.
 A pair $\pr{\pi}{\varphi} \in {S}_{\sigma}$ is an \emph{$S$-homomorphism} if
 \begin{align*}
  \forall \pr{s}{t} \in {R}_{\sigma}\[p(\varphi(\pr{s}{t})) = \pr{\pi(s)}{\pi(t)}\].
 \end{align*}
 Define $\homo(\sigma) = \{\pr{\pi}{\varphi} \in {S}_{\sigma}: \pr{\pi}{\varphi} \ \text{is an $S$-homomorphism}\}$, and as before, $\homo(\sigma)$ is a closed subset of ${S}_{\sigma}$.
 For $\sigma \in {\FFF}_{< \omega}$, $H \subseteq \homo(\sigma)$, and $s \in {T}_{\[\sigma\]}$, define $H(s) = \{\pi(s): \exists \varphi\[\pr{\pi}{\varphi} \in H\]\} \subseteq X$.
 And for $\pr{s}{t} \in {T}_{\[\sigma\]} \times {T}_{\[\sigma\]}$, define ${H}_{M}(s, t) = \{\varphi(\pr{s}{t}): \exists \pi\[\pr{\pi}{\varphi} \in H\]\} \subseteq M$.
 If $H$ is a Borel subset of $\homo(\sigma)$, then $H(s)$ and ${H}_{M}(s, t)$ are both ${\mathbf{\Sigma}}^{1}_{1}$.
 For $\sigma \in {\FFF}_{< \omega}$, $2 \leq n < \omega$, and $i < n$, define ${r}_{n, i, \sigma}: {S}_{\next{\sigma}{n}} \rightarrow {S}_{\sigma}$ so that for every $\pr{\pi}{\varphi} \in {S}_{\next{\sigma}{n}}$, ${r}_{n, i, \sigma}(\pr{\pi}{\varphi}) = \pr{\pi'}{\varphi'}$, where for each $s \in {T}_{\[\sigma\]}$, $\pi'(s) = \pi(\next{s}{i})$, and for each $\pr{s}{t} \in {T}_{\[\sigma\]} \times {T}_{\[\sigma\]}$, $\varphi'(\pr{s}{t}) = \varphi(\pr{\next{s}{i}}{\next{t}{i}})$.
 Each ${r}_{n, i, \sigma}$ is continuous.
 For $\sigma \in {\FFF}_{< \omega}$, $2 \leq n < \omega$, and $H \subseteq \homo(\sigma)$, define
 \begin{align*}
  {H}^{\otimes n} = \left\{ \pr{\pi}{\varphi} \in \homo(\next{\sigma}{n}): \forall i < n\[{r}_{n, i, \sigma}(\pr{\pi}{\varphi}) \in H\] \right\}.
 \end{align*}
 If $H$ is Borel, then so is ${H}^{\otimes n}$.
 Now, the notions of \emph{tiny}, \emph{small}, and \emph{large} can be defined verbatim as in Definition \ref{def:tiny}.
 Using the same arguments as in the proofs of Lemma \ref{lem:0islarge} to Lemma \ref{lem:diamter}, the following can be proved:
 \begin{enumerate}[resume=dichotomy2]
  \item
  if $\wbdicr(\pr{X}{R}) > {\aleph}_{0}$, then $\homo(\emptyset)$ is large;
  \item
  if $\sigma \in {\FFF}_{< \omega}$, $s = E(\sigma)$, $H \subseteq \homo(\sigma)$, and $\langle {x}_{0}, \dotsc, {x}_{m} \rangle$ is any cycle in $H(s)$, then there exists $\pr{\pi}{\varphi} \in {H}^{\otimes (m+1)}$ such that $\pi(\next{s}{i}) = {x}_{i}$, for each $i < m+1$;
  \item
  if $\sigma \in {\FFF}_{< \omega}$, $s = E(\sigma)$, and $H \subseteq \homo(\sigma)$ is Borel and large, then there exist $2 \leq n < \omega$ and $G \subseteq H$ such that $G$ is Borel, $H \setminus G$ is small, ${G}^{\otimes n}$ is large, and $G(s)$ does not contain any cycles of length less than $n-1$;
  \item
  if $d$ and ${d}_{1}$ are compatible complete metrics on $X$ and $M$ respectively, $\sigma \in {\FFF}_{< \omega}$, and $H \subseteq \homo(\sigma)$ is Borel and large, then for every $\varepsilon > 0$, there exists $G \subseteq H$ so that $G$ is Borel and large, for every $t \in {T}_{\[\sigma\]}$, ${\di}_{d}(\overline{G(t)}) < \varepsilon$, and for every $\pr{u}{t} \in {T}_{\[\sigma\]} \times {T}_{\[\sigma\]}$, ${\di}_{{d}_{1}}(\overline{{G}_{M}(u, t)}) < \varepsilon$.
 \end{enumerate}
 Fixing compatible complete metrics $d$ and ${d}_{1}$ on $X$ and $M$ respectively, it is possible to use (3)--(6) to construct $f \in \FFF$ and $\seq{H}{n}{\in}{\omega}$ satisfying:
 \begin{enumerate}[resume=dichotomy2]
  \item
  ${H}_{n} \subseteq \homo({\sigma}_{n})$ is Borel and large, where ${\sigma}_{n} = f \restrict n$, and ${H}_{n}({s}_{n})$ does not contain any cycles of length less than $f(n)-1$, where ${s}_{n} = E({\sigma}_{n})$;
  \item
  ${H}^{\otimes f(n)}_{n}$ is large and ${H}_{n+1} \subseteq {H}^{\otimes f(n)}_{n}$;
  \item
  for each $t \in {T}_{\[{\sigma}_{n}\]}$, ${\di}_{d}(\overline{{H}_{n}(t)}) < {2}^{-n}$ and for each $\pr{u}{t} \in {T}_{\[{\sigma}_{n}\]} \times {T}_{\[{\sigma}_{n}\]}$, ${\di}_{{d}_{1}}(\overline{{\left({H}_{n}\right)}_{M}(u, t)}) < {2}^{-n}$;
  \item
  for any $m < n$, $s \in {T}_{\[{\sigma}_{m}\]}$ and $x \in {\omega}^{< \omega}$ such that $\cat{s}{x} \in {T}_{\[{\sigma}_{n}\]}$,
  \begin{align*}
   {H}_{n}(\cat{s}{x}) \subseteq {H}_{m}(s),
  \end{align*}
  and for any $\pr{s}{t} \in {T}_{\[{\sigma}_{m}\]} \times {T}_{\[{\sigma}_{m}\]}$ and $x \in {\omega}^{< \omega}$ such that $\pr{\cat{s}{x}}{\cat{t}{x}} \in {T}_{\[{\sigma}_{n}\]} \times {T}_{\[{\sigma}_{n}\]}$,
  \begin{align*}
   {\left( {H}_{n} \right)}_{M}(\pr{\cat{s}{x}}{\cat{t}{x}}) \subseteq {\left( {H}_{m} \right)}_{M}(\pr{s}{t}).
  \end{align*}
 \end{enumerate}
 Now, the verification of (2)(a) is identical to the corresponding argument in the proof of Theorem \ref{thm:dichotomy1}.
 Also, the map $\Pi: \[{T}_{f}\] \rightarrow X$ can be defined exactly as in the proof of Theorem \ref{thm:dichotomy1} and the verification of (2)(b) is similar to the verification of the first half of (2)(b) of Theorem \ref{thm:dichotomy1}.
 To see the final sentence of (2), assume that $\pr{X}{R}$ is $k$-uniform.
 For each $n \in \omega$, ${H}_{n+1}$ is a non-empty subset of ${H}^{\otimes f(n)}_{n}$, so consider $\pr{\pi}{\varphi} \in {H}_{n+1}$.
 Then $\langle \pi(\next{{s}_{n}}{0}), \dotsc, \pi(\next{{s}_{n}}{f(n)-1}) \rangle$ is a cycle of length $f(n)-1$ in ${H}_{n}({s}_{n})$.
 Since ${H}_{n}({s}_{n})$ does not contain any cycle of length less than $f(n)-1$, $\langle \pi(\next{{s}_{n}}{0}), \dotsc, \pi(\next{{s}_{n}}{f(n)-1}) \rangle$ must be a minimal cycle, and so by $k$-uniformity, $f(n)-1 \leq k$, as claimed.
\end{proof}
As noted after Definition \ref{def:H}, Item (1) of Theorems \ref{thm:dichotomy1} and \ref{thm:dichotomy2} is equivalent to the condition that $\bdicr\left(\pr{X}{R}\right) \leq {\aleph}_{0}$.
We do not know whether the homomorphism $\Pi$ can be chosen to be minimal in the case when $\pr{X}{R}$ is analytic, but not Borel.
\begin{Question} \label{q:1}
 Suppose $\pr{X}{R}$ is an analytic digraph with $\wbdicr\left(\pr{X}{R}\right) > {\aleph}_{0}$.
 Is it true that for every selector $E$, there exist $f \in \FFF$ and a continuous minimal homomorphism $\Pi$ from ${\GGG}_{0}(E, f)$ to $\pr{X}{R}$?
\end{Question}
Next, we will show that the ${\GGG}_{0}$ dichotomy of Kechris, Solecki, and Todorcevic is a consequence of Theorem \ref{thm:dichotomy2}.
\begin{Definition} \label{def:g0}
 $D \subseteq {2}^{< \omega}$ is said to be \emph{dense} if $\forall s \in {2}^{< \omega} \exists t \in D\[s \subseteq t\]$.
 
 For $D \subseteq {2}^{< \omega}$, ${\GGG}_{0}(D)$ is the graph $\pr{X}{R}$, where $X = {2}^{\omega}$ and
 \begin{align*}
  R = \left\{ \pr{\ins{s}{i}{x}}{\ins{s}{i+1 \mod 2}{x}}: s \in D \wedge i \in 2 \wedge x \in {2}^{\omega} \right\}.
 \end{align*}
\end{Definition}
\begin{Corollary}[${G}_{0}$-Dichotomy~\cite{MR1667145}] \label{cor:g0dichotomy}
 Let $\pr{X}{R}$ be any analytic graph.
 Exactly one of the following alternatives holds:
 \begin{enumerate}
  \item
  $\wbcr(\pr{X}{R}) \leq {\aleph}_{0}$;
  \item
  for every $D \subseteq {2}^{< \omega}$ with the property that $\forall n \in \omega\[\lc D \cap {2}^{n} \rc = 1\]$, there is a continuous homomorphism from ${\GGG}_{0}(D)$ to $\pr{X}{R}$.
 \end{enumerate}
\end{Corollary}
\begin{proof}
 The argument of Lemma \ref{lem:covM} shows that if $D \subseteq {2}^{< \omega}$ is dense and has the property that $\forall n \in \omega \[\lc D \cap {2}^{n} \rc = 1\]$, then $\wbcr({\GGG}_{0}(D)) \geq \cov(\MMM)$.
 It follows from this and from the existence of a dense $D \subseteq {2}^{< \omega}$ such that $\forall n \in \omega\[\lc D \cap {2}^{n} \rc = 1\]$ that the alternatives (1) and (2) are mutually exclusive.
 
 Assume that $\wbcr(\pr{X}{R}) > {\aleph}_{0}$.
 This implies $\wbdicr(\pr{X}{R}) > {\aleph}_{0}$.
 Fix any $D \subseteq {2}^{< \omega}$ so that $\forall n \in \omega\[\lc D \cap {2}^{n} \rc = 1\]$.
 Define a selector $E: {\FFF}_{< \omega} \rightarrow {\omega}^{< \omega}$ as follows.
 If $\sigma \in {\FFF}_{< \omega}$ is such that $\forall k \in \dom(\sigma)\[\sigma(k)=2\]$, then define $E(\sigma)$ to be the unique element of $D \cap {2}^{\dom(\sigma)}$.
 Otherwise, define $E(\sigma)$ to be an arbitrary element of ${T}_{\[\sigma\]}$.
 By Theorem \ref{thm:dichotomy2}, there exists $f \in \FFF$ and a continuous homomorphism from ${\GGG}_{0}(E, f)$ to $\pr{X}{R}$.
 Furthermore, since $\pr{X}{R}$ is $1$-uniform, $f(k) \leq 2$ and since $f \in \FFF$, $f(k) = 2$, for all $k \in \omega$.
 Therefore ${\GGG}_{0}(E, f) = {\GGG}_{0}(D)$ and the conclusion follows.
\end{proof}
Once again, as mentioned after Definition \ref{def:H}, Item (1) of Corollary \ref{cor:g0dichotomy} is equivalent to the assertion that $\bcr\left(\pr{X}{R}\right) \leq {\aleph}_{0}$.
Next, we associate a Borel quasi order with every Borel digraph as follows.
\begin{Definition} \label{def:PX}
 Let $\GGG = \pr{X}{R}$ be a digraph.
 Define ${P}_{\GGG} = 2 \times X$ and ${\leq}_{\GGG} = \{\pr{\pr{i}{x}}{\pr{j}{y}} \in {P}_{\GGG} \times {P}_{\GGG}: \left( \pr{i}{x} = \pr{j}{y} \right) \vee \left( i = 0 \wedge j = 1 \wedge \pr{x}{y} \in R \right) \}$.
 Define $\PPP(\GGG) = \pr{{P}_{\GGG}}{{\leq}_{\GGG}}$.
 It is clear that $\PPP(\GGG)$ is a quasi-order and that ${E}_{{\leq}_{\GGG}}$ is equality.
 Furthermore, if $\GGG$ is a Borel digraph, then $\PPP(\GGG)$ is a Borel quasi order.
\end{Definition}
\begin{Lemma} \label{lem:XinAPG}
 Suppose $\GGG = \pr{X}{R}$ is a Borel digraph.
 Then there is a Borel homomorphism from $\GGG$ to $\pr{{\AAA}_{\PPP(\GGG)}}{{\RRR}_{\PPP(\GGG)}}$.
 In particular, if $X \neq \emptyset$, then ${\odim}_{B}(\PPP(\GGG)) \geq \wbdicr(\GGG)$.
\end{Lemma}
\begin{proof}
 Define $f: X \rightarrow {\AAA}_{\PPP(\GGG)}$ by $f(x) = \pr{\pr{1}{x}}{\pr{0}{x}}$, for every $x \in X$.
 Note that $f(x) \in {\AAA}_{\PPP(\GGG)}$ because $\pr{x}{x} \notin R$ and that $f$ is a Borel map with respect to any Polish topology on ${\AAA}_{\PPP(\GGG)}$ which extends the topology on ${\AAA}_{\PPP(\GGG)}$ inherited from ${P}_{\GGG} \times {P}_{\GGG}$.
 Further, for any $x, y \in X$, $\pr{x}{y} \in R$ if and only if $\pr{0}{x} \; {\leq}_{\GGG} \; \pr{1}{y}$ if and only if $f(x) \; {\RRR}_{\PPP(\GGG)} \; f(y)$.
 It follows that $\wbdicr\left( \pr{{\AAA}_{\PPP(\GGG)}}{{\RRR}_{\PPP(\GGG)}} \right) \geq \wbdicr\left( \GGG \right)$.
 If $x \in X$, then $\pr{0}{x}, \pr{1}{x} \in {P}_{\GGG}$, and ${\[\pr{0}{x}\]}_{{E}_{{\leq}_{\GGG}}} \neq {\[\pr{1}{x}\]}_{{E}_{{\leq}_{\GGG}}}$.
 Therefore, if $X \neq \emptyset$, then $\lc {P}_{\GGG} \slash {E}_{{\leq}_{\GGG}} \rc > 1$, and so by Theorem \ref{thm:odimB=HB}, ${\odim}_{B}\left( \PPP(\GGG) \right) = \wbdicr\left( \pr{{\AAA}_{\PPP(\GGG)}}{{\RRR}_{\PPP(\GGG)}} \right) \geq \wbdicr\left( \GGG \right)$.
\end{proof}
\begin{Corollary} \label{cor:dimpg0ef}
 For every dense selector $E$ and for every $f \in \FFF$, we have ${\odim}_{B}\left( \PPP\left( {\GGG}_{0}(E, f) \right) \right) \geq \cov(\MMM)$.
\end{Corollary}
We are now ready to state and prove our third and final dichotomy.
The basic objects of this dichotomy are the Borel quasi orders of the form $\PPP\left({\GGG}_{0}(E, f)\right)$.
The morphisms will be continuous order preserving maps that also preserve the order incomparability of certain elements.
The dichotomy characterizes the Borel quasi orders of uncountable Borel order dimension in terms of the existence of such a morphism from one of the basic objects.
\begin{Theorem} \label{thm:borelodimdichotomy}
 For any Borel quasi order $\PPP = \pr{P}{\leq}$, exactly one of the following holds:
 \begin{enumerate}
  \item
  ${\odim}_{B}\left( \PPP \right) \leq {\aleph}_{0}$;
  \item
  for every selector $E$, there exist $f \in \FFF$ and a continuous $\varphi: 2 \times \[{T}_{f}\] \rightarrow P$ such that both of the following hold:
  \begin{enumerate}
   \item[(a)]
    $\forall x, y \in \[{T}_{f}\]\[\pr{0}{x} {\leq}_{{\GGG}_{0}(E, f)} \pr{1}{y} \implies \varphi(\pr{0}{x}) \leq \varphi(\pr{1}{y})\]$;
   \item[(b)]
   for every $x \in \[{T}_{f}\]$, $\varphi(\pr{0}{x})$ and $\varphi(\pr{1}{x})$ are $\leq$-incomparable.
  \end{enumerate}
 \end{enumerate}
\end{Theorem}
\begin{proof}
 Let us first see that (1) and (2) are mutually exclusive.
 Assume (2).
 Fix a dense selector $E$.
 Let $f \in \FFF$ and $\varphi: 2 \times \[{T}_{f}\] \rightarrow P$ be a continuous map satisfying (2)(a) and (2)(b).
 Define $\psi: \[{T}_{f}\] \rightarrow {\AAA}_{\PPP}$ by $\psi(x) = \pr{\varphi(\pr{1}{x})}{\varphi(\pr{0}{x})}$, for all $x \in \[{T}_{f}\]$.
 Note that $\psi(x) \in {\AAA}_{\PPP}$ because of (2)(b) and that $\psi$ is a Borel map with respect to any Polish topology on ${\AAA}_{\PPP}$ which extends the topology on ${\AAA}_{\PPP}$ inherited from $P \times P$.
 Further, if $\pr{x}{y} \in {R}_{E, f}$, then $\pr{0}{x} {\leq}_{{\GGG}_{0}(E, f)} \pr{1}{y}$, which by (2)(a) implies $\varphi(\pr{0}{x}) \leq \varphi(\pr{1}{y})$, which in turn implies $\pr{\psi(x)}{\psi(y)} \in {\RRR}_{\PPP}$.
 By (2)(b), $\lc P \slash {E}_{\leq} \rc > 1$.
 It follows that ${\odim}_{B}(\PPP) = \wbdicr\left( \pr{{\AAA}_{\PPP}}{{\RRR}_{\PPP}} \right) \geq \wbdicr({\GGG}_{0}(E, f)) > {\aleph}_{0}$.
 
 Now assume that ${\odim}_{B}(\PPP) > {\aleph}_{0}$.
 By Corollary \ref{cor:HBS}, $\wbdicr\left( \pr{{\BBB}_{\PPP}}{{\SSS}_{\PPP}} \right) > {\aleph}_{0}$.
 Let $E$ be any selector.
 Using Theorem \ref{thm:dichotomy1}, fix $f \in \FFF$ and a continuous minimal homomorphism $\Pi$ from ${\GGG}_{0}(E, f)$ to $\pr{{\BBB}_{\PPP}}{{\SSS}_{\PPP}}$.
 Define $\varphi: 2 \times \[{T}_{f}\] \rightarrow P$ as follows.
 Given $x \in \[{T}_{f}\]$, $\Pi(x) = \pr{{p}_{x}}{{q}_{x}} \in {\BBB}_{\PPP}$.
 By the definition of ${\BBB}_{\PPP}$, ${p}_{x} \nleq {q}_{x}$ and ${q}_{x} \nleq {p}_{x}$.
 Define $\varphi(\pr{0}{x}) = {q}_{x}$ and $\varphi(\pr{1}{x}) = {p}_{x}$.
 It is easily verified that $\varphi$ is continuous and (2)(b) holds by definition.
 For (2)(a), assume $\pr{0}{x} {\leq}_{{\GGG}_{0}(E, f)} \pr{1}{y}$.
 Then $\pr{x}{y} \in {R}_{E, f}$, whence $\pr{\Pi(x)}{\Pi(y)} \in {\SSS}_{\PPP}$, whence $\pr{\Pi(x)}{\Pi(y)} \in {\RRR}_{\PPP}$, whence ${q}_{x} \leq {p}_{y}$, whence $\varphi(\pr{0}{x}) \leq \varphi(\pr{1}{y})$, as needed.
\end{proof}
\begin{Corollary} \label{cor:dimcov}
 For any Borel quasi order $\PPP$, ${\odim}_{B}(\PPP) \leq {\aleph}_{0}$ or ${\odim}_{B}(\PPP) \geq \cov(\MMM)$.
\end{Corollary}
\begin{proof}
 If ${\odim}_{B}(\PPP) > {\aleph}_{0}$, then (2) of Theorem \ref{thm:borelodimdichotomy} holds.
 If (2) of Theorem \ref{thm:borelodimdichotomy} holds, then the argument in the first paragraph of the proof of Theorem \ref{thm:borelodimdichotomy} shows that ${\odim}_{B}(\PPP) \geq \wbdicr({\GGG}_{0}(E, f))$, for some dense selector $E$ and $f \in \FFF$.
 By Lemma \ref{lem:covM}, $\wbdicr({\GGG}_{0}(E, f)) \geq \cov(\MMM)$, whence ${\odim}_{B}(\PPP) \geq \cov(\MMM)$.
\end{proof}
Our next task is to obtain further structural information about Borel quasi orders of countable Borel order dimension.
We will show that every such order has a Borel linearization.
Kanovei~\cite{MR1607451} proved a dichotomy that characterized the Borel quasi orders that are Borel linearizable.
The basic object of his dichotomy is the following quasi order.
\begin{Definition} \label{def:kanovei}
 Recall that for $x, y \in {2}^{\omega}$, $x {E}_{0} y$ if and only if $\{n \in \omega: x(n) \neq y(n)\}$ is finite.
 On ${2}^{\omega}$, define the quasi ordering ${\leq}_{0}$ by $x \; {\leq}_{0} \; y$ if and only if $x = y$, or $x \neq y$ and $x {E}_{0} y$ and for the maximal $n \in \omega$ such that $x(n) \neq y(n)$, we have $x(n) < y(n)$.
 Clearly, $\pr{{2}^{\omega}}{{\leq}_{0}}$ is a Borel quasi order.
 It is easily verified that ${E}_{{\leq}_{0}}$ is equality and elements in distinct ${E}_{0}$ classes are incomparable.
\end{Definition}
Kanovei proved that a Borel quasi order has a Borel linearization if and only if there is no continuous injection from $\pr{{2}^{\omega}}{{\leq}_{0}}$ that is order preserving and maps ${E}_{0}$-inequivalent elements to incomparable ones (see Theorem \ref{thm:kanovei} below).
\begin{Lemma} \label{lem:kanovei1}
 Suppose $\preceq$ is a Borel quasi order on ${2}^{\omega}$ extending ${\leq}_{0}$.
 Then $\preceq$ is a meager subset of ${2}^{\omega} \times {2}^{\omega}$.
\end{Lemma}
\begin{proof}
 Suppose not.
 There exist $s, t \in {2}^{< \omega}$ such that $\lc s \rc = \lc t \rc$ and $\preceq \cap \left( U \times V \right)$ is comeager in $U \times V$, where $U = \{z \in {2}^{\omega}: \next{s}{1} \subseteq z\}$ and $V = \{z \in {2}^{\omega}: \next{t}{0} \subseteq z\}$.
 Define $\psi: U \times V \rightarrow U \times V$ by $\psi\left( \pr{\ins{s}{1}{x}}{\ins{t}{0}{y}} \right) = \pr{\ins{s}{1}{y}}{\ins{t}{0}{x}}$, for every $x, y \in {2}^{\omega}$.
 Then $\psi$ is an auto-homeomorphism of $U \times V$, and hence $\preceq \cap \left( U \times V \right) \cap {\psi}^{-1}\left( \preceq \cap \left( U \times V \right) \right)$ is non-empty.
 Find $x, y \in {2}^{\omega}$ so that $\ins{s}{1}{x} \preceq \ins{t}{0}{y}$ and $\ins{s}{1}{y} \preceq \ins{t}{0}{x}$.
 Since $\lc s \rc = \lc t \rc$, we have $\ins{s}{1}{y} \preceq \ins{t}{0}{x} \; {\leq}_{0} \; \ins{s}{1}{x} \preceq \ins{t}{0}{y} \; {\leq}_{0} \; \ins{s}{1}{y}$.
 This means $\pr{\ins{s}{1}{y}}{\ins{t}{0}{x}} \in {E}_{\preceq} = {E}_{{\leq}_{0}}$, which is impossible as ${E}_{{\leq}_{0}}$ is equality.
\end{proof}
\begin{Lemma} \label{lem:kanovei2}
 ${\odim}_{B}\left(\pr{{2}^{\omega}}{{\leq}_{0}}\right) > {\aleph}_{0}$.
\end{Lemma}
\begin{proof}
 Suppose $\seq{\preceq}{i}{<}{\omega}$ is a sequence of Borel quasi orders on ${2}^{\omega}$ extending ${\leq}_{0}$.
 Write $M = {\bigcup}_{i < \omega}{{\preceq}_{i}}$ and $N = \{\pr{y}{x}: \pr{x}{y} \in M\}$.
 By Lemma \ref{lem:kanovei1}, $M \cup N$ is a meager subset of ${2}^{\omega} \times {2}^{\omega}$.
 Choose any $\pr{y}{x} \in \left( {2}^{\omega} \times {2}^{\omega} \right) \setminus \left( M \cup N \right)$.
 Then $x \; {\not\leq}_{0} \; y$, and yet $y \; {\not\preceq}_{i} \; x$, for any $i < \omega$.
 Therefore $\seq{\preceq}{i}{<}{\omega}$ cannot witness ${\odim}_{B}\left(\pr{{2}^{\omega}}{{\leq}_{0}}\right) \leq {\aleph}_{0}$.
\end{proof}
\begin{Theorem}[Kanovei~\cite{MR1607451}] \label{thm:kanovei}
 Suppose $\pr{P}{\leq}$ is a Borel quasi order.
 Then exactly one of the following two conditions is satisfied:
 \begin{enumerate}
  \item
  $\pr{P}{\leq}$ is Borel linearizable;
  \item
  there is a continuous 1-1 map $F: {2}^{\omega} \rightarrow P$ satisfying both of the following:
  \begin{enumerate}
   \item[(a)]
   $a \; {\leq}_{0} \; b \implies F(a) \leq F(b)$;
   \item[(b)]
   $a \; {\cancel{E}}_{0} \; b \implies F(a) \ \text{and} \ F(b)$ are $\leq$-incomparable.
  \end{enumerate}
 \end{enumerate}
\end{Theorem}
\begin{Theorem} \label{thm:ctbledimlin}
 Let $\PPP = \pr{P}{\leq}$ be a Borel quasi order.
 If ${\odim}_{B}(\PPP) \leq {\aleph}_{0}$, then $\PPP$ is Borel linearizable.
\end{Theorem}
\begin{proof}
 Assume for a contradiction that $\PPP$ is not Borel linearizable.
 By Theorem \ref{thm:kanovei} fix a continuous 1-1 map $F: {2}^{\omega} \rightarrow P$ satisfying (2)(a) and (2)(b) of Theorem \ref{thm:kanovei}.
 Fix any selector $E$.
 By Lemma \ref{lem:kanovei2} and Theorem \ref{thm:borelodimdichotomy}, find $f \in \FFF$ and a continuous $\varphi: 2 \times \[{T}_{f}\] \rightarrow {2}^{\omega}$ such that (2)(a) and (2)(b) of Theorem \ref{thm:borelodimdichotomy} are satisfied with respect to ${\leq}_{0}$.
 Write $\psi = F \circ \varphi$.
 Then $\psi: 2 \times \[{T}_{f}\] \rightarrow P$ is continuous.
 For $x, y \in \[{T}_{f}\]$, if $\pr{0}{x} \; {\leq}_{{\GGG}_{0}(E, f)} \; \pr{1}{y}$, then $\varphi(\pr{0}{x}) \; {\leq}_{0} \; \varphi(\pr{1}{y})$, whence $\psi(\pr{0}{x}) \leq \psi(\pr{1}{y})$.
 For any $x \in \[{T}_{f}\]$, $\varphi(\pr{0}{x})$ and $\varphi(\pr{1}{x})$ are ${\leq}_{0}$-incomparable.
 Note that if $\varphi(\pr{0}{x}) {E}_{0} \varphi(\pr{1}{x})$, then, by the definition of ${\leq}_{0}$, $\varphi(\pr{0}{x})$ and $\varphi(\pr{1}{x})$ would be ${\leq}_{0}$-comparable.
 Hence $\varphi(\pr{0}{x}) {\cancel{E}}_{0} \varphi(\pr{1}{x})$, whence $\psi(\pr{0}{x})$ and $\psi(\pr{1}{x})$ are $\leq$-incomparable.
 Thus (2) of Theorem \ref{thm:borelodimdichotomy} holds for $\PPP$, which means that ${\odim}_{B}(\PPP) > {\aleph}_{0}$, a contradiction.
\end{proof}
By a celebrated theorem of Harrington, Marker, and Shelah~\cite{borelorder}, Borel linearizability provides a good deal of structural information about a Borel quasi order.
In fact, if $\PPP$ is Borel linearizable, then $\PPP$ can be embedded into ${2}^{\alpha}$ equipped with the lexicographical ordering, where $\alpha$ is some countable ordinal.
In Theorem \ref{thm:hms} and Corollary \ref{cor:hms} below $\lex{}{}$ denotes the lexicographical ordering on ${2}^{\alpha}$.
The result of Harrington, Marker, and Shelah applies to all the so called thin quasi orders.
\begin{Definition}[Harrington, Marker, and Shelah~\cite{borelorder}] \label{def:hms}
 $\PPP$ is \emph{thin} if there is no perfect set of pairwise incomparable elements.
\end{Definition}
\begin{Theorem}[Harrington, Marker, and Shelah~\cite{borelorder}] \label{thm:hms}
 If $\PPP = \pr{P}{\leq}$ is a thin Borel quasi order, then for some $\alpha < {\omega}_{1}$, there is a Borel $f: P \rightarrow {2}^{\alpha}$ such that
 \begin{enumerate}
  \item
  $x \leq y \implies \lex{f(x)}{f(y)}$ and
  \item
  $x \; {E}_{\leq} \; y \iff f(x)=f(y)$, for all $x, y \in P$.
 \end{enumerate}
\end{Theorem}
Theorem \ref{thm:hms} has a number of important consequences for Borel quasi orders and linear orders.
Among other things, it implies that there are no Borel realizations of certain combinatorial counterexamples such as Suslin trees or Suslin lattices.
See~\cite{suslinlattices} for further details.
By combining Theorem \ref{thm:ctbledimlin} with Theorem \ref{thm:hms} we get the following corollary which essentially says that every Borel quasi order of countable Borel order dimension is Borel embeddable into $\pr{{2}^{\alpha}}{\llex}$, for some countable $\alpha$.
\begin{Corollary} \label{cor:hms}
 Suppose $\PPP = \pr{P}{\leq}$ is a Borel quasi order with ${\odim}_{B}(\PPP) \leq {\aleph}_{0}$.
 Then for some $\alpha < {\omega}_{1}$, there is a Borel $f: P \rightarrow {2}^{\alpha}$ such that
 \begin{enumerate}
  \item
   $x \leq y \implies \lex{f(x)}{f(y)}$;
  \item
   $x \; {E}_{\leq} \; y \iff f(x) = f(y)$,
 \end{enumerate}
 for all $x, y \in P$.
\end{Corollary}
We will end this section with the following, somewhat ambitious, problem.
\begin{Question} \label{q:2}
 Is there a classification of all the Borel quasi orders of countable Borel order dimension?
 How about those of finite or bounded Borel order dimension?
\end{Question}
\section{Borel order dimension of locally countable Borel quasi orders} \label{sec:consistency}
In this section, we will use the dichotomies proved in Section \ref{sec:dichotomy} to obtain more information about locally countable Borel quasi orders.
Recall this definition.
\begin{Definition} \label{def:locallyc}
 A quasi order $\PPP = \pr{P}{\leq}$ is said to be \emph{locally countable} (\emph{locally finite}) if for every $x \in P$, $\{ y \in P: y \leq x \}$ is countable (finite).
\end{Definition}
Many important naturally occurring examples of Borel quasi orders are locally countable, for instance, the Turing degrees.
\begin{Definition} \label{def:turing}
 For $x, y \in {2}^{\omega}$, \emph{$x \; {\leq}_{T} \; y$} will mean that $x$ is Turing reducible to $y$.
 Define $\DDD = \pr{{2}^{\omega}}{{\leq}_{T}}$.
\end{Definition}
Observe that if $E$ is a dense selector and $f \in \FFF$, then $\PPP({\GGG}_{0}(E, f))$ is locally countable.
This is because for any $z \in \[{T}_{f}\]$, for each $l \in \omega$, there are at most finitely many $y$ such that $z \restrict l = E(f \restrict l) = y \restrict l$ and $y \; {R}_{E, f} \; z$.
Therefore, $\{y \in \[{T}_{f}\]: y \; {R}_{E, f} \; z\}$ is countable, hence $\{\pr{0}{y} \in {P}_{{\GGG}_{0}(E, f)}: \pr{0}{y} \; {\leq}_{{\GGG}_{0}(E, f)} \; \pr{1}{z}\}$ is countable.
So Theorem \ref{thm:borelodimdichotomy} says that for any Borel quasi order $\PPP$, if ${\odim}_{B}(\PPP) > {\aleph}_{0}$, then this fact is witnessed by an order-preserving continuous map from some locally countable Borel quasi order.

The classical order dimension of $\DDD$ was investigated in detail in \cite{posetdim} and in \cite{MR4228343}.
Kumar and Raghavan~\cite{MR4228343} showed that $\DDD$ has the largest order dimension among all locally countable orders of size at most continuum.
Higuchi, Lempp, Raghavan, and Stephan~\cite{posetdim} proved that if ${2}^{{\aleph}_{0}} = {\kappa}^{+}$, where $\cf(\kappa) > \omega$, then $\odim(\DDD) \leq \kappa$.
Hence we obtain the following corollary to the results of Section \ref{sec:dichotomy} that is worth explicitly stating.
\begin{Corollary}\label{cor:pfa}
 If $\PFA$ holds, then $\odim(\DDD) = {\aleph}_{1} < {\aleph}_{2} = {\odim}_{B}(\DDD) = {2}^{{\aleph}_{0}}$.
\end{Corollary}
\begin{proof}
 Since $\PFA$ implies ${2}^{{\aleph}_{0}} = {\aleph}_{2}$, Proposition 4.3 and Theorem 3.11 of \cite{posetdim} imply that $\odim(\DDD) = {\aleph}_{1}$.
 Since ${\odim}_{B}(\DDD) \geq \odim(\DDD)$, Corollary 3.32 implies that ${\odim}_{B}(\DDD) \geq \cov(\MMM)$.
 Under $\PFA$, $\cov(\MMM) = {\aleph}_{2} = {2}^{{\aleph}_{0}}$, so the result follows.
\end{proof}
The following notions were crucial for the results in Kumar and Raghavan~\cite{MR4228343}.
They were able to use the notion of separation defined below to provide a purely combinatorial characterization of $\odim(\DDD)$.
\begin{Definition} \label{def:separating}
 Let $P$ be a set and $\lambda$ a cardinal.
 A family $\F \subseteq \Pset(P)$ \emph{separates points from elements of $\pc{P}{< \lambda}$} if for every $A \in \pc{P}{< \lambda}$ and $p \in P \setminus A$, there exists $X \in \F$ such that $p \in X$ and $A \cap X = \emptyset$.
 
 A sequence $\seq{<}{i}{\in}{I}$ of partial orders on $P$ \emph{separates points from elements of $\pc{P}{< \lambda}$} if for every $A \in \pc{P}{< \lambda}$ and $p \in P \setminus A$, there exists $i \in I$ so that $\forall q \in A\[q \; {<}_{i} \; p\]$.
\end{Definition}
We will now show that the notion of separation by Borel sets or by Borel partial orders can be used to provide upper bounds on ${\odim}_{B}(\PPP)$, where $\PPP$ is any locally countable Borel quasi order.
\begin{Definition} \label{def:+Xlk}
 Let $X$ be a Polish space and $\lambda, \kappa$ cardinals.
 We say that \emph{$\dagger(X, \lambda, \kappa)$ holds} if there is a sequence $\seq{<}{i}{<}{\kappa}$ of Borel partial orders on $X$ which separates points from elements of $\pc{X}{< \lambda}$.
\end{Definition}
\begin{Lemma} \label{lem:sepext}
 Let $X$ be a Polish space and $\preceq$ a Borel quasi order on $X$ with $\pr{X}{\preceq}$ being locally countable.
 Let $<$ be a Borel partial order on $X$.
 Then there exists a Borel quasi order $\trianglelefteq$ on $X$ that extends $\preceq$ and has the property that for any $x, y \in X$, if $\forall u \preceq x\[u < y\]$, then $x \trianglelefteq y$.
\end{Lemma}
\begin{proof}
 Define $R = \left\{ \pr{y}{x} \in X \times X: \exists u \preceq x \forall v \preceq y \[v < u\]\right\}$.
 Define
 \begin{align*}
  \trianglelefteq \; = \left\{ \pr{x}{y} \in X \times X: \pr{x}{y} \in \mathord{\preceq} \vee \pr{x}{y} \in R \right\}.
 \end{align*}
 It is an easy exercise to verify that $\trianglelefteq$ is a quasi order on $X$ which is an extension of $\preceq$ in the sense of Definition \ref{def:extendquasi}.
 If $x, y \in X$ are such that $\forall u \preceq x\[u < y\]$, then $\pr{x}{y} \in R$ by definition, whence $x \trianglelefteq y$.
 Therefore to complete the proof, it suffices to verify that $\trianglelefteq$ is a Borel subset of $X \times X$, and since it is known that $\preceq$ is a Borel subset of $X \times X$, it suffices for this to check that $R$ is Borel.
 
 To this end, write $Q = \{\pr{y}{x} \in X \times X: x \preceq y\}$.
 $Q$ is Borel and by the hypothesis that $\pr{X}{\preceq}$ is locally countable, ${Q}_{y} = \{x \in X: \pr{y}{x} \in Q\}$ is countable for every $y \in X$.
 By the Luzin--Novikov theorem write $Q = {\bigcup}_{n \in \omega}{{Q}_{n}}$, where each ${Q}_{n}$ is Borel and is the graph of a function.
 Thus for each $n \in \omega$, $\dom({Q}_{n})$ is a Borel subset of $X$.
 We write ${Q}_{n}(x)$ to denote the unique $y \in X$ with $\pr{x}{y} \in {Q}_{n}$, for every $n \in \omega$ and $x \in \dom({Q}_{n})$.
 For $n, m \in \omega$, define
 \begin{align*}
  {B}_{n, m} = \{\pr{x}{y} \in \dom({Q}_{n}) \times \dom({Q}_{m}): \pr{{Q}_{n}(x)}{{Q}_{m}(y)} \in \; <\}.
 \end{align*}
 ${B}_{n, m}$ is Borel.
 For each $n, m \in \omega$, define
 \begin{align*}
  {C}_{n, m} = {B}_{n, m} \cup \left( \left(X \setminus \dom({Q}_{n})\right) \times \dom({Q}_{m})\right),
 \end{align*}
 and note that ${C}_{n, m}$ is also Borel.
 Therefore, ${\bigcup}_{m \in \omega}{\bigcap}_{n \in \omega}{{C}_{n, m}}$ is Borel, and it is easily checked that $R = {\bigcup}_{m \in \omega}{\bigcap}_{n \in \omega}{{C}_{n, m}}$.
\end{proof}
\begin{Lemma} \label{lem:dagger01}
 Suppose $X$ is a Polish space, $\kappa$ is a cardinal, $\lambda \in \{{\aleph}_{0}, {\aleph}_{1}\}$, and that $\dagger(X, \lambda, \kappa)$ holds.
 If $\pr{X}{\preceq}$ is a Borel quasi order with the property that for every $x \in X$, $\lc \{y \in X: y \preceq x\} \rc < \lambda$, then ${\odim}_{B}(\pr{X}{\preceq}) \leq \kappa$.
\end{Lemma}
\begin{proof}
 Note that the hypothesis on $\pr{X}{\preceq}$ implies that it is locally countable, hence Lemma \ref{lem:sepext} is applicable.
 By $\dagger(X, \lambda, \kappa)$, fix a sequence $\seq{<}{i}{<}{\kappa}$ of Borel partial orders on $X$ which separates points from elements of $\pc{X}{< \lambda}$.
 By Lemma \ref{lem:sepext}, for every $i < \kappa$, let ${\trianglelefteq}_{i}$ be a Borel quasi order on $X$ that extends $\preceq$ and has the property that for any $x, y \in X$, if $\forall u \preceq x\[u \; {<}_{i} \; y\]$, then $x \; {\trianglelefteq}_{i} \; y$.
 Consider $x, y \in X$ with $y \npreceq x$.
 Then $A = \{u \in X: u \preceq x\} \in \pc{X}{< \lambda}$ and $y \in X \setminus A$.
 Hence there exists $i < \kappa$ such that $\forall u \in A\[u \; {<}_{i} \; y\]$, whence $x \; {\trianglelefteq}_{i} \; y$.
 Therefore, $\seq{\trianglelefteq}{i}{<}{\kappa}$ is a witness that ${\odim}_{B}(\pr{X}{\preceq}) \leq \kappa$.
\end{proof}
\begin{Lemma} \label{lem:daggerholds}
 For any Polish space $X$, $\dagger(X, {\aleph}_{0}, {\aleph}_{0})$ holds.
\end{Lemma}
\begin{proof}
 Fix a closed $F \subseteq {\omega}^{\omega}$ and a continuous bijection $f: F \rightarrow X$.
 For $s \in {\omega}^{< \omega}$, let ${F}_{s} = F \cap \{r \in {\omega}^{\omega}: s \subseteq r\}$.
 Since $f$ is 1-1 and continuous, ${B}_{s} = f''{F}_{s}$ is a Borel subset of $X$.
 Define ${<}_{s} = \{\pr{x}{y} \in X \times X: x \notin {B}_{s} \wedge y \in {B}_{s}\}$.
 Then ${<}_{s}$ is Borel and it is easily seen to be a partial order on $X$.
 Now suppose $A \subseteq X$ is finite and $y \in X \setminus A$.
 Then ${f}^{-1}(A) \subseteq F$ is finite and if $u \in F$ is unique so that $f(u) = y$, then $u \notin {f}^{-1}(A)$.
 Hence there exists $s \in {\omega}^{< \omega}$ such that $u \in {F}_{s}$, but ${f}^{-1}(A) \cap {F}_{s} = \emptyset$.
 It follows that $y \in {B}_{s}$, but $A \cap {B}_{s} = \emptyset$, whence $\forall x \in A\[x \; {<}_{s} \; y\]$.
 Therefore, $\seq{<}{s}{\in}{{\omega}^{< \omega}}$ witnesses $\dagger(X, {\aleph}_{0}, {\aleph}_{0})$.
\end{proof}
It was pointed out in \cite{posetdim} that a theorem of Kierstead and Milner~\cite{MR1420396} implies that for every locally finite partial order $\pr{P}{<}$ of size at most ${2}^{{2}^{{\aleph}_{0}}}$, $\odim\left( \pr{P}{<} \right) \leq {\aleph}_{0}$.
It turns out that the Borel analog of this true as well.
\begin{Theorem} \label{thm:locallyfinite}
 If $\pr{X}{\preceq}$ is a locally finite Borel quasi order, then ${\odim}_{B}\left( \pr{X}{\preceq} \right) \leq {\aleph}_{0}$.
\end{Theorem}
\begin{proof}
 By combining Lemmas \ref{lem:dagger01} and \ref{lem:daggerholds}.
\end{proof}
\begin{Corollary} \label{cor:locallyfinitelinear}
 Every locally finite Borel quasi order has a Borel linearization.
\end{Corollary}
Next we use ideas from \cite{MR4518086} to show that the Borel order dimension of every locally countable Borel quasi order may consistently be any regular cardinal below the continuum, with the continuum being arbitrarily large.
Thus $\ZFC$ does not prove that there is some locally countable Borel quasi order whose Borel order dimension is equal to ${2}^{{\aleph}_{0}}$.
In this model, we also get that for some fixed $\lambda \leq {2}^{{\aleph}_{0}}$, $\wbdicr({\GGG}_{0}(E, f)) = \lambda$, for every $f \in \FFF$ and dense selector $E$.
Geschke~\cite{MR2779703} has obtained consistency results on the weak Borel chromatic number of Borel graphs.

Harrington introduced a c.c.c.\@ forcing for adding a generic ${G}_{\delta}$ set.
Miller~\cite{MR548475, Mi1} introduced a modification of Harrington's forcing and used it to obtain consistency results about Borel hierarchies.
A new c.c.c.\@ forcing notion, which is an amalgam of a forcing from Kumar and Raghavan~\cite{MR4518086} and the one from Miller~\cite{Mi1}, is introduced below.
Our forcing generically adds a family of ${\aleph}_{1}$ many ${G}_{\delta}$ sets that separates countable sets in the ground model from points.
\begin{Definition} \label{def:alphadelta}
 For ordinals $\alpha, \delta$, define ${\alpha}^{\[< \delta\]} = \{X \subseteq \alpha: \otp(X) < \delta\}$.
\end{Definition}
\begin{Definition} \label{def:fkappadelta}
 Let $\kappa$ be a cardinal and $\delta < {\omega}_{1}$ be an indecomposable ordinal.
 Define
 \begin{align*}
  &{\G}^{\ast}_{\kappa, \delta} = \left\{g: g \ \text{is a function} \ \wedge \dom(g) \in {\kappa}^{\[< \delta\]} \wedge \ran(g) \subseteq \pc{\omega}{< {\aleph}_{0}}\right\},\\
  &{\G}_{\kappa, \delta} = \left\{g \in {\G}^{\ast}_{\kappa, \delta}: \forall n \in \omega\[\lc\left\{\alpha \in \dom(g): n \in g(\alpha)\right\}\rc < {\aleph}_{0}\]\right\}.
 \end{align*}
 For ${g}_{1}, {g}_{2} \in {\G}_{\kappa, \delta}$, define ${g}_{2} \; {\leq}_{\G} \; {g}_{1}$ if and only if $\dom({g}_{2}) \supseteq \dom({g}_{1})$ and $\forall \alpha \in \dom({g}_{1})\[{g}_{2}(\alpha) \supseteq {g}_{1}(\alpha)\]$.
 It is clear that ${\leq}_{\G}$ is a quasi order.
\end{Definition}
\begin{Definition} \label{def:F}
 Define
 \begin{align*}
  \F = \left\{f: f \ \text{is a function} \ \wedge \dom(f) \in \pc{{\omega}^{< \omega}}{< {\aleph}_{0}} \wedge \ran(f) \subseteq \pc{\omega}{< {\aleph}_{0}}\right\}.
 \end{align*}
 For ${f}_{1}, {f}_{2} \in \F$ define ${f}_{2} \; {\leq}_{\F} \; {f}_{1}$ if and only if $\dom({f}_{2}) \supseteq \dom({f}_{1})$ and $\forall s \in \dom({f}_{1})\[{f}_{2}(s) \supseteq {f}_{1}(s)\]$.
 Clearly, ${\leq}_{\F}$ is a quasi order.
\end{Definition}
\begin{Definition} \label{def:RS}
 Let $E: {2}^{{\aleph}_{0}} \rightarrow {\omega}^{\omega}$ be a bijection.
 Let $\delta < {\omega}_{1}$ be indecomposable.
 Define ${\RR}_{E, \delta}$ to be the collection of all $\pr{f}{g, h} \in \F \times {\G}_{{2}^{{\aleph}_{0}}, \delta} \times \pc{{2}^{{\aleph}_{0}}}{< {\aleph}_{0}}$ satisfying the following conditions:
 \begin{enumerate}
  \item
  $h \cap \dom(g) = \emptyset$;
  \item
  $\forall \alpha \in \dom(g) \forall l \in \omega\[E(\alpha) \restrict l \in \dom(f) \implies g(\alpha) \cap f(E(\alpha) \restrict l) = \emptyset\]$.
 \end{enumerate}
 Given $p = \pr{{f}_{p}}{{g}_{p}, {h}_{p}} \in {\RR}_{E, \delta}$ and $q = \pr{{f}_{q}}{{g}_{q}, {h}_{q}} \in {\RR}_{E, \delta}$, define $q \; {\leq}_{\RR} \; p$ if and only if ${f}_{q} \; {\leq}_{\F} \; {f}_{p}$, ${g}_{q} \; {\leq}_{\G} \; {g}_{p}$, and ${h}_{q} \supseteq {h}_{p}$.
 Then it is clear that $\pr{{\RR}_{E, \delta}}{{\leq}_{\RR}}$ is a forcing notion, and when $p = \pr{f}{g, h} \in {\RR}_{E, \delta}$, we will write ${f}_{p}, {g}_{p}$, and ${h}_{p}$ to denote $f, g$, and $h$ respectively.
 
 Define ${\fSS}_{E}$ to be the finite support product of the ${\RR}_{E, \delta}$ as $\delta$ ranges over the countable indecomposable ordinals.
 In other words, ${\fSS}_{E}$ is the collection of all $p$ such that $p$ is a function, $\dom(p)$ is a finite subset of $\{\delta < {\omega}_{1}: \delta \ \text{is indecomposable}\}$, and $\forall \delta \in \dom(p)\[p(\delta) \in {\RR}_{E, \delta}\]$.
 And for $p, q \in {\fSS}_{E}$, $q \; {\leq}_{\fSS} \; p$ if and only if $\dom(q) \supseteq \dom(p)$ and $\forall \delta \in \dom(p)\[q(\delta) \; {\leq}_{\RR} \; p(\delta)\]$.
 Then $\pr{{\fSS}_{E}}{{\leq}_{\fSS}}$ is a forcing notion.
\end{Definition}
\begin{Lemma} \label{lem:density}
 Let $E: {2}^{{\aleph}_{0}} \rightarrow {\omega}^{\omega}$ be a bijection.
 Let $\delta < {\omega}_{1}$ be indecomposable.
 Then the following hold:
 \begin{enumerate}
  \item
  for each $n \in \omega$, ${D}_{n} = \{q \in {\RR}_{E, \delta}: \exists s \in \dom({f}_{q})\[n \in f(s)\]\}$ is dense;
  \item
  for each $p \in {\RR}_{E, \delta}$, $\beta \in {h}_{p}$, and $n \in \omega$,
  \begin{align*}
   {D}_{p, \beta, n} = \left\{r \; {\leq}_{\RR} \; p: \exists l \in \omega\[E(\beta) \restrict l \in \dom({f}_{r}) \wedge n \in {f}_{r}(E(\beta) \restrict l)\]\right\}
  \end{align*}
  is dense below $p$;
  \item
  for each $p \in {\RR}_{E, \delta}$ and $\alpha \in \dom({g}_{p})$, ${D}_{p, \alpha} = \{r \; {\leq}_{\RR} \; p: {g}_{r}(\alpha) \neq \emptyset\}$ is dense below $p$.
 \end{enumerate}
\end{Lemma}
\begin{proof}
 For (1): fix $p \in {\RR}_{E, \delta}$.
 Then $X = \{E(\alpha): \alpha \in \dom({g}_{p}) \wedge n \in {g}_{p}(\alpha)\} \subseteq {\omega}^{\omega}$ is finite.
 Hence it is possible to find $s \in {\omega}^{< \omega} \setminus \dom({f}_{p})$ such that $s \not\subseteq x$, for all $x \in X$.
 Define ${f}_{q} = {f}_{p} \cup \{\pr{s}{\{n\}}\}$, ${g}_{q} = {g}_{p}$, and ${h}_{q} = {h}_{p}$.
 To see that (2) of Definition \ref{def:RS} holds, suppose $\alpha \in \dom({g}_{q})$, $l \in \omega$, and that $E(\alpha) \restrict l \in \dom({f}_{q})$.
 If $E(\alpha) \restrict l \in \dom({f}_{p})$, then ${g}_{q}(\alpha) \cap {f}_{q}(E(\alpha)\restrict l) = {g}_{p}(\alpha) \cap {f}_{p}(E(\alpha)\restrict l) = \emptyset$.
 Otherwise, $E(\alpha) \restrict l = s$, which means $E(\alpha) \notin X$, and so $n \notin {g}_{p}(\alpha)$.
 Hence ${g}_{q}(\alpha) \cap {f}_{q}(E(\alpha) \restrict l) = {g}_{p}(\alpha) \cap \{n\} = \emptyset$.
 Therefore, $q \; {\leq}_{\RR} \; p$ and $q \in {D}_{n}$, as required.
 
 For (2): fix any $q \; {\leq}_{\RR} \; p$.
 Once again, $X = \{E(\alpha): \alpha \in \dom({g}_{q}) \wedge n \in {g}_{q}(\alpha)\}$ is a finite set.
 Further, since $\beta \in {h}_{p} \subseteq {h}_{q}$ and ${h}_{q} \cap \dom({g}_{q}) = \emptyset$, $E(\beta) \notin X$.
 Therefore, there exists $l \in \omega$ such that $E(\beta) \restrict l \in {\omega}^{< \omega} \setminus \dom({f}_{q})$ and $E(\beta) \restrict l \not\subseteq x$, for all $x \in X$.
 Define ${f}_{r} = {f}_{q} \cup \left\{\pr{E(\beta) \restrict l}{\{n\}}\right\}$, ${g}_{r} = {g}_{q}$, and ${h}_{r} = {h}_{q}$.
 As before, suppose $\alpha \in \dom({g}_{r})$, $m \in \omega$, and that $E(\alpha) \restrict m \in \dom({f}_{r})$.
 If $E(\alpha) \restrict m \in \dom({f}_{q})$, then ${g}_{r}(\alpha) \cap {f}_{r}(E(\alpha) \restrict m) = {g}_{q}(\alpha) \cap {f}_{q}(E(\alpha) \restrict m) = \emptyset$, while if not, then $E(\beta) \restrict l \subseteq E(\alpha)$, whence $n \notin {g}_{q}(\alpha)$ and ${g}_{r}(\alpha) \cap {f}_{r}(E(\alpha) \restrict m) = {g}_{q}(\alpha) \cap \{n\} = \emptyset$.
 Thus $r \; {\leq}_{\RR} \; q$ and $r \in {D}_{p, \beta, n}$, as needed.
 
 For (3): fix any $q \; {\leq}_{\RR} \; p$.
 Then $\bigcup \ran({f}_{q})$ is a finite subset of $\omega$.
 Choose $n \in \omega \setminus \left( \bigcup \ran({f}_{q}) \right)$.
 Define ${g}_{r}$ to be the function such that $\dom({g}_{r}) = \dom({g}_{q})$, ${g}_{r}(\alpha) = {g}_{q}(\alpha) \cup \{n\}$, and ${g}_{r}(\beta) = {g}_{q}(\beta)$, for all $\beta \in \dom({g}_{q}) \setminus \{\alpha\}$.
 Note ${g}_{r} \in {\G}_{{2}^{{\aleph}_{0}}, \delta}$ and ${g}_{r} \; {\leq}_{\G} \; {g}_{q}$.
 Define ${f}_{r} = {f}_{q}$ and ${h}_{r} = {h}_{q}$.
 Note ${h}_{r} \cap \dom({g}_{r}) = {h}_{q} \cap \dom({g}_{q}) = \emptyset$.
 Suppose $\beta \in \dom({g}_{r})$, $l \in \omega$, and that $E(\beta) \restrict l \in \dom({f}_{r})$.
 Then ${g}_{r}(\beta) \cap {f}_{r}(E(\beta) \restrict l) \subseteq \left({g}_{q}(\beta) \cup \{n\}\right) \cap {f}_{q}(E(\beta) \restrict l) = {g}_{q}(\beta) \cap {f}_{q}(E(\beta) \restrict l) = \emptyset$ because $n \notin {f}_{q}(E(\beta) \restrict l)$.
 Thus $r \; {\leq}_{\RR} \; q$ and $r \in {D}_{p, \alpha}$, as needed.
\end{proof}
\begin{Definition} \label{def:gdelta}
 Suppose $\V$ is a transitive model of a sufficiently large fragment of $\ZFC$.
 In $\V$, suppose that $E: {2}^{{\aleph}_{0}} \rightarrow {\omega}^{\omega}$ is a bijection and that $\delta < {\omega}_{1}$ is indecomposable.
 Suppose $G$ is a $(\V, {\RR}^{\V}_{E, \delta})$-generic filter.
 In $\VG$, define the following sets.
 For each $n \in \omega$, ${U}_{G, n} = \bigcup\left\{\[s\]: \exists p \in G\[s \in \dom({f}_{p}) \wedge n \in {f}_{p}(s)\]\right\}$, where $\[s\] = \{x \in {\omega}^{\omega}: s \subseteq x\}$, for every $s \in {\omega}^{< \omega}$.
 Define ${\F}_{G} = {\bigcap}_{n \in \omega}{{U}_{G, n}}$.
 Define ${\G}_{G} = \{E(\alpha): \exists p \in G\[\alpha \in \dom({g}_{p})\]\}$ and ${\sH}_{G} = \{E(\alpha): \exists p \in G\[\alpha \in {h}_{p}\]\}$.
\end{Definition}
\begin{Lemma} \label{lem:gdelta}
 ${\F}_{G}$ is a ${G}_{\delta}$ set such that ${\sH}_{G} \subseteq {\F}_{G}$ and ${\G}_{G} \cap {\F}_{G} = \emptyset$.
\end{Lemma}
\begin{proof}
 It is clear from the definition that ${\F}_{G}$ is a ${G}_{\delta}$ set.
 Suppose $p \in G$, $\alpha \in {h}_{p}$ and $n \in \omega$.
 By (2) of Lemma \ref{lem:density}, ${D}_{p, \alpha, n}$ is dense below $p$, and so, there exist $q \in G$ and $l \in \omega$ with $E(\alpha) \restrict l \in \dom({f}_{q})$ and $n \in {f}_{q}(E(\alpha) \restrict l)$, whence $E(\alpha) \in {U}_{G, n}$.
 As this is for every $n \in \omega$, $E(\alpha) \in {\F}_{G}$.
 This shows ${\sH}_{G} \subseteq {\F}_{G}$.
 Next, suppose for a contradiction that $p \in G$, $\alpha \in \dom({g}_{p})$, and that $E(\alpha) \in {\F}_{G}$.
 By (3) of Lemma \ref{lem:density}, there exist $q \in G$ and $n \in \omega$ with $q \; {\leq}_{\RR} \; p$ and $n \in {g}_{q}(\alpha)$.
 Since we have assumed $E(\alpha) \in {\F}_{G}$, there exist $r \in G$ and $l \in \omega$ so that $E(\alpha) \restrict l \in \dom({f}_{r})$ and $n \in {f}_{r}(E(\alpha) \restrict l)$.
 Find ${r}^{\ast} \in G$ with ${r}^{\ast} \; {\leq}_{\RR} \; q, r$.
 Then we have $\alpha \in \dom({g}_{{r}^{\ast}})$, $l \in \omega$, $E(\alpha) \restrict l \in \dom({f}_{{r}^{\ast}})$, and $n \in {g}_{{r}^{\ast}}(\alpha) \cap {f}_{{r}^{\ast}}(E(\alpha) \restrict l) = \emptyset$ because ${r}^{\ast}$ satisfies (2) of Definition \ref{def:RS}.
 This contradiction shows ${\G}_{G} \cap {\F}_{G} = \emptyset$.
\end{proof}
\begin{Definition} \label{def:PQ}
 Suppose $\delta < {\omega}_{1}$ is indecomposable and that ${\aleph}_{0} \leq \kappa$ is a cardinal.
 Define ${\QQ}_{\kappa, \delta}$ to be the collection of all $p$ such that $p$ is a function, $\dom(p) \in {\kappa}^{\[< \delta\]}$, $\ran(p) \subseteq 2$, and $\{\xi \in \dom(p): p(\xi) = 1\}$ is finite.
 For $p, q \in {\QQ}_{\kappa, \delta}$, define $q \; {\leq}_{\QQ} \; p$ if and only if $q \supseteq p$.
 
 Define ${\PP}_{\kappa}$ to be the finite support product of the ${\QQ}_{\kappa, \delta}$ as $\delta$ ranges over the indecomposable countable ordinals.
 In other words, ${\PP}_{\kappa}$ is the collection of all $p$ such that $p$ is a function, $\dom(p)$ is a finite subset of $\{\delta < {\omega}_{1}: \delta \ \text{is indecomposable}\}$, and $\forall \delta \in \dom(p)\[p(\delta) \in {\QQ}_{\kappa, \delta}\]$.
 For $p, q \in {\PP}_{\kappa}$, $q \; {\leq}_{\PP} \; p$ if and only if $\dom(q) \supseteq \dom(p)$ and $\forall \delta \in \dom(p)\[q(\delta) \; {\leq}_{\QQ} \; p(\delta)\]$.
 
 For an indecomposable $\delta < {\omega}_{1}$ and a bijection $E: {2}^{{\aleph}_{0}} \rightarrow {\omega}^{\omega}$ define ${j}_{\delta}: {\RR}_{E, \delta} \rightarrow {\QQ}_{{2}^{{\aleph}_{0}}, \delta}$ as follows.
 Given $p \in {\RR}_{E, \delta}$, ${j}_{\delta}(p) = q$, where $q$ is the function such that $\dom(q) = \dom({g}_{p}) \cup {h}_{p}$, ${q}^{-1}(\{1\}) = {h}_{p}$, and ${q}^{-1}(\{0\}) = \dom({g}_{p})$.
 Define $j: {\fSS}_{E} \rightarrow {\PP}_{{2}^{{\aleph}_{0}}}$ as follows.
 Given $p \in {\fSS}_{E}$, $j(p) = q$, where $q$ is the function such that $\dom(q) = \dom(p)$ and $\forall \delta \in \dom(q)\[q(\delta) = {j}_{\delta}(p(\delta))\]$.
\end{Definition}
The forcings ${\QQ}_{\kappa, \delta}$ and ${\PP}_{\kappa}$ were first defined by Kumar and Raghavan~\cite{MR4518086}.
They proved that ${\PP}_{\kappa}$ is c.c.c.\@ for all infinite $\kappa$.
We will now use this fact to prove that ${\fSS}_{E}$ is c.c.c.
\begin{Lemma} \label{lem:compatible}
 Suppose $E: {2}^{{\aleph}_{0}} \rightarrow {\omega}^{\omega}$ is a bijection and $\delta < {\omega}_{1}$ is indecomposable.
 For any $p, q \in {\RR}_{E, \delta}$, if ${f}_{p} = {f}_{q}$ and ${j}_{\delta}(p) \; {\not\perp}_{\QQ} \; {j}_{\delta}(q)$, then $p \; {\not\perp}_{\RR} \; q$.
\end{Lemma}
\begin{proof}
 Define ${g}_{r}$ to be the function such that $\dom({g}_{r}) = \dom({g}_{p}) \cup \dom({g}_{q})$, $\forall \alpha \in \dom({g}_{p}) \setminus \dom({g}_{q})\[{g}_{r}(\alpha) = {g}_{p}(\alpha)\]$, $\forall \alpha \in \dom({g}_{q}) \setminus \dom({g}_{p})\[{g}_{r}(\alpha) = {g}_{q}(\alpha)\]$, and $\forall \alpha \in \dom({g}_{p}) \cap \dom({g}_{q})\[{g}_{r}(\alpha) = {g}_{p}(\alpha) \cup {g}_{q}(\alpha)\]$.
 Note that ${g}_{r} \in {\G}_{{2}^{{\aleph}_{0}}, \delta}$ and that ${g}_{r} \; {\leq}_{\G} \; {g}_{p}, {g}_{q}$.
 Define ${f}_{r} = {f}_{p} = {f}_{q}$ and ${h}_{r} = {h}_{p} \cup {h}_{q}$.
 Since ${j}_{\delta}(p) \; {\not\perp}_{\QQ} \; {j}_{\delta}(q)$, it follows that $\dom({g}_{p}) \cap {h}_{q} = \emptyset = \dom({g}_{q}) \cap {h}_{p}$.
 Therefore, $\dom({g}_{r}) \cap {h}_{r} = \emptyset$.
 Next, fix $\alpha \in \dom({g}_{r})$, $l \in \omega$, and assume that $E(\alpha) \restrict l \in \dom({f}_{r}) = \dom({f}_{p}) = \dom({f}_{q})$.
 Then if $\alpha \in \dom({g}_{p})$, then ${g}_{p}(\alpha) \cap {f}_{p}(E(\alpha) \restrict l) = \emptyset$ and if $\alpha \in \dom({g}_{q})$, then ${g}_{q}(\alpha) \cap {f}_{q}(E(\alpha) \restrict l) = \emptyset$.
 It follows that ${g}_{r}(\alpha) \cap {f}_{r}(E(\alpha) \restrict l) = \emptyset$.
 Thus we have verified that $r \in {\RR}_{E, \delta}$.
 Since $r \; {\leq}_{\RR} \; p, q$, $p \; {\not\perp}_{\RR} \; q$, as claimed.
\end{proof}
\begin{Lemma} \label{lem:seccc}
 ${\fSS}_{E}$ is c.c.c.
\end{Lemma}
\begin{proof}
 For a contradiction, assume that $\seq{p}{\gamma}{<}{{\omega}_{1}}$ is an antichain in ${\fSS}_{E}$.
 Write $I = \{\delta < {\omega}_{1}: \delta \ \text{is indecomposable}\}$ and ${D}_{\gamma} = \dom({p}_{\gamma}) \in \pc{I}{< {\aleph}_{0}}$.
 Also, for each $\gamma < {\omega}_{1}$ and each $\delta \in {D}_{\gamma}$, write ${p}_{\gamma}(\delta) = \pr{{f}_{\gamma, \delta}}{{g}_{\gamma, \delta}, {h}_{\gamma, \delta}}$ instead of $\pr{{f}_{{p}_{\gamma}(\delta)}}{{g}_{{p}_{\gamma}(\delta)}, {h}_{{p}_{\gamma}(\delta)}}$ for ease of notation.
 Find $A \in \pc{{\omega}_{1}}{{\aleph}_{1}}$ and $D \in \pc{I}{< {\aleph}_{0}}$ such that $\seq{D}{\gamma}{\in}{A}$ is a $\Delta$-system with root $D$.
 Since ${\F}^{D}$ is a countable set, find $\seq{f}{\delta}{\in}{D} \in {\F}^{D}$ and $B \in \pc{A}{{\aleph}_{1}}$ such that $\forall \gamma \in B \forall \delta \in D\[{f}_{\gamma, \delta} = {f}_{\delta}\]$.
 For each $\gamma \in B$, define ${q}_{\gamma} = j({p}_{\gamma}) \in {\PP}_{{2}^{{\aleph}_{0}}}$ and note that $\dom({q}_{\gamma}) = \dom({p}_{\gamma}) = {D}_{\gamma}$ and that $\forall \delta \in D\[{q}_{\gamma}(\delta) = {j}_{\delta}({p}_{\gamma}(\delta))\]$.
 Now consider any $\gamma, {\gamma}^{\ast} \in B$ with $\gamma \neq {\gamma}^{\ast}$.
 By hypothesis, ${p}_{\gamma} \; {\perp}_{\fSS} \; {p}_{{\gamma}^{\ast}}$.
 So there exists $\delta \in D$ where ${p}_{\gamma}(\delta) \; {\perp}_{\RR} \; {p}_{{\gamma}^{\ast}}(\delta)$.
 Since ${f}_{\gamma, \delta} = {f}_{\delta} = {f}_{{\gamma}^{\ast}, \delta}$, Lemma \ref{lem:compatible} implies that ${j}_{\delta}({p}_{\gamma}(\delta)) \; {\perp}_{\QQ} \; {j}_{\delta}({p}_{{\gamma}^{\ast}}(\delta))$, in other words that ${q}_{\gamma}(\delta) \; {\perp}_{\QQ} \; {q}_{{\gamma}^{\ast}}(\delta)$.
 Therefore, ${q}_{\gamma} \; {\perp}_{\PP} \; {q}_{{\gamma}^{\ast}}$.
 Thus we have shown that $\seq{q}{\gamma}{\in}{B}$ is an uncountable antichain in ${\PP}_{{2}^{{\aleph}_{0}}}$.
 However, it is proved in \cite{MR4518086}[Lemma 4.2.\@ (1), pp.\@ 8-9] that ${\PP}_{\kappa}$ is c.c.c.\@ for all $\kappa \geq \omega$.
 This contradiction concludes the proof.
\end{proof}
\begin{Theorem} \label{thm:gdeltaseparating}
 Let $\V$ be a transitive model of a sufficiently large fragment of $\ZFC$.
 In $\V$, suppose that $E: {2}^{{\aleph}_{0}} \rightarrow {\omega}^{\omega}$ is a bijection.
 Let $G$ be a $(\V, {\fSS}^{\V}_{E})$-generic filter.
 Then in $\VG$ there exists a family $\DDD$ of at most ${\aleph}_{1}$ many ${G}_{\delta}$ subsets of ${\omega}^{\omega}$ such that for any $A \in \V \cap \pc{{\omega}^{\omega}}{< {\aleph}_{1}}$ and any $x \in \left( \V \cap {\omega}^{\omega} \right) \setminus A$, there exists $D \in \DDD$ with $x \in D$ and $A \cap D = \emptyset$.
\end{Theorem}
\begin{proof}
 Note that since ${\fSS}_{E}$ is c.c.c.\@ in $\V$, all cardinals are preserved.
 For $p \in {\fSS}_{E}$ and $\delta \in \dom(p)$, write $p(\delta) = \pr{{f}_{p, \delta}}{{g}_{p, \delta}, {h}_{p, \delta}}$ instead of $\pr{{f}_{p(\delta)}}{{g}_{p(\delta)}, {h}_{p(\delta)}}$.
 For each indecomposable $\delta < {\omega}_{1}$, define $G(\delta) = \{p(\delta): p \in G\}$.
 Then $G(\delta)$ is a $(\V, {\RR}^{\V}_{E, \delta})$-generic filter, and let ${\F}^{\V\[G(\delta)\]}_{G(\delta)}$, ${\G}_{G(\delta)}$, and ${\sH}_{G(\delta)}$ be as in definition \ref{def:gdelta}.
 By Lemma \ref{lem:gdelta}, ${\F}^{\V\[G(\delta)\]}_{G(\delta)}$ is a ${G}_{\delta}$ subset of ${\left( {\omega}^{\omega} \right)}^{\V\[G(\delta)\]}$ such that ${\sH}_{G(\delta)} \subseteq {\F}^{\V\[G(\delta)\]}_{G(\delta)}$ and ${\G}_{G(\delta)} \cap {\F}^{\V\[G(\delta)\]}_{G(\delta)} = \emptyset$.
 Define ${D}_{\delta} = {\F}^{\V\[G\]}_{G(\delta)}$.
 Then ${D}_{\delta}$ is a ${G}_{\delta}$ subset of ${\left( {\omega}^{\omega} \right)}^{\V\[G\]}$ and we claim that the family $\DDD = \{{D}_{\delta}: \delta < {\omega}_{1} \wedge \delta \ \text{is indecomposable}\}$ has the required properties.
 Working in $\V$, fix $A \in \pc{{\omega}^{\omega}}{< {\aleph}_{1}}$ and $x \in {\omega}^{\omega} \setminus A$.
 Let $B = {E}^{-1}(A) \in \pc{{2}^{{\aleph}_{0}}}{< {\aleph}_{1}}$ and let $\alpha \in {2}^{{\aleph}_{0}}$ be unique such that $E(\alpha) = x$.
 Note that $\alpha \notin B$.
 Define
 \begin{align*}
  {D}_{A, x} = \left\{ p \in {\fSS}_{E}: \exists \delta \in \dom(p)\[\alpha \in {h}_{p, \delta} \wedge B \subseteq \dom({g}_{p, \delta})\] \right\}.
 \end{align*}
 We will verify that ${D}_{A, x}$ is dense in ${\fSS}_{E}$.
 To see this consider ${p}^{\ast} \in {\fSS}_{E}$ and find $\delta < {\omega}_{1}$ such that $\delta$ is indecomposable, $\delta \notin \dom({p}^{\ast})$, and $\otp(B) < \delta$.
 Define ${f}_{p, \delta} = \emptyset$.
 Define ${g}_{p, \delta}$ to be the function such that $\dom({g}_{p, \delta}) = B$, and $\forall \beta \in B\[{g}_{p, \delta}(\beta) = \emptyset\]$.
 Define ${h}_{p, \delta} = \{\alpha\}$.
 Then $\pr{{f}_{p, \delta}}{{g}_{p, \delta}, {h}_{p, \delta}} \in {\RR}_{E, \delta}$.
 Define $p = {p}^{\ast} \cup \left\{ \pr{\delta}{\pr{{f}_{p, \delta}}{{g}_{p, \delta}, {h}_{p, \delta}}} \right\}$.
 Then $p \in {D}_{A, x}$ and $p \; {\leq}_{\fSS} \; {p}^{\ast}$, showing that ${D}_{A, x}$ is dense.
 Since $G$ is $(\V, {\fSS}_{E})$-generic, there exists $p \in G \cap {D}_{A, x}$.
 Let $\delta \in \dom(p)$ be so that $\alpha \in {h}_{p, \delta}$ and $B \subseteq \dom({g}_{p, \delta})$.
 Then by definition $x \in {\sH}_{G(\delta)}$ and $A \subseteq {\G}_{G(\delta)}$, whence $x \in {\F}^{\V\[G(\delta)\]}_{G(\delta)}$ and $A \cap {\F}^{\V\[G(\delta)\]}_{G(\delta)} = \emptyset$.
 Since this is absolute, $x \in {D}_{\delta}$ and $A \cap {D}_{\delta} = \emptyset$, as required.
\end{proof}
\begin{Theorem} \label{thm:lambdakappa}
 Suppose ${\V}_{0}$ is any transitive model of a sufficiently large fragment of $\ZFC$.
 In ${\V}_{0}$, suppose that ${\aleph}_{1} \leq \lambda \leq \kappa$ are cardinals so that $\lambda$ is regular and ${\kappa}^{{\aleph}_{0}} = \kappa$.
 Then there is a c.c.c.\@ forcing extension in which all of the following hold:
 \begin{enumerate}
  \item
  ${2}^{{\aleph}_{0}} = \kappa$;
  \item
  for every locally countable Borel quasi order $\PPP$, either ${\odim}_{B}(\PPP) \leq {\aleph}_{0}$ or ${\odim}_{B}(\PPP) = \lambda$;
  \item
  for every dense selector $E$ and every $f \in \FFF$, $\wbdicr({\GGG}_{0}(E, f)) = \lambda$.
 \end{enumerate}
\end{Theorem}
\begin{proof}
 First, produce a c.c.c.\@ forcing extension $\V \supseteq {\V}_{0}$ so that ${2}^{{\aleph}_{0}} = \kappa$ holds in $\V$ (e.g.\@ force with $\Fn(\kappa \times \omega, 2)$ over ${\V}_{0}$, see \cite{Kunen}[Lemma 5.14]).
 Now work in $\V$ and define a finite support iteration $\left\langle {\PP}_{\alpha}; {\mathring{\QQ}}_{\alpha}: \alpha \leq \lambda \right\rangle$ of c.c.c.\@ forcings such that for each $\alpha < \lambda$, ${\mathring{\QQ}}_{\alpha}$ is a full ${\PP}_{\alpha}$-name such that
 \begin{align*}
  {\forces}_{{\PP}_{\alpha}} \; {\exists E\[E: {2}^{{\aleph}_{0}} \rightarrow {\omega}^{\omega} \ \text{is a bijection and} \ {\mathring{\QQ}}_{\alpha} = {\fSS}_{E}\]}.
 \end{align*}
 By Lemma \ref{lem:seccc}, ${\forces}_{{\PP}_{\alpha}}{{\mathring{\QQ}}_{\alpha} \ \text{is c.c.c.\@}}$, for every $\alpha < \lambda$, so ${\PP}_{\lambda}$ is c.c.c.\@ and all cofinalities and cardinals are preserved.
 Further, for each $\alpha < \lambda$, ${\forces}_{{\PP}_{\alpha}}{\lc {\mathring{\QQ}}_{\alpha} \rc = {2}^{{\aleph}_{0}}}$, so by standard arguments, ${\forces}_{{\PP}_{\lambda}}{{2}^{{\aleph}_{0}} = \kappa}$.
 Therefore, (1) holds.
 
 Suppose $G$ is a $(\V, {\PP}_{\lambda})$-generic filter.
 For $\alpha < \lambda$, define $G(\alpha) = \{p \restrict \alpha: p \in G\}$.
 By Theorem \ref{thm:gdeltaseparating}, there is a family ${\DDD}_{\alpha}$ of at most ${\aleph}_{1}$ many ${G}_{\delta}$ subsets of ${\left( {\omega}^{\omega} \right)}^{\V\[G(\alpha+1)\]}$ with the property that for every $A \in \V\[G(\alpha)\] \cap \pc{{\omega}^{\omega}}{< {\aleph}_{1}}$ and $x \in \left(\V\[G(\alpha)\] \cap {\omega}^{\omega}\right) \setminus A$, there exists $D \in {\DDD}_{\alpha}$ such that $x \in {D}^{\V\[G(\alpha+1)\]}$ and $A \cap {D}^{\V\[G(\alpha+1)\]} = \emptyset$.
 Define $\DDD = \{{D}^{\VG}: \exists \alpha < \lambda\[D \in {\DDD}_{\alpha}\]\}$.
 Then $\DDD$ is a family of at most $\lambda$ many ${G}_{\delta}$ subsets of ${\left({\omega}^{\omega}\right)}^{\VG}$.
 If $A \in \VG \cap \pc{{\omega}^{\omega}}{< {\aleph}_{1}}$ and $x \in \left( \VG \cap {\omega}^{\omega} \right) \setminus A$, then since ${\PP}_{\lambda}$ is c.c.c.\@ and $\cf(\lambda) > \omega$, there exists $\alpha < \lambda$ such that $A, x \in \V\[G(\alpha)\]$, whence for some $D \in {\DDD}_{\alpha}$, $x \in {D}^{\V\[G(\alpha+1)\]}$ and $A \cap {D}^{\V\[G(\alpha+1)\]} = \emptyset$.
 Since this is absolute, $x \in {D}^{\VG}$ and $A \cap {D}^{\VG} = \emptyset$.
 
 Now, working in $\VG$, let $X$ be an arbitrary Polish space.
 Fix a closed $F \subseteq {\omega}^{\omega}$ and a continuous bijection $\varphi: F \rightarrow X$.
 For each $L \in \DDD$, ${B}_{L} = \varphi''\left(F \cap L\right)$ is a Borel subset of $X$ because $\varphi$ is 1-1 and continuous.
 Define
 \begin{align*}
  {<}_{L} = \left\{ \pr{x}{y} \in X \times X: x \notin {B}_{L} \wedge y \in {B}_{L} \right\}.
 \end{align*}
 Then ${<}_{L}$ is a Borel partial order on $X$.
 If $A \in \pc{X}{< {\aleph}_{1}}$ and $u \in X \setminus A$, then ${\varphi}^{-1}(A) \subseteq F$ and $\lc {\varphi}^{-1}(A) \rc < {\aleph}_{1}$.
 If $x \in F$ is unique with $\varphi(x) = u$, then $x \in {\omega}^{\omega} \setminus {\varphi}^{-1}(A)$, and so, for some $L \in \DDD$, $x \in L$, but $L \cap {\varphi}^{-1}(A) = \emptyset$.
 It follows that $u \in {B}_{L}$ and that $A \cap {B}_{L} = \emptyset$, whence $\forall v \in A\[v \; {<}_{L} \; u\]$.
 Thus $\left\{ {<}_{L}: L \in \DDD \right\}$ witnesses $\dagger(X, {\aleph}_{1}, \lambda)$ in $\VG$.
 
 Now suppose in $\VG$ that $\PPP$ is a locally countable Borel quasi order with ${\odim}_{B}(\PPP) > {\aleph}_{0}$.
 By Lemma \ref{lem:dagger01}, ${\odim}_{B}(\PPP) \leq \lambda$.
 On the other hand, by Corollary \ref{cor:dimcov}, ${\odim}_{B}(\PPP) \geq \cov(\MMM)$.
 Since ${\PP}_{\lambda}$ is a finite support iteration of non-trivial forcing notions, a well-known argument (see \cite{Kunen}[Chapter VIII, Exercise J3, pp.\@ 299]) shows that Cohen reals are added at each stage $\alpha < \lambda$ with $\cf(\alpha) = \omega$.
 Since $\lambda$ is regular and ${\PP}_{\lambda}$ is c.c.c.\@, it follows that $\cov(\MMM) \geq \lambda$ holds in $\VG$.
 Therefore, ${\odim}_{B}(\PPP) = \lambda$, as required for (2).
 
 Finally, suppose in $\VG$ that $E$ is a dense selector and that $f \in \FFF$.
 It is clear that $\PPP({\GGG}_{0}(E, f))$ is a locally countable Borel quasi order.
 By (2) and by Lemma \ref{lem:XinAPG}, $\wbdicr({\GGG}_{0}(E, f)) \leq {\odim}_{B}(\PPP({\GGG}_{0}(E, f))) \leq \lambda$.
 On the other hand, by Lemma \ref{lem:covM}, $\lambda \leq \cov(\MMM) \leq \wbdicr({\GGG}_{0}(E, f))$.
 Therefore, $\wbdicr({\GGG}_{0}(E, f)) = \lambda$.
 This concludes the proof of the theorem.
\end{proof}
It was mentioned earlier that Kumar and Raghavan~\cite{MR4228343} showed that $\DDD$ has the largest order dimension among the locally countable quasi orders of size continuum.
We do not know if the Borel analog of this result is provable in $\ZFC$.
\begin{Question} \label{q:boreld}
 Is it provable in $\ZFC$ that for every locally countable Borel quasi order $\PPP$, ${\odim}_{B}(\PPP) \leq {\odim}_{B}(\DDD)$?
\end{Question}
It is worth pointing out that Lutz~\cite{MR4313010} has proved that not all Borel quasi orders are Borel embeddable in $\DDD$.
\def\polhk#1{\setbox0=\hbox{#1}{\ooalign{\hidewidth
  \lower1.5ex\hbox{`}\hidewidth\crcr\unhbox0}}}
\providecommand{\bysame}{\leavevmode\hbox to3em{\hrulefill}\thinspace}
\providecommand{\MR}{\relax\ifhmode\unskip\space\fi MR }
\providecommand{\MRhref}[2]{%
  \href{http://www.ams.org/mathscinet-getitem?mr=#1}{#2}
}
\providecommand{\href}[2]{#2}

\end{document}